\documentclass[12pt]{article}
\usepackage{color}
\usepackage{epsfig,psfrag,amsmath,amssymb,latexsym}
\usepackage{amscd}
\usepackage{amsfonts}
\usepackage{graphicx}
\newtheorem{theorem}{Theorem}[section]

\newtheorem{corollary}{Corollary}[section]

\newtheorem{definition}{Definition}[section]

\newtheorem{lemma}{Lemma}[section]

\newtheorem{proposition}{Proposition}[section]
\newtheorem{remark}{Remark}[section]

\newcommand{\NN}{\mathbb{N}}
\newcommand{\RR}{\mathbb{R}}

\newcommand{\PP}{\mathbb{P}}

\newcommand{\EE}{\mathbb{E}}
\newcommand{\II}{\mathbb{I}}

\numberwithin{equation}{section}

\begin{document}

\title{\textbf{ULTRAMETRIC AND TREE POTENTIAL}}

\author{Claude DELLACHERIE
\thanks{Laboratoire Rapha\"el Salem,
UMR 6085, Universit\'e de Rouen, Site Colbert, 76821 Mont Saint
Aignan Cedex, France; Email: Claude.Dellacherie@univ-rouen.fr.
This author thanks support from Nucleus Millennium P04-069-F for
his visit to CMM-DIM at Santiago},
Servet MARTINEZ \thanks{CMM-DIM; Universidad
de Chile; Casilla 170-3 Correo 3 Santiago; Chile; Email:
smartine@dim.uchile.cl. The author's research is supported by
Nucleus Millennium Information and Randomness P04-069-F.},
Jaime SAN MARTIN \thanks{CMM-DIM; Universidad
de Chile; Casilla 170-3 Correo 3
Santiago; Chile; Email: jsanmart@dim.uchile.cl. The author's research
is supported by Nucleus Millennium Information and Randomness
P04-069-F.}
}

\maketitle

\begin{abstract}
We study infinite tree and ultrametric matrices, and their action
on the boundary of the tree. For each tree matrix we show the
existence of a symmetric random walk associated to it and we study
its Green potential. We provide a representation theorem for
harmonic functions that includes simple expressions for any
increasing harmonic function  and the Martin kernel. In the
boundary, we construct the Markov kernel whose Green function is
the extension of the matrix and we simulate it by using a cascade
of killing independent exponential random variables and
conditionally independent uniform variables. For ultrametric
matrices we supply probabilistic conditions to study its potential
properties when immersed in its minimal tree matrix extension.

\end{abstract}

\section {Introduction and Basic Notation}

\subsection{ Introduction}

Here we study ultrametric and tree matrices, the random walk they
induce on trees and its potential theory. There exists a broad
literature in this field (a complete state-of-the-art study can be
found in \cite{lyons}). The main difference between our work and most
part of this literature, is that our starting point is not a
random walk on a tree, but a tree matrix or more general, an
ultrametric matrix. In this viewpoint, the random walk is
constructed from  the matrix, a nontrivial fact, even in the
finite case. Hence, most of the concepts must be expressed with
respect to the matrix, that turns to have two representations. One
in the tree as the sum of  a potential  and a harmonic basis. The
other one in the boundary of the tree as the potential of a Markov
process. Our results are not a simple translation of well-known
results from walks on trees to matrices. New phenomenon appear:
the formula for monotone harmonic functions; the predictable
representation property of tree matrices, that is the keystone for
a wide class of relations including the Martin kernel at $\infty$;
the formula relating different levels of the process in the
boundary which allow us to simulate it, in a constructive way.
Below we give the framework of our work and summarize some of the
main results.

\medskip

An ultrametric matrix $U=(U_{ij}: i,j\in I)$ is a symmetric
nonnegative matrix verifying the ultrametric inequality $U_{ij}\ge
\min\{U_{ik},U_{kj}\}$ for all $i,j,k\in I$. When $I$ is finite it
was shown in \cite{martinez1994}, \cite{dell1996}, that the
inverse $U^{-1}$ of a nonsingular ultrametric matrix $U$ is a
diagonal dominant Stieltjes matrix (see \cite{nabben1994} for a
linear algebra proof of this fact). Then, $U$ is proportional to
the Green potential of a subMarkov kernel $P$, that is $U=\alpha\,
\sum_{n\ge 0}P^n$. Thus, if we consider for $i\neq j$ the distance
$d(i,j)=1/U_{ij}$, then $d$ is an ultrametric distance and $1/d$
is a Green potential (a phenomenon that happens in $\RR^3$ with
the Newtonian potential and the Euclidian distance, or in $\RR^d,
d\ge 3$, when we allow an increasing function of the Euclidian
distance).

\medskip

Tree matrices are a special case of ultrametric matrices. They are
defined by a rooted tree $(I,{\cal T})$ (with root $r$) and a
strictly increasing function $w:\{|k|: k\in I\}\to \RR_+$, where
$|k|$, the level of $k$, is the length of the geodesic from a site
$k$ to $r$. Then, the tree matrix $U$ is defined as
$U_{ij}=w_{|i\land j|}$, with $i\land j$ been the farthest vertex
from $r$ that is common to the geodesic from $i$ and $j$ to $r$.
When $I$ is finite, $U$ is the potential of a Markov process,
whose skeleton is a simple symmetric random walk on the tree, only
defective at the root. Me also mention here that every other
ultrametric matrix is obtained by restriction of this class (see
\cite{dell1996}). That is for every ultrametric matrix $U$ there
exists a minimal extension tree matrix $\widetilde U$, defined on
$(\widetilde I, \widetilde {\cal T})$, such that $U=\widetilde
U|_I$. This minimal tree $\widetilde {\cal T}$ has all the
information that is required to understand the one step
transitions of the Markov process associated to $U$. In fact,
$P_{ij}>0$ if and only if the geodesic in $\widetilde {\cal T}$
joining $i$ and $j$ does not contain other points in $I$.

\medskip

One of the purposes of this paper is to extend this study to
countably infinite ultrametric and tree matrices. Each ultrametric
matrix $U$ defines a natural kernel $W$ in the boundary
$\partial_\infty$ of the tree. This class of operators were
already considered in \cite{lyons1} and \cite{lyons2}, were a deep
study of potential properties is done, mainly in connection to
dimension and capacity on the boundary.

We show $W$ is a stochastic integral operator whose associated
filtration ${\cal F}=({\cal F}_k)$ is given  by the tree
structure, see Proposition \ref{prostoch}. The operator $W$ allows
to represent harmonic functions in the infinite tree (see
Corollary \ref{c201}). This representation is an alternative to
the well known Martin kernel representation, supplied for example
in the basic reference \cite{cartier} and in \cite{sawyer1997}. We
describe the set of increasing (along the branches) harmonic
functions  as those functions that can be written in terms of $U$,
see Theorem \ref{p302}. Also, we characterize the set of bounded
harmonic functions which are the difference of two harmonic
increasing functions (see Theorem \ref{representation}).

\medskip

In the finite setting, a tree  matrix $U$ is the potential of a
continuous time Markov chain,  the leaves of the tree being
reflecting states (see Proposition \ref{p2}). Nevertheless, in the
infinite transient case, each column of $U$ is the sum of a
potential and a nontrivial harmonic function, as follows from
relation (\ref{re21}). This last result uses two main ingredients.
The first one comes from the finite case analysis: when imposing
Dirichlet boundary conditions at the boundary, a finite tree
matrix is the sum of the potential matrix and a matrix whose
columns generate the harmonic functions (see Proposition
\ref{p3}). The second element is the exit measure $\mu$ at the
boundary.

\medskip

 We mainly consider the potential of tree matrices for Markov
semigroups defective at the root, because this is natural in the
finite case. But, in the transient infinite case  we can reflect
the process at the root as we do in section \ref{secc7} and by a
limit procedure we can represent the Martin kernel in a similar
way as for the absorbed chain, see Theorem \ref{th701}. Also
explicit computations for homogeneous trees are done, retrieving
some known formulae (\cite{sawyer1991}, \cite{sawyer1997}).

\medskip

In section \ref{ul5} we study ultrametric matrices $U=(U_{ij}:
i,j\in I)$. Under some explicit hypotheses, we associate to $U$ a
minimal tree matrix $\widetilde U=(\widetilde U_{\tilde {\imath}
\tilde {\jmath}}: \tilde {\imath}, \tilde {\jmath} \in \widetilde
I)$ extending it, with a natural immersion of the sites $I$ into
$\widetilde I$. In Theorem \ref{p100} we show that a canonical
generator $Q$ can be associated to $U$ with the help of the
generator $\widetilde Q$ associated to $\widetilde U$; and in
Theorem \ref{t1} it is shown that the harmonic functions defined
by $Q$ can be retrieved from the harmonic functions defined by
$\widetilde Q$. The key hypothesis is that a random walk starting
from $\widetilde I \setminus I$ is trapped at the cemetery or it
reaches $I$ with probability one.

\medskip

Let us turn to the process in the boundary of the tree. The fact
that $W$ is an stochastic integral operator reveals to be the main
property which allow us to study the generator $-W^{-1}$. In
Theorem \ref{nucleo} we describe the transition probability kernel
of the subMarkov semigroup $(e^{-tW^{-1}})$ acting on the boundary
and having $U$ as the kernel potential. In Theorem \ref{esc10} we
supply a recursive formula satisfied by the process in terms of:
the killing time, the process killed at a successor of the root,
and the process starting afresh from the distribution  $\mu$. This
allows us  to give a constructive simulation of the process in
terms of exponential random variables (killing times) and
independent random variables distributed $\mu$ conditioned to the
atoms of the natural filtration.

\medskip

There is a large literature on stochastic processes on the
$p$-adic field. See for example the works of \cite{albeverio1},
\cite{albeverio2}, \cite{albeverio3}, \cite{albeverio4}, \cite{Khr},
\cite{kochubei2} (see also the references therein). We notice that
in these works there is a natural measure in the boundary, the
Haar measure for the $p-$adic tree, or an absolutely continuous
probability measure with respect to the Haar measure for the $p-$adic field.
In our work the tree needs to be
locally finite, but no other hypothesis is needed as homogeneity.
Even with this generality, the exit measure at $\infty$ fulfills
the requirements allowing us to describe the process at the
boundary. We also mention here, among others, the works of
\cite{kochubei1} and \cite{evans1989} in local fields and the work
of \cite{aldous1999} in disconnected spaces.

\medskip

Ultrametricity is an important tool in applied areas: taxonomy
(see \cite{benzecri1973}); the problem of maximal flow on finite
graphs, namely the Theorem of Gomory-Hu (see \cite{berge1962});
statistical physics (see \cite{derr1981}) to explore the
ultrametric Parisi solution to spin-glass models (see
\cite{mezard1987}, \cite{Rammal} and references therein).

\medskip

One of the tools we use in this work is the notion of stochastic
integral operator (s.i.o.), which is the natural framework in
which ultrametricity appears in stochastic analysis. An operator
$Y$ acting on a space $L^2$ is an s.i.o. (see \cite{dell1977}) if
for some filtration ${\cal F}=({\cal F}_t)$, $Y$ can be written as
$$
Yf=\int_0^\infty H_t\; d\EE(f | {\cal F}_t) \hbox{ where } H=(H_t)
\hbox{ is a } {\cal F}-\hbox{predictable process}.
$$
The fact that $H$ is predictable will play a fundamental role in
the analysis of $W$. The characterization of s.i.o. on countable
spaces leaded to study the relations between ultrametric matrices
and filtrations (see \cite{dart1988}). On the other hand the
continuous version of ultrametric matrices needs to consider
operators of the form $V=\int_0^{\infty} \EE(\; |{\cal F}_t)
dG_t$, where $(G_t)$ is a bounded increasing and adapted process.
In \cite{dell1998} it is shown that these operators are Markov
potential kernels (a proof of it that uses backward stochastic
differential equations can be found in \cite{yu1998}). This result
is in the spirit and constitutes a generalization of the one
obtained in \cite{bouleau1989}.

\medskip

\subsection {Trees}

\medskip

Here we fix notation and recall some well-known notions on trees.
Let $(I,{\cal T})$ be a connected non-oriented and locally finite
tree. $I$ is the set of sites and ${\cal T}\subset I\times I$ is
the set of links. Two sites $i,j$ are neighbors if $(i,j)\in {\cal
T}$. The set of sites with a unique neighbor is called the
extremal set and is denoted by $\cal E$. The geodesic joining $i$
and $j$ is denoted by $\it{ geod}(i,j)$ and its length is written
$|i-j|$. In particular $g(i,i)$ only contains $i$ and its length
is $0$. We assume the tree is rooted by $r\in I$ and we write
$|i|=|i-r|$. We introduce the following order relation on $I$:
\begin{equation}
\label{orden}
i\preceq j \hbox{ if } i\in {\it geod}(r,j).
\end{equation}
The element $i\land j = \max ({\it geod }(r,i)\cap {\it geod }(r,j))$
denotes the $\preceq-$ minimum between $i$ and $j$. For $i\in
I\setminus \{ r\}$ there is a unique element $i^-$ verifying
$(i^-,i)\in {\cal T}$ and $i^-\preceq i$, called the
predecessor of $i$. It verifies $|i^-| = |i| - 1$. The
set of successors of $i\in I$ is denoted by
$S_i = \{ j\in I: j^-=i\}$, it is a finite set that
could be empty.
By $i^+$ we mean a generic element of $S_i$ and
${\cal L} = \{ i\in I: S_i = \phi\}$ is
the set of leaves of the tree. We
notice that ${\cal L}\subseteq {\cal E}$, and $r$ is the only point that
could be extremal without being a leaf. The branch of the tree
born at $i\in I$, is denoted by
$[i, \infty)=\{j\in I: i\preceq j\}$.

\medskip

Assume that $I$ is countably infinite.
An infinite path $(i_n\in I: n\in \NN)$ in the tree
with origin $i_0$, is such that $(i_n, i_{n+1})\in {\cal T}$ for
every $n\in \NN=\{0,1,2,\cdots\}$. If all the $i_n$ are different this path
is called an infinite chain.
The following relation
$$
(i_n: n\in \NN)\sim (j_n: n\in \NN) \Leftrightarrow |\{i_n: n\in \NN\}
\cap \{j_n: n\in \NN\}|=\infty ,
$$
is an equivalence relation in the set of chains.
The quotient set is the
boundary of the tree $(I,{\cal T})$ (see \cite{cartier})
and we denote it by $\partial_{\infty}$

\medskip

For every $i\in I$ and $\xi \in \partial_{\infty}$
there exists a unique chain of origin $i$ which is in the
equivalence class $\xi$, and it is called the {\it geodesic}
between $i$ and $\xi$, and denoted by ${\it geod}(i,\xi)$.
For a fixed $\xi\in \partial_{\infty}$ and
$n\in \NN$, we denote by $\xi(n)$ the unique point in the geodesic
${\it geod}(r,\xi)$
such that $|\xi(n)|=n$, thus $\xi(0)=r$.
Let $(i_n: n\ge 0)$ be an infinite path, the following criterion
stated in \cite{cartier}, is useful
to establish convergence to a point in the boundary,
\begin{equation}
\label{re3}
\left(\forall j\in I: \; |\{n\in \NN: i_n=j\}|<\infty \right)
\Rightarrow \exists ! \; \xi =
\lim\limits_{n\to \infty} i_n\in \partial_{\infty}.
\end{equation}
In this case there exists a subsequence $(k_n\!: \!n\!\ge \!0)$
verifying
${\it geod}(i_0,\xi) = (i_{k_n}\!:\! n\!\ge \!0)$.

\medskip

For $i\in I$,
$\xi \in \partial_{\infty}$ we put $i\preceq \xi$ if $i\in {\it
geod}(r,\xi)$. Hence we can extend $\land$ to $I\cup \partial_{\infty}$ by
\begin{equation}
\label{re1}
\xi \land \eta=\max\left({\it geod}(r,\xi)\cap {\it geod}(r,\eta)\right).
\end{equation}
Hence, $\xi \land \xi=\xi$ and
if $\xi\neq \eta$ then
$\xi \land \eta \in I$. In this last case  $\xi\land \eta=i$ if and only if
$\xi(|i|)= \eta(|i|)$ and $\xi(n)\neq \eta(n)$ for $n>|i|$.

\medskip

The extended subtree hanging from $i\in I$ is $[i, \infty]=\{z\in
I\cup \partial_{\infty}: i\preceq z\}$. The set $I\cup
\partial_{\infty}$ is endowed with the topology ${\bf T}$
generated by the basis of open sets ${\cal A}=\{[i, \infty]: i\in
I\} \cup \{\{i\}: i\in I\}$. The sets in ${\cal A}$ are open and
closed in ${\bf T}$. The topological space $(I\cup
\partial_{\infty}, {\bf T})$ is compact, totally discontinuous and
metrically generated, the trace topology on $I$ is the discrete
one and $I$ is an open dense subset in $I\cup \partial_{\infty}$.
Also ${\cal A}$ is a semi-algebra generating the Borel
$\sigma-$algebra $\sigma( {\bf T} )$. We use the following
notation
\begin{equation}
\label{re2}
\partial_\infty(i)=
[i, \infty]\cap \partial_{\infty}= \{\eta \in \partial_{\infty}:
i\preceq \eta\}.
\end{equation}
The class of sets ${\cal C}=\{\partial_\infty(i): i\in I\}$
is a basis of open
(and closed) sets generating ${\bf T}\cap \partial_{\infty}$ and
it is also a semi-algebra generating the trace of $\sigma({\bf T})$ on
$\partial_{\infty}$.
Therefore for $\xi \in \partial_{\infty}:\;
\xi=\lim\limits_{n\to \infty} \xi(n) \hbox{ and }
\partial_\infty(\xi(n))=\{\eta \in \partial_\infty:\;
|\xi\wedge \eta|\ge n\}$.

\medskip

It will be useful to add an state $\partial_r\not\in I$ and the
oriented link $(r,\partial_r)$. We put
$r^-=\partial_r$ and $|\partial_r|=-1$.

\medskip

In the sequel we adopt the following notation. For any nonempty subset
$J\subseteq I$ we denote by
$\II_{J\times J}$ the identity $J\times J$ matrix. If $M$ is an
$I\times I$ matrix and $J, K\subseteq I$ are nonempty, the matrix
$M_{JK}=(M_{jk}: j\in J, k\in K)$ (also denoted by $M_{J,K}$) is the
restriction of $M$ to $J\times K$.
By ${\bf 1}_A$ we mean the characteristic function of
a set $A$, and ${\bf 1}$ is the constant function taking the value
$1$ in its domain of definition.

\medskip

\section{Tree Matrices}

In \cite{dell1996} we have introduced the notion of tree matrices
in the finite case. Here we give a general version of it. Let $(I,
{\cal T})$ be a tree with root $r$. Put ${\bf N} = \{|i|: i\in I
\}$, which is equal to $\NN$ when the tree is infinite.

\medskip

\begin{definition}
\label{d1} A tree matrix $U=(U_{ij}:i,j\in I)$ is defined by an
strictly positive and strictly increasing function $w:{\bf
N}\to(0,\infty)$ as follows,
$$
U_{ij} = w_{|i\land j|} \hbox{ for } i,j\in I.
$$
\end{definition}

The matrix $U$ is  strictly positive and symmetric, and it verifies
$U_{ij} = U_{i\land j, i\land j}$. In particular $U_{i^-i} =
U_{i^-i^-}= w_{|i|-1}$ when $i^-\in I$. Notice that $U_{i^+i^+}=
w_{|i|+1}$ does not depend on the particular element $i^+\in S_i$.
We extend $U$ to $I\cup \{\partial_r\}$ by putting
$U_{i\partial_r} = U_{\partial_r i}=w_{-1}=0$ for
every $i\in I\cup \{\partial_r\}$.

\medskip

By using (\ref{re1}) we can extend $U$ to $I\cup \partial_{\infty}$
in the following way
\begin{equation}
\label{ext357}
\hbox{ for }\xi,\eta\in \partial_{\infty},\;\; U_{\xi \eta}=
w_{|\xi\land \eta|}  \hbox{ if } \xi\neq \eta \hbox{ and }
U_{\xi \xi}=\lim\limits_{n\to \infty} U_{\xi(n) \xi(n)}.
\end{equation}
This extension is continuous in both variables:
$U_{\xi \eta}=\lim\limits_{n\to \infty, m\to \infty}
U_{\xi(n) \eta(m)}$ for $\xi,\eta\in \partial_{\infty}$.

\medskip

We associate to $U$ a symmetric matrix
$Q = (Q_{ij}: i,j\in I)$
supported by the tree and the diagonal, that is
$Q_{ij}=0$ if $i\neq j$ and $(i,j) \not\in {\cal T}$. This matrix
$Q$ is given by
\begin{equation}
\label{re200}
\begin{array}{ll}
& Q_{ii^-} = Q_{i^-i}=
\big(w_{|i|}-w_{|i|-1}\big)^{-1} \hbox{ for }i^-,i\in I; \cr
&  \cr
& Q_{ii} = -\big( (w_{|i|}-w_{|i|-1})^{-1}+
|S_i|(w_{|i|+1}-w_{|i|})^{-1} \big)
 \hbox{ for }i\in I.
\end{array}
\end{equation}
Observe that $Q_{ii^+}=Q_{i^+i}=\big(w_{|i|+1}-w_{|i|}\big)^{-1}$
does not depend on $i^+\in S_i$.
When $i\in {\cal L}$ is a leaf, then
$Q_{ii} = -Q_{ii^-}$. The matrix $Q$ verifies
$Q_{ij}\ge 0$ if $j\neq i$ and $\sum_{j\in I}Q_{ij}\leq 0$ for
$i\in I$. Then $Q$ is a $q$-matrix,
it is conservative in the sites  $i\in I\setminus \{r\}$, that is
$\sum\limits_{j\in I}Q_{ij}=0$, and defective at $r$ since
$\sum\limits_{j\in I}Q_{rj}=-w^{-1}_0$.
We call $\widehat Q$ the extension of $Q$ to
$I\cup \{\partial_r\}$, given by
\begin{equation}
\label{re81}
{\widehat Q}_{r\partial_r}=w_0^{-1}\; \hbox{ and } \;
{\widehat Q}_{i\partial_r}= 0 \hbox{ for }
i\neq r, i\in I\cup \{\partial_r\}.
\end{equation}
This extension is a nonsymmetric conservative
$q-$matrix in $I\cup \{\partial_r\}$, having $\partial_r$ as an absorbing
state.

\bigskip

\centerline{\resizebox{10cm}{!}{\begin{picture}(0,0)%
\includegraphics{treematrix.pstex}%
\end{picture}%
\setlength{\unitlength}{4144sp}%
\begingroup\makeatletter\ifx\SetFigFont\undefined%
\gdef\SetFigFont#1#2#3#4#5{%
  \reset@font\fontsize{#1}{#2pt}%
  \fontfamily{#3}\fontseries{#4}\fontshape{#5}%
  \selectfont}%
\fi\endgroup%
\begin{picture}(8685,7008)(901,-6853)
\put(901,-61){\makebox(0,0)[lb]{\smash{\SetFigFont{17}{20.4}{\rmdefault}{\mddefault}{\updefault}{\color[rgb]{0,0,0}$w_{-1} = 0 \cdots$}%
}}}
\put(901,-961){\makebox(0,0)[lb]{\smash{\SetFigFont{17}{20.4}{\rmdefault}{\mddefault}{\updefault}{\color[rgb]{0,0,0}$w_0\cdots$}%
}}}
\put(901,-1880){\makebox(0,0)[lb]{\smash{\SetFigFont{17}{20.4}{\rmdefault}{\mddefault}{\updefault}{\color[rgb]{0,0,0}$w_1\cdots$}%
}}}
\put(901,-2761){\makebox(0,0)[lb]{\smash{\SetFigFont{17}{20.4}{\rmdefault}{\mddefault}{\updefault}{\color[rgb]{0,0,0}$w_2\cdots$}%
}}}
\put(901,-3661){\makebox(0,0)[lb]{\smash{\SetFigFont{17}{20.4}{\rmdefault}{\mddefault}{\updefault}{\color[rgb]{0,0,0}$w_3\cdots$}%
}}}
\put(4550,-1890){\makebox(0,0)[lb]{\smash{\SetFigFont{17}{20.4}{\rmdefault}{\mddefault}{\updefault}{\color[rgb]{0,0,0}$i\land j$}%
}}}
\put(6481,-961){\makebox(0,0)[lb]{\smash{\SetFigFont{17}{20.4}{\rmdefault}{\mddefault}{\updefault}{\color[rgb]{0,0,0}$r$}%
}}}
\put(5221,-4156){\makebox(0,0)[lb]{\smash{\SetFigFont{17}{20.4}{\rmdefault}{\mddefault}{\updefault}{\color[rgb]{0,0,0}$i$}%
}}}
\put(5986,-4126){\makebox(0,0)[lb]{\smash{\SetFigFont{17}{20.4}{\rmdefault}{\mddefault}{\updefault}{\color[rgb]{0,0,0}$j$}%
}}}
\put(3840,-6586){\makebox(0,0)[lb]{\smash{\SetFigFont{17}{20.4}{\rmdefault}{\mddefault}{\updefault}{\color[rgb]{0,0,0}$\xi$}%
}}}
\put(6481,-61){\makebox(0,0)[lb]{\smash{\SetFigFont{17}{20.4}{\rmdefault}{\mddefault}{\updefault}{\color[rgb]{0,0,0}$\partial_r$}%
}}}
\put(9586,-5776){\makebox(0,0)[lb]{\smash{\SetFigFont{17}{20.4}{\rmdefault}{\mddefault}{\updefault}{\color[rgb]{0,0,0}$\partial_\infty$}%
}}}
\put(4100,-5000){\makebox(0,0)[lb]{\smash{\SetFigFont{17}{20.4}{\rmdefault}{\mddefault}{\updefault}{\color[rgb]{0,0,0}$\xi\land\eta$}%
}}}
\put(5380,-6781){\makebox(0,0)[lb]{\smash{\SetFigFont{17}{20.4}{\rmdefault}{\mddefault}{\updefault}{\color[rgb]{0,0,0}$\eta$}%
}}}
\end{picture}
}}

\medskip

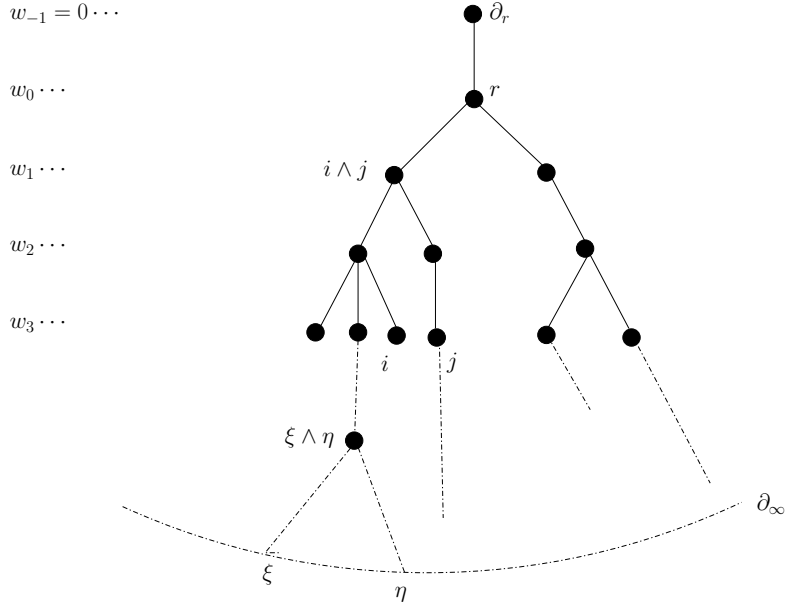
\centerline{Figure 1: Tree Matrix}

\bigskip

Observe that if $M$ is an $I\times I$ matrix then the formal products of
matrices $QM$ and $MQ$ are well defined because each line and column of
$Q$ has finite support.

\medskip

\begin{proposition}
\label{p1} The $q-$matrix $Q$ verifies $(-Q)U=U(-Q)=\II$.
\end{proposition}

\begin{proof}
From symmetry it suffices
to show $(-Q)U=\II$. For $i,k\in I$ we have
$$
(QU)_{ik}=Q_{ii^-}U_{i^-k}+Q_{ii}U_{ik}+Q_{ii^+}\sum_{j\in S_i}U_{jk} .
$$

If $k\land i \preceq i^-$ we have $i\neq r$ and
$k\land i = k\land i^- = k\land i^+$. Then $(QU)_{ik}=0$ because $Q$
is conservative at $i\in I$.

\medskip

For $k=i$ we have

\begin{eqnarray*}
(QU)_{ii} & \!=\! &Q_{ii^-}U_{i^-i}+Q_{ii}U_{ii}+|S_i| Q_{ii^+}U_{ii}\\
& \!=\! &
Q_{ii^-}U_{i^-i}-Q_{ii^-}U_{ii}-|S_i|Q_{ii^+}U_{ii}+|S_i|Q_{ii^+}U_{ii}
\!=\!-Q_{ii^-}(U_{ii}-U_{i^-i})=-1.
\end{eqnarray*}

The last case left to analyze is when $k\land i^+= i^+$
for some and a unique $i^+\in S_i$.
Then
$k\land i^-= i^-, k\land i = i = k\land j$
for $j\in S_i\setminus \{i^+\}$. Hence

\begin{eqnarray*}
(QU)_{ik} & =&Q_{ii^-}U_{i^-i^-}+Q_{ii}U_{ii}+(|S_i|-1)Q_{ii^+}U_{ii}
+Q_{ii^+}U_{i^+i^+}\\
& =& (QU)_{ii}+Q_{ii^+}(U_{i^+i^+}-U_{ii})=-1+1.
\end{eqnarray*}
$\Box$
\end{proof}

\medskip

\begin{remark}
\label{ly1} As we shall see $Q$ is a generator of a Markov process
with state space $I\cup \{\partial_r\}$. Its discrete skeleton has
transition probabilities given by
$$
p_{ij}={Q_{ij}\over \sum\limits_{k\in S_i \cup \{i^-\}} Q_{ik}}\;
\hbox{ for } j\in S_i \cup \{i^-\}.
$$
In the electrical circuits interpretation, this corresponds to a
chain whose conductances are given by $C_{ii^-}:=Q_{ii-}$ (see
\cite{KSK} section 9, and \cite{lyons2} section 2).
\end{remark}

Let us study more closely the case when $(I, {\cal T})$ is a finite
tree rooted at $r$. Since the state space is finite, the matrix
$Q=-U^{-1}$ is an infinitesimal
generator defective only at $r$. Let $(X_t:{0\le t<\zeta})$ be the
associated Markov process taking values on $I$, with lifetime
$\zeta$. We denote
$({\widehat X}_t:{0\le t<\infty})$ the Markov chain associated
to the extension ${\widehat Q}$. We notice that
$\partial_r$ is an absorbing state for ${\widehat X}$. Let
$T_{\partial_r}=\inf\{t\ge 0: {\widehat X}_t= \partial_r\}$.
Then, when starting from an state in $I$, the chains
$({\widehat X}_t: 0\le t<T_{\partial_r})$
and $(X_t:{0\le t<\zeta})$, have the same distribution. Hence
$\zeta=T_{\partial_r}$. Therefore, if necessary we can assume that
$X$ is defined in $I\cup \{\partial_r\}$.
%
%

\begin{proposition}
\label{p2}
Let $(I, {\cal T})$ be a finite tree rooted at $r$.
Then $U$ is the potential matrix of the chain
$(X_t:{0\le t<\zeta})$, that is
$U= \int^{\infty}_0 e^{tQ} dt$ or equivalently
$U_{ij}=\EE_i \Big(\int^\infty_0 {\bf 1}_{\{X_t=j\}}dt \Big)$.
\end{proposition}

\begin{proof}
Since $(e^{tQ})$ is the semigroup of $(X_t:{0\le t<\zeta})$ we get
$U=-Q^{-1}=\int^{\infty}_0 e^{tQ} dt$. $\Box$
\end{proof}

\medskip

We set $n+1=|{\bf N}|=\max\{ |i|: i \in I\}$. Consider the sets
\begin{equation*}
B^{n+1}=\{i\in I: |i|=n+1\} \hbox{ and }
{\widetilde B}^n=\{i\in I: |i|=n, S_i\neq \phi\}.
\end{equation*}
Hence $B^{n+1}=\cup_{i\in {\widetilde B}^n}S_i$. To avoid the
trivial situation we assume $n\ge 1$. We will also set $I^m=\{i\in
I: |i|\le m\}$, so $I=I^{n+1}$.
\medskip

We denote by $T_i= \inf \{t\ge 0: X_t=i \}$ the hitting time of
$i\in I$, and by $T_{{\widetilde B}^n}:=
\inf \{T_i: i\in {\widetilde B}^n\}$
and $T_{B^{n+1}}:= \inf \{T_i: i\in B^{n+1}\}$
the hitting times of ${\widetilde B}^n$ and $B^{n+1}$, respectively, .

\medskip

Let $Q_{I^n I^n}$ be the restriction of $Q$ to $I^n \times I^n$.
The chain $(X_t: t < T_{\partial_r}\land T_{B^{n+1}})$
killed at $B^{n+1}\cup \{\partial_r\}$ has
generator $Q_{I^n I^n}$ and
semigroup $(e^{t{Q_{I^n I^n}}})$. Its potential
$V^{(n)} :=-(Q_{I^n I^n})^{-1}$  verifies
$$
V^{(n)}_{ij} = \EE_i \Big(\int_0^{T_{\partial_r}\land T_{B^{n+1}}}
{\bf 1}_{\{X_t=j \}} dt \Big)
\hbox{ for } \; i,j\in I^n.
$$

Further consider the $q$-matrix ${\bar Q}^{(n)}$ defined in $I_n$ by
$$
{\bar Q}^{(n)}_{I^{n}\setminus {\widetilde B}^n, I^{n}}=
Q_{I^{n}\setminus {\widetilde B}^n, I^n}
\hbox{ and }
{\bar Q}^{(n)}_{{\widetilde B}^{n} I^n}=0.
$$

\begin{definition}
\label{d2} Given a $q$-matrix $Q$ on the set $I$, we say that a
function $h:I \to \RR$ is $Q-$harmonic if it verifies $ Q h = 0$.
\end{definition}

From the definition it is clear that $h$ is $ Q-$harmonic iff
$e^{t Q} h= h$, for all $t\ge 0$. In the next proposition we
present a result that we will need in what follows. Its proof is
standard and it is based on the Doob' sampling theorem.
\medskip


\begin{proposition}
\label{p1144}
A function $h:I^n \to \RR$
is ${\bar Q}^{(n)}-$harmonic if and only if
$$
\EE_{i} \big( h(X_{\tau \land T_{{\widetilde B}^n}})\big) =h(i)
\hbox{ for } i\in I^{n} \hbox{ and any stopping time } \tau\le \infty.
$$

\end{proposition}

%
%

\medskip

The class of ${\bar Q}^{(n)}-$harmonic functions,
denoted by ${\cal H}^n$, is a linear space with dimension
$\hbox{ dim }{\cal H}^{n} =|{\widetilde B}^n|$.
Indeed, for each $k\in {\widetilde B}^n$
the function $h^k(i)=\EE_i({\bf 1}_k(X_{T_{{\widetilde B}^n}}))$
is the unique harmonic function which verifies
$h^k(j)=\delta_{k j}$ for $j \in {\widetilde B}^n$.
The class of these harmonic functions constitutes a basis for
${\cal H}^n$.

\medskip

\begin{proposition}
\label{p3}
Let $(I, {\cal T})$ be a finite tree rooted at $r$.
The matrix $H:= U_{I^n I^n}-V^{(n)}$ is
symmetric, and its columns
generate the space ${\cal H}^{n}$ of ${\bar Q}^{(n)}-$harmonic
functions. Moreover, the columns of
$\;U_{I^n \; {\widetilde B}^n}$
is a basis of this space.
\end{proposition}

\begin{proof}
First, let us introduce the matrices
$W = (W_{ik}: i\in I^n, k\in {\widetilde B}^{n})$,
$E = (E_{i \ell} : i\in I^n, \ell\in B^{n+1})$,
$D= (D_{ik} : i\in I^n, k\in {\widetilde B}^n)$,
whose terms are
$$
W_{ik} = \PP_i \{ X_{T_{{\widetilde B}^n}} = k \}, \;
E_{i \ell} = \PP_i \{ X_{T_{B^{n+1}}} = \ell \}, \;
D_{ik} = \PP_i \{ X_{T_{B^{n+1}}}\in S_k \}.
$$
Let $W^k$ be the $k$ column of $W$, with
$k\in {\widetilde B}^n$. We notice that $h^k = W^k$ then
$(W^k: k\in {\widetilde B}^n)$
is a basis of ${\cal H}^{n}$. In particular
${\bar Q}^{(n)}W^ k=0$.

\medskip

From definition $D_{ik}=\sum_{\ell\in S_k } E_{i \ell}$,
or equivalently $D= E M^t$ where $M^t$ is the
transposed of the incidence matrix $M=(M_{k \ell} :
k\in {\widetilde B}^n, \ell\in B^{n+1})$,
with $M_{k \ell}=1$ if $\ell\in S_k$ and
$M_{k \ell}=0$ otherwise.

\medskip

Let $i\in I^n$ and  $k\in {\widetilde B}^n$. Since
$$
\PP_i \{ T_k < \infty \} = \sum_{j\in {\widetilde B}^n}
\PP_i \{ X_{T_{{\widetilde B}^n}} = j \} \PP_j \{ T_k < \infty \}
\hbox{ and }
U_{ik} =
\PP_i \{ T_k < \infty \} U_{kk},
$$
we find $U_{ik} =
\sum_{j\in {\widetilde B}^n}
\PP_i \{ X_{T_{{\widetilde B}^n}} = j \} U_{jk}$. Hence we obtain
\begin{equation}
\label{re6}
U_{I^n \; {\widetilde B}^n } =
WU_{{\widetilde B}^n \; {\widetilde B}^n} \hbox{ and so }
W=U_{I^n \; {\widetilde B}^n}
(U_{{\widetilde B}^n\; {\widetilde B}^n})^{-1}.
\end{equation}
Analogously we get $E = U_{I^n \; B^{n+1}}
(U_{B^{n+1} \; B^{n+1}})^{-1}$.
From the equality $D= E M^t$ we find \hfill\break
$D=U_{I^n \; B^{n+1}}
(U_{ B^{n+1}\; B^{n+1}})^{-1} M^t$.
Since
$U_{i \ell}=w_{|i|}=U_{ik}$ when $k\in {\widetilde B}^n$, $\ell\in S_k$,
we obtain
\begin{equation}
\label{re7}
U_{I^n \; B^{n+1}}= U_{I^n {\widetilde B}^n} M,
\end{equation}
and then
$D=U_{I^n \;{\widetilde B}^n} M (U_{B^{n+1}\; B^{n+1}})^{-1}
M^t$.
Let us show
\begin{equation}
\label{re8}
H = U_{I^n \; {\widetilde B}^n} M
(U_{B^{n+1} \; B^{n+1}})^{-1}
M^t U_{{\widetilde B}^n\; I^n},
\end{equation}
or equivalently $H = DU_{{\widetilde B}^n \;I^n}$. For $i,j\in I^n$ we
have
\begin{eqnarray*}
U_{ij} & =&
\EE_i \Big( \int_0^{\infty} {\bf 1}_{\{ X_t=j\}}dt \Big) \\
& =&
\EE_i \Big(\int_0^{T_{B^{n+1}}} {\bf 1}_{\{X_t=j\}}dt
\Big)
+\EE_i \Big(T_{B^{n+1}}<\infty, \EE_{X_{T_{B^{n+1}}}}
\Big( \int_0^{\infty} {\bf 1}_{\{X_t=j\}} dt \Big) \Big).
\end{eqnarray*}
Hence
$U_{ij} = V^{(n)}_{ij} + \sum_{\ell\in B^{n+1}}
\PP_i \{X_{T_{B^{n+1}}}=\ell\} U_{\ell j}$,
or equivalently
\begin{equation}
\label{re60}
U_{ij}= V^{(n)}_{ij}+\EE_i(T_{B^{n+1}}<\infty, U_{ X_{T_{B^{n+1}}} j}) \;.
\end{equation}
Then $H_{ij}=\EE_i(T_{B^{n+1}}<\infty, U_{ X_{T_{B^{n+1}}} j})$ and
by using (\ref{re7}) we find
$$
H_{ij} = \sum_{\ell\in B^{n+1}}
\PP_i \{X_{T_{B^{n+1}}}=
\ell\} U_{\ell j} = \sum_{k\in {\widetilde B}^n}
\PP_i \{X_{T_{B^{n+1}}}\in  S_k\} U_{kj}
\hbox{ for } i,j\in I^n\;,
$$
which gives us $H = DU_{{\widetilde B}^n \;I^n}$, that is (\ref{re8})
holds.
From (\ref{re8}) we deduce $\hbox{ rank } H =
\hbox{ rank } U_{{\widetilde B}^n \; I^n} = |{\widetilde B}^n| =
\hbox{ dim } {\cal H}^{n}$. On the other hand,
from (\ref{re6}) and (\ref{re8}) we get
\begin{equation}
\label{re9}
H = W U_{{\widetilde B}^n\; {\widetilde B}^n} M
(U_{B^{n+1} \; B^{n+1}})^{-1}
M^t U_{{\widetilde B}^n \; {\widetilde B}^n} W^t.
\end{equation}
From ${\bar Q}^{(n)}W =0$ we obtain ${\bar Q}^{(n)}H =0$.
Therefore, the columns of $H$ belong to the space ${\cal H}^n$.
Given that $rank(H)=dim({\cal H}^n)$ the columns of $H$
generate this space. On the other hand from
(\ref{re8})
the columns of $U_{I^n \; {\widetilde B}^n}$ generate
${\cal H}^n$. Since the rank
of this matrix is equal the dimension of ${\cal H}^n$ the
Proposition is shown. $\Box$
\end{proof}

\medskip

\section{Harmonic Functions and the Martin Kernel}

\medskip

From now on we assume that $(I, {\cal T})$ is an infinite rooted
tree. We also assume that each branch is infinite.
We consider the minimal transition semigroup ${\widehat
P}_t$ associated to $\widehat Q$ the extension of $Q$ to $I\cup
\{\partial_r\}$ made in (\ref{re81}). One way to construct this
semigroup is by truncating the state space by an increasing
sequence of finite sets and then use \cite{anderson} Proposition
2.14. Let $\widehat X=(\widehat X_t: 0\le t <\widehat \zeta)$ be a
time continuous Markov process with infinitesimal generator
$\widehat Q$ and lifetime $\widehat \zeta$. If we stop $\widehat
X$ at the hitting time of $\partial_r$ we obtain a Markov process
$X=(\widehat X_t: 0\le t <\zeta)$ whose state space is $I$ and
lifetime $\zeta=T_{\partial_r}\wedge\widehat \zeta$. The
infinitesimal generator for $X$ is given by $Q$. We denote by
$(P_t)$ the semigroup associated to $X$ and by $V=\int_0^\infty
P_t dt$, the potential induced on $I$. We will denote by $Y=(Y_n:
n \in \NN)$ the discrete skeleton on $I$ induced by $X$.

\medskip

Let $I^n=\{i\in I: |i|\le n\}$.
As in the previous section
$V^{(n)}$ is the potential associated to $Q_{I_n I_n}$ and
${\cal H}^n$ is the set of ${\bar Q}^{(n)}$-harmonic functions in $I^n$.
Consider the chain
$X^{(n)} :=(X_t: t<T_{\partial_r}\land T_{B^{n+1}})$
killed at $B^{n+1}\cup \{\partial_r\}$,
with generator $Q_{I^n I^n}$. The Markov semigroup is
$P^{(n)}_t=e^{tQ_{I^n I^n}}$
and $V^{(n)}= \int^\infty_0 P^{(n)}_t dt = -Q_{I^n I^n}^{-1}$
is the associated potential. Clearly we have
$(P^{(n)}_t)_{ij}\le (P^{(n+1)}_t)_{ij}$ and
$V^{(n)}_{ij}\le V^{(n+1)}_{ij}$
for $i,j \in I^n$. Moreover, by the Monotone Convergence Theorem
their limits are $(P_t)$ and $V$, respectively.
From (\ref{re60}) we get $V^{(n)}_{ij}\le U_{ij}$, then
$V\le U$.

\medskip

Let us see, by a classical procedure (for instance see \cite{cartier}),
that $X_{\zeta}$ is a well defined variable in
$I\cup \partial_{\infty}\cup \partial_r $. In the case
$T_{\partial_r}<\infty$
this is obvious because $T_{\partial_r}=\zeta$ and
$X_\zeta=\partial_r$. So we can assume $T_{\partial_r}=\infty$.
We define $R_n=\inf\{t\ge 0:\; |X_t|\ge n\}$ and
$R_\infty := \lim\limits_{n\to \infty} R_n$.
An argument based on Borel Cantelli Lemma
shows that the set of trajectories visiting a site $j\in I$ an
infinite number times by $(Y_n)$, has $\PP_i-$measure 0. In fact
for such trajectories we necessarily have $T_{\partial_r}<\infty$.
The trajectories that visit each site of $I$ only a finite number
of times and are not absorbed at $\partial_r$ must converge to a
point in the boundary $\partial_\infty$ (see (\ref{re3})).
Therefore $\zeta=T_{\partial_r}\wedge R_\infty$ and
$X_\zeta$ is well defined. It verifies
\begin{equation}
X_{\zeta}= \partial_r \hbox{ if }T_{\partial_r}<R_\infty
\hbox{ and }
X_{\zeta}=\lim\limits_{n\to \infty} X_{R_n}=
\lim\limits_{n\to \infty} X_{{\zeta}}(n)\in \partial_\infty
\hbox{ if }
R_\infty\le T_{\partial_r}\;.
\end{equation}
Here, as already introduced, $X_{\zeta}(n)$ is the
point at level $n$ in $geod(r, X_{\zeta})$.

\medskip

The tree matrix is said to be {\it transient} whenever
$\PP_r\{T_{\partial_r}<\infty\}<1$ or equivalently
$\PP_r\{X_{\zeta}\in \partial_\infty\}>0$. Otherwise, the tree
matrix is said to be {\it recurrent}. This classification
corresponds to the recurrence or transient property for the chain
reflected at $r$. For a simple criterion on transience see
\cite{tlyons}.

\medskip

Since $U_{i \partial_r}=
U_{\partial_r i}=0$ for every $i\in I\cup \{\partial_r\}$,
equality (\ref{re60})
can be written as
$$
U_{ij}=V^{(n)}_{ij}+\EE_i(U_{X_{T_{B^{n+1}}}\, j}).
$$
From $U_{X_{T_{B^{n+1}}} j}\le U_{jj}$ and
$\lim\limits_{n\to \infty} U_{X_{T_{B^{n+1}}}\, j}=U_{X_\zeta \, j}$
$\PP_i-$a.e., we obtain
$$
\lim\limits_{n\to \infty}\EE_i(U_{X_{T_{B^{n+1}}}\, j})=
\EE_i(U_{X_\zeta \, j}).
$$
By combining these relations with
$\lim\limits_{n\to \infty}V_{ij}=V^{(n)}_{ij}$,
allow us to get
\begin{equation}
\label{re21}
U_{ij}= V_{ij}+ \EE_i(U_{ X_{\zeta}\, j})=
V_{ij}+ \int_{\partial_{\infty}}
U_{ \eta j} \; \PP_i \{ X_{\zeta}\in d \eta\} \;.
\end{equation}
Given that $V_{ij}=V_{ji}=\PP_j\{T_i<\infty\} V_{ii}$
the following limit exists
\begin{equation}
\label{exis22}
V_{i \xi}:=\lim\limits_{j\to \xi} V_{ij}=V_{ii}\cdot
\lim\limits_{j\to \xi}\PP_j\{T_i<\infty\}\ge 0,
\hbox{ for } i\in I, \xi\in \partial_{\infty}.
\end{equation}
Therefore, passing to the limit $j\to \xi \in \partial_\infty$
in relation (\ref{re21}) and using
the Monotone Convergence Theorem lead to
\begin{equation}
\label{4568}
U_{i\xi}=V_{i\xi}+\int_{\partial_{\infty}} U_{\eta \xi}
\PP_i \{X_{\zeta}\in d \eta \} \;.
\end{equation}

\medskip

A conclusion derived from (\ref{re21}) is that the recurrent case
$\PP_r\{X_{\zeta}\in \partial_\infty\}=0$
is completely characterized by the equality $V=U$.
In particular the tree matrix $U$
is the potential of $(X_t)$.

\medskip

In the transient case we denote by $\mu$ the exit measure on the
boundary $\partial_\infty$, that is the probability
measure defined on $\partial_\infty$ by
\begin{equation}
\label{muab}
\mu(\bullet)=\PP_r\{X_{\zeta}\in \bullet \; \big | \; X_{\zeta}\in
\partial_\infty\}.
\end{equation}

\begin{remark}
\label{r11}
If $U$ is unbounded, that is
$w_n$ tends to infinity as $n$ increases, the measure $\mu$ is atomless.
In fact, from (\ref{4568}) we get
$$
\infty > w_0=U_{r \xi}\ge  \int_{\partial_{\infty}\setminus \{\xi\}}
U_{ \eta \xi} \PP_r \{ X_{\zeta}\in d \eta \}+ \infty \cdot \PP_r
\{ X_{\zeta} = \xi \}.
$$
\end{remark}

In what follows we concentrate on the transient case. Nevertheless,
when appropriate, we shall point out the corresponding results for
the recurrent case.

\medskip

\subsection{Harmonic Functions}
\label{subs11}

\medskip

In this subsection we study basic properties of the harmonic
functions on $I$. We notice that the restriction of a ${\widehat
Q}-$harmonic function to $I$ is not necessarily $Q-$har\-mo\-nic.
An example of this is the constant ${\bf 1}$ function. In fact,
the unique ${\widehat Q}-$harmonic functions whose restrictions
are $Q-$harmonics are those vanishing at $\partial_r$. Obviously
the reciprocal also holds, that is, the only ${\widehat
Q}-$harmonic extension of a $Q-$harmonic function is the one
extended by $0$ at $\partial_r$. In the sequel a harmonic function
is to be understood as a $Q-$harmonic function, and for a function
defined on a subset of $I\cup
\partial_\infty$ we assume implicitly that it takes the value 0 at
$\partial_r$, unless otherwise is specified.

\medskip

In what follows an important role is played by the function
\begin{equation}
\label{gfun}
\bar g(j)=\PP_j\{T_{\partial_r}<\infty\}, \; j\in I\cup\{\partial_r\}\, ,
\end{equation}
which is the Martin kernel for $\widehat Q$ at $\partial_r$.
We point out that both $\bar g$ and $1-\bar g$ are ${\widehat Q}-$harmonics, but
only $1-\bar g$ is $Q-$harmonic.
We also note that $\bar g$ is nonnegative and decreasing on each branch,
which allows to define for $\eta \in \partial_\infty$
$$
\bar g(\eta):=\lim\limits_{j\to \eta} \PP_j\{T_{\partial_r}<\infty\}.
$$

\medskip

Given $g:I\to {\overline \RR}$  an
extended real function defined on the tree, we consider the
sequence of functions $(g_n)$ defined on the boundary by
$$
g_n(\xi)=g(\xi(n)) \hbox{ for } n\in \NN
\hbox{ and } \xi\in \partial_{\infty}\;.
$$
This notion enable us to study limiting properties
on the boundary for functions defined on the extended tree.
\begin{definition}
\label{def41}
Let $g:I\to {\overline \RR}$ and
$\varphi:\partial_{\infty}\to {\overline \RR}$.
We put $\lim g= \varphi$
pointwise (respectively $\mu-$a.e.) if
$\lim\limits_{n\to \infty} g_n=\varphi$
pointwise (respectively $\mu-$a.e.).
\end{definition}

Let ${\bar R}_n:=\inf\{t\ge 0:\; |X_t|\ge n \hbox{ or }
X_t=\partial_r\}$.
A standard argument gives,
$$
h:I\to \RR \hbox{ is harmonic } \Leftrightarrow
\left[
\forall n\ge 1, \; \forall \tau \hbox{ stopping time }:\;
\forall i\in I ,\;
h(i)=\EE_i\left(h(X_{\tau\wedge {\bar R}_n})\right)\right]\,.
$$
%
%
%
In the transient case, an application of the Dominated Convergence
Theorem and the Fatou's Theorem  gives
that for any bounded harmonic function $h:I\to \RR$  the limit
$\varphi=\lim h$ exists $\mu-$a.e. and moreover
$$
h(i)=\EE_i\big(\varphi\big(X_{\zeta}\big)\big).
$$
Indeed, this is a consequence of Theorem 2.6 in \cite{cartier}, because $h$ is bounded
if and only if $h/(1-\bar g)$ is bounded. Thus, if $h_1,h_2$ are bounded harmonic
functions such that $\lim
h_1=\lim h_2$ $\mu-$a.e. then $h_1\equiv h_2$. Obviously  in the
recurrent case the unique bounded harmonic function is $h\equiv
0$.

\begin{proposition}
\label{c2}
If $U$ is bounded then the tree matrix is transient.
\end{proposition}

\begin{proof}
The function $h(i)=U_{i\eta}$ is harmonic,
bounded and non-zero which implies that the tree matrix
must be transient. $\Box$
\end{proof}

\medskip

A distinguished class of harmonic functions is given by the Martin
kernel at $\infty$, see \cite{cartier}, \cite{KSK} or
\cite{massimo1987}.

\begin{definition}
\label{d31}
The Martin kernel (at $\infty$),
$\kappa:I\times \partial_{\infty}\to \RR$ is given by
$$
\kappa(i,\eta):=\lim\limits_{j\to \eta}\frac{V_{ij}}{V_{rj}},
\hbox{ for } i\in I, \eta\in \partial_{\infty}.
$$
\end{definition}

\medskip

It is well known that $\kappa(\bullet,\eta)$ is a well defined
harmonic function on $I$ (see \cite{cartier} or
\cite{massimo1987}).

Consider $i\in I,\, \xi \in \partial_\infty$ and $n> |i\land
\xi|$. Take $j=\xi(n)$ and denote $C^n=\partial_\infty(\xi(n))$.
The strong Markov property implies
$$
\PP_i\{X_{\zeta}\in C^n\}= \PP_i\{T_j<\infty\}\PP_j\{X_{\zeta}\in
C^n\}= {V_{ij}\over V_{jj}} \PP_{\xi(n)}\{X_{\zeta}\in C^n\}.
$$
On the other hand $\PP_i\{X_{\zeta}\in C^n\}= \PP_i\{T_{i\land
\xi}<\infty\}\PP_{i\land \xi} \{X_{\zeta}\in C^n\}$. Then
$$
{{V_{ij}}\over {V_{rj}}}={{\PP_i\{X_{\zeta}\in C^n\}} \over
{\PP_r\{X_{\zeta}\in C^n\}}}= \frac{\PP_i\{T_{i\land
\xi}<\infty\}} {\PP_r\{T_{i\land \xi}<\infty\}},
$$
Passing to the limit we get that
\begin{equation}
\label{rn111} \kappa(i,\xi)=\lim\limits_{j\to \xi}
{{\PP_i\{X_{\zeta}\in
\partial_\infty(j)\}} \over{\PP_r\{X_{\zeta}\in
\partial_\infty(j)\}}}
={\PP_i\{X_{\zeta}\in \partial_\infty(\xi(n))\}\over
\PP_r\{X_{\zeta}\in \partial_\infty(\xi(n))\}}
=\frac{\PP_i\{T_{i\land \xi}<\infty\}} {\PP_r\{T_{i\land
\xi}<\infty\}}.
\end{equation}

In particular $\kappa(i,\bullet)$ is the Radon-Nykodim derivative
of $\PP_i\{X_\zeta \in \bullet\}$ with respect to $\PP_r\{X_\zeta
\in \bullet\}$ (see \cite{cartier}) so
$$
U_{i \xi}=V_{i \xi}+\int_{\partial_\infty} U_{\xi \eta}\;
\kappa(i,\eta)\; \PP_r\{X_\zeta \in d\eta\}\;.
$$

\begin{remark}
When the tree is recurrent, that is $V=U$, the Martin kernel
is easily computed as
$$
\kappa(i,\eta)=\lim\limits_{j\to \eta} \frac{V_{ij}}
{V_{rj}}=\frac{U_{i \eta}}{w_0}.
$$
Therefore, $\{U_{\bullet \eta}/w_0: \eta \in \partial_\infty\}$ is
the Martin kernel.
\end{remark}

\medskip

\subsection{ Regular and Accessible Points}

\medskip

A close study  between $U$ and the potential $V$, in the transient
case, needs the description of the regular points on
$\partial_\infty$. In the classical setting regularity is needed
for the continuity up to the boundary for the Dirichlet boundary
problem (see for example \cite{chung1995}, Theorem 1.23). In our
context see Lemma \ref{70} (ii) .

\medskip

\begin{definition}
\label{d10}
A point $\eta \in
\partial_\infty$ is said to be regular if $\bar g(\eta)=0$, that is
$$
\lim\limits_{j\to \eta} \PP_j\{T_{\partial_r}<\infty\}=0,
$$
and is said to be accessible if it belongs to the closed
support of $\mu$, that is
$$
\PP_r\{X_{\zeta}\in [\eta(n),\infty]\}>0 \;\hbox{ for all } n.
$$
If $\eta$ is not regular we say it is irregular and if it is not
accessible we say it is inaccessible. We denote by
$\partial_\infty^{reg}$ the set of regular points and by
$\partial_\infty^{inac}$ the set of inaccessible points.
\end{definition}

\medskip

The classification on accessible and inaccessible points is the
same if instead of $\PP_r$, we use $\PP_i$  for any $i\in I$.
Similarly $\eta$ is regular if and only if $\lim\limits_{j\to
\eta}\PP_j\{T_i<\infty\}=0$ for all $i\in I$. From (\ref{exis22})
this is exactly the case when $V_{i\eta}=0$.

\medskip

\begin{lemma}
\label{70} (i) The measure $\mu$ concentrates on the set of
regular points: $\mu(\partial_\infty^{reg})=1$.

\medskip

\noindent (ii) A point $\eta\in \partial_\infty$ is regular if and
only if any bounded continuous real function $f$ defined in
$\partial_\infty\cup \{\partial_r\}$ with $f(\partial_r)=0$, verifies

\begin{equation}
\label{re20}
\lim\limits_{j\to \eta}\EE_j(f(X_{\zeta}))=f(\eta).
\end{equation}

\medskip

\noindent (iii) Every regular point is accessible.

\end{lemma}

\medskip

\begin{proof}
(i) The function $\bar g(j)=\PP_j\{T_{\partial_r}<\infty\}$ is
bounded and $\widehat Q-$harmonic and verifies $\bar
g(\partial_r)=1$. Using  that $\bar g(r)=\EE_r(\bar
g(X_{T_{B^n}\wedge T_{\partial_r}}))$, the Dominated Convergence
Theorem gives
$$
\bar g(r)=\EE_r(\bar g(X_{\zeta}))=\PP_r\{T_{\partial_r}<\infty\} +\int
\bar g(\xi) \PP_r\{X_{\zeta} \in d\xi\}.
$$
From this relation we conclude that $\bar g=0$ $\mu-$a.e..
Therefore $\mu(\partial_\infty^{reg})=1$.

\medskip

(ii) Since $f$ is continuous and bounded,
for every $\varepsilon>0$ fixed there exists $n$ such that
$|f(\xi)-f(\eta)|\le \varepsilon$ if $\xi\in [\eta(n),\infty]\cap
\partial_\infty$. Then for $j\in [\eta(n),\infty)$ we have
$$
|\EE_j(f(X_{\zeta}))-f(\eta)|\le 2 M \PP_j\{T_{\eta(n)}<
\zeta\}+2\varepsilon \PP_j\{\zeta \le T_{\eta(n)}\},
$$
where $M$ is any bound for $f$. From this inequality we conclude
that
$$
\limsup\limits_{j\to \infty} |\EE_j(f(X_{\zeta}))-f(\eta)|\le
2\varepsilon,
$$
and then we obtain the desired limit in (\ref{re20}).

\medskip

Conversely, assume now that (\ref{re20}) holds for $f={\bf
1}_{\partial_\infty}$ (so $f(\partial_r)=0$). Then
$$
\EE_j(f(X_{\zeta}))=\PP_j\{R_\infty\le
T_{\partial_r}\}=
1-\PP_j\{T_{\partial_r}<\infty\} \mathop{\longrightarrow}
\limits_{j\to \eta}1=f(\eta),
$$
proving that $\eta$ is regular.

\medskip

(iii) Let $\eta$ be a regular point. Take any $n$ and
consider $f$ the indicator function of
$A=\partial_\infty(\eta(n))$. For large $j$ we have
$\PP_j\{X_{\zeta}\in A\}>0$ which implies
$\PP_r\{X_{\zeta}\in A\}>0$ and $\eta$ is accessible. $\Box$
\end{proof}

\medskip

\subsection{Potential for inaccessible points}

\medskip

We will show that, in the set of inaccessible points, the
potential reduces to the recurrent case. For every inaccessible
point $\eta$ we denote by $N^\eta$ the smallest integer $n\ge 0$
for which $\mu(\partial_\infty(\eta(n)))=0$. Since
$\{\partial_\infty(\eta(N^\eta)):\; \eta \in
\partial_\infty^{inac}\}$ is an open cover of $\partial_\infty^{inac}$,
we can find a finite or countable set
$\{\eta_s:\; s\in {\cal N}\} \subseteq \partial_\infty^{inac}$ such that
$$
\partial_\infty^{inac}=
\biguplus\limits_{s \in {\cal N}} (\sqsubset_s\cap \; \partial_\infty),
$$
where $\sqsubset_s=[m_s,\infty]$ is the infinite tree
hanging from $m_s:=\eta_s(N^{\eta_s})$, $s\in {\cal N}$.

\medskip

\begin{lemma}
\label{l10} Let $j\in \sqsubset_s$ then
$\PP_j\{T_{m_s}<\infty\}=\PP_j\{T_{m_s}<\zeta\}=1$, that is, the
restriction of $U$ to the subtree hanging from $m_s$ is recurrent.
This also implies that every inaccessible point is irregular.
\end{lemma}

\medskip

\begin{proof}
Observe that $\PP_j-$a.e. on the set $\{\zeta \le T_{m_s}\}$
we have $X_{\zeta}\in \; \sqsubset_s\cap\;
\partial_\infty \subseteq \partial_\infty^{inac}$. Since
$0=\PP_r\{X_{\zeta}\in \sqsubset_s\}\ge \PP_r\{T_j<\infty\}
\PP_j\{X_{\zeta}\in \sqsubset_s\}$,
we conclude $\PP_j\{\zeta \le T_{m_s}\}=0$ and the result follows.
$\Box$
\end{proof}

\medskip

\begin{proposition}
For inaccessible points the potential $V$ verifies
\begin{equation}
\label{la450}
V|_{\sqsubset_s\times\sqsubset_s}=
U|_{\sqsubset_s\times\sqsubset_s}-(U_{m_sm_s}-V_{m_sm_s})
\hbox{ and } V|_{\sqsubset_s\times\sqsubset_t}
\hbox{ is constant for } s\neq t.
\end{equation}
\end{proposition}

\begin{proof}
Consider $i,j \in \sqsubset_s$ we deduce from (\ref{re21})
and Lemma \ref{l10} that
$$
U_{ij}-V_{ij}=\int_{\partial_\infty^{reg}} U_{j\eta}
\PP_i\{X_{\zeta}\in d\eta\}=\int_{\partial_\infty^{reg}}
U_{m_s\eta} \PP_{m_s}\{X_{\zeta}\in d\eta\},
$$
which implies the first relation in (\ref{la450}).
Finally if $i\in \sqsubset_s,\; j\in \sqsubset_t$, $s\neq t$,
we obtain that
$$
V_{ij}=U_{m_sm_t}-\int_{\partial_\infty^{reg}} U_{m_t\eta}
\PP_{m_s}\{X_{\zeta}\in d\eta\},
$$
which implies the second part in (\ref{la450}). $\Box$
\end{proof}

\begin{remark}
Since $V|_{\sqsubset_s\times\sqsubset_s}$ is strictly positive and
it is equal to $U|_{\sqsubset_s\times\sqsubset_s}$ minus a
constant, it follows that the potential
$V|_{\sqsubset_s\times\sqsubset_s}$ is a tree matrix. We recall
that this is exactly the case when the tree matrix is recurrent as
it is $U|_{\sqsubset_s\times\sqsubset_s}$, see Lemma \ref{l10}.
\end{remark}

\medskip

\noindent {\bf Example}. The following example shows that not all
accessible points are regular. On figure 1 we have chosen a
particular tree, rooted at $r=0$, consisting on a special branch
determined by the nodes $0,1,2,...$ and subtrees $T_0,T_1,...$.

\bigskip

\centerline{\resizebox{6cm}{!}{\begin{picture}(0,0)%
\includegraphics{arbol1.pstex}%
\end{picture}%
\setlength{\unitlength}{3947sp}%
\begingroup\makeatletter\ifx\SetFigFont\undefined%
\gdef\SetFigFont#1#2#3#4#5{%
  \reset@font\fontsize{#1}{#2pt}%
  \fontfamily{#3}\fontseries{#4}\fontshape{#5}%
  \selectfont}%
\fi\endgroup%
\begin{picture}(5857,6153)(2106,-6673)
\put(3601,-2161){\makebox(0,0)[lb]{\smash{\SetFigFont{17}{20.4}{\rmdefault}{\mddefault}{\updefault}{\color[rgb]{0,0,0}$1$}%
}}}
\put(4231,-2761){\makebox(0,0)[lb]{\smash{\SetFigFont{17}{20.4}{\rmdefault}{\mddefault}{\updefault}{\color[rgb]{0,0,0}$2$}%
}}}
\put(6226,-4786){\makebox(0,0)[lb]{\smash{\SetFigFont{17}{20.4}{\rmdefault}{\mddefault}{\updefault}{\color[rgb]{0,0,0}$n$}%
}}}
\put(3101,-1636){\makebox(0,0)[lb]{\smash{\SetFigFont{17}{20.4}{\rmdefault}{\mddefault}{\updefault}{\color[rgb]{0,0,0}$0$}%
}}}
\put(6871,-5386){\makebox(0,0)[lb]{\smash{\SetFigFont{17}{20.4}{\rmdefault}{\mddefault}{\updefault}{\color[rgb]{0,0,0}$n+1$}%
}}}
\put(2476,-736){\makebox(0,0)[lb]{\smash{\SetFigFont{17}{20.4}{\rmdefault}{\mddefault}{\updefault}{\color[rgb]{0,0,0}$\partial_r$}%
}}}
\put(2176,-2536){\makebox(0,0)[lb]{\smash{\SetFigFont{17}{20.4}{\rmdefault}{\mddefault}{\updefault}{\color[rgb]{0,0,0}$T_0$}%
}}}
\put(2701,-3061){\makebox(0,0)[lb]{\smash{\SetFigFont{17}{20.4}{\rmdefault}{\mddefault}{\updefault}{\color[rgb]{0,0,0}$T_1$}%
}}}
\put(3301,-3661){\makebox(0,0)[lb]{\smash{\SetFigFont{17}{20.4}{\rmdefault}{\mddefault}{\updefault}{\color[rgb]{0,0,0}$T_2$}%
}}}
\put(5326,-5686){\makebox(0,0)[lb]{\smash{\SetFigFont{17}{20.4}{\rmdefault}{\mddefault}{\updefault}{\color[rgb]{0,0,0}$T_n$}%
}}}
\end{picture}
}}
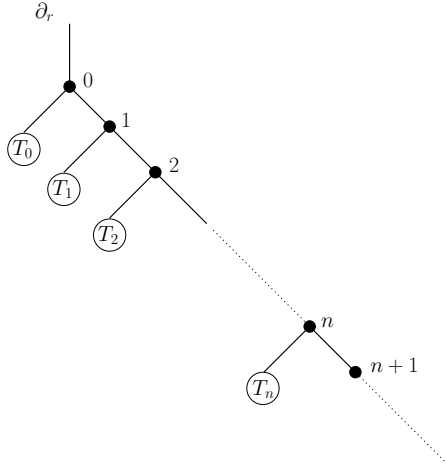
\centerline{ Figure 2. }

\medskip

Each subtree $T_k$ is regular in the sense that any node $s\in
T_i$ at level $m$ measured from the root of $T_k$ (thus at level
$m+k+1$ measured from $r$) has a constant number of descendants
equal to $s_{m+k+2}$. The weight function $w_n$ verifies $w_0=1$
and $w_{n+1}-w_n=2^n(w_n-w_{n-1})$. Also we take $s_p=2^p$. In
this way
$$
{Q_{s\, s-}\over (-Q_{ss})}={(w_{m+k+1}-w_{m+k})^{-1} \over
(w_{m+k+1}-w_{m+k})^{-1}+s_{m+k+2}
(w_{m+k+2}-w_{m+k+1})^{-1}}=1/3.
$$
The level process on each subtree $T_i$ is clearly a birth and
death chain with birth rate $2/3$ and death rate $1/3$. Therefore
$T_i$ is transient and henceforth
$$
\PP_0\{X_{\zeta}\in \partial_\infty(r_i)\}\ge
\PP_0\{T_{r_i}<\infty\}\PP_{r_i}\{X_{\zeta}\in
\partial_\infty(r_i)\}>0.
$$

On the other hand
$$
{Q_{k \, k-1}\over(-Q_{kk})}={(w_k-w_{k-1})^{-1}\over
(w_k-w_{k-1})^{-1}+2(w_{k+1}-w_k)^{-1}}=1-{1\over 2^{n-1}+1}.
$$
This implies that $\prod\limits_{k=1}^\infty {Q_{k\,
k-1}\over(-Q_{kk})}=a\in (0,1)$. Since $\PP_\eta
\{T_{\partial_r}<\infty\}\ge {1\over
w_0}\left(\prod\limits_{k=1}^n
{Q_{k\,k-1}\over(-Q_{kk})}\right)\ge {a\over w_0}$, the point
$\eta\in \partial_\infty$ determined by the special branch, is
irregular but accessible.

\medskip

\subsection{The Kernel at the Boundary is a Filtered Operator}
\label{subs12}

\medskip

Let us introduce the operator $W$, acting on $L^p(\mu)$, with kernel $U$.
We point out that $U$ and $W$ acting on
$\partial_\infty$ where introduced in \cite{lyons1} section 4, and
they are used in \cite{lyons2} section 2.3 to study the capacity
function on the boundary.

\medskip

\begin{definition}
\label{d80}
For any (positive) bounded, real and
measurable function $f$ with domain in $\partial_{\infty}$ we
define
$$
Wf(\eta)=\int_{\partial_{\infty}}U_{ \eta \xi} f(\xi) \mu(d \xi)
$$
which is also a (positive) real and measurable function.
\end{definition}

\medskip

We notice that the integral defining $W$ can be made over
$\partial_\infty$ or $\partial_\infty^{reg}$, because this last
set is of full measure $\mu$. We have from  (\ref{4568}) and
$w_0=U_{r \eta}$ that
$$
W{\bf 1}(\eta)=\int_{\partial_\infty} U_{\eta\xi}
\mu(d \xi)={w_0-V_{r\eta} \over \PP_r\{X_{\zeta}\in
\partial_\infty\}}.
$$
Then $Wf$ is bounded for any bounded $f$. Since $V_{r\eta}=0$ for
any regular point $\eta$, we conclude that $W{\bf 1}$ is constant
$\mu$-a.e., where this constant, denoted by $\alpha$, is given by
$\alpha={w_0} / {\PP_r\{X_{\zeta}\in
\partial_\infty\}}$. In general we have $W{\bf 1}\le \alpha$ in
$\partial_\infty$.

\medskip

The action of $W$ on measures is given by
$\nu W(A)=\int W{\bf 1}_A(\xi) \nu(d\xi)$. It is
direct to see that $\mu W=\alpha \mu$.  Then $\alpha^{-1}W$ is a
Markov operator preserving $\mu$. Hence, for every $p\ge 1$, the
operator $W:L^p(\mu)\to L^p(\mu)$ is well defined,
$||W||_p=\alpha$ and $W$ is self adjoint in $L^2(\mu)$.

\medskip

Recall notations $\partial_\infty(i)=[i,\infty]\cap
\partial_{\infty}$ made in (\ref{re2}) and
$geod(r,\xi)=(\xi(k): k\in \NN)$ for
$\xi \in \partial_{\infty}$. We put
$$
C^k(\xi)=\partial_{\infty}(\xi(k))=\{\eta \in \partial_\infty:\;
\xi(k)=\eta(k)\}.
$$

We also consider
$$
\Delta_k(w)=w_k-w_{k-1} \hbox{ for }
k\in \NN,\; \Delta_{-1}(w)=0.
$$
Notice that $\Delta_k(w)>0$ for $k\in \NN$.
For $f\in L^1(\mu)$ it is verified
\begin{equation}
\label{re30}
Wf(\eta)=\sum_{k\in \NN} w_k
\int_{C^k(\eta)\setminus C^{k+1}(\eta)} f d\mu=
\sum_{k\in \NN} \Delta_k(w) \int_{C^k(\eta)}
f d\mu.
\end{equation}

The set function $C^k$, with domain $\partial_{\infty}$, takes a
finite number of values. We denote by ${\cal F}_k$ the
$\sigma-$field in $\partial_{\infty}$ generated by the sets
$(C^k)$. This sequence of $\sigma-$fields is increasing and
generating, that is ${\cal F}_\infty=\sigma(\bf T)$. Thus, ${\cal
F}=({\cal F}_k: k\in \NN)$ is a generating filtration in
$\partial_{\infty}$. With this notation equality (\ref{re30}) can
be written as
\begin{equation}
\label{re31}
Wf(\bullet)=\sum_{k\in \NN} \Delta_k(w) \mu (C^k(\bullet))
\EE_{\mu} (f | {\cal F}_k)(\bullet).
\end{equation}

Now, consider on $\partial_\infty$ the following process
\begin{equation}
\label{basta1}
G=(G_n: n\in \NN) \hbox{ where }G_n(\eta)=
\sum_{k\ge n} \Delta_k(w) \mu(C^k(\eta)).
\end{equation}
Since $G_0=W{\bf 1}\le \alpha$ we obtain that $G_0$ is a
convergent series. On the other hand, since every regular point is
accessible we conclude that $\mu(C^k(\xi))>0$ for every $k\in
\NN,\; \xi \in
\partial_{\infty}^{reg}$ and in particular
$G_n>0,\; \mu-$a.e. for every $n\in \NN$. We also have
$$
G_n(\eta)=G_0-\sum_{k=0}^{n-1} \Delta_k(w) \mu (C^k(\eta))
\hbox{ is } {\cal F}_{n-1} \hbox{ measurable }.
$$
Therefore if $|\xi \land \eta|\ge n$ we have
$G_i(\eta)-G_{i+1}(\eta)=G_i(\xi)-G_{i+1}(\xi),\; i=0,\dots,n$.
Moreover, if $\xi, \eta$ are regular points then
$G_0(\eta)=G_0(\xi)=\alpha$ and
\begin{equation}
\label{igual}
G_i(\xi)=G_i(\eta),\hbox{ for all } i\le |\xi\land \eta|.
\end{equation}

The process $(G_n)$ is
${\cal F}$-predictable, positive,
bounded by $\alpha$ and decreasing to $0$ as $n\to \infty$.
Then $G_n \EE_{\mu}( \; | {\cal F}_n)$
converges to $0$ in $L^p(\mu)$ for every $p\in [1, \infty]$.
Therefore, integration by parts on (\ref{re31}) gives
\begin{equation}
\label{fil88}
W=\sum_{n\in \NN} (G_n - G_{n+1})\EE_{\mu}(\; | {\cal F}_n)=
\sum_{n\in \NN} G_n \big(\EE_{\mu}(\; | {\cal F}_n)-\EE_{\mu}(\; |
{\cal F}_{n-1})\big).
\end{equation}
This equality being in the sense of operators. Thus,
we have shown the following result.

\medskip

\begin{proposition}
\label{prostoch}
The self adjoint operator $W$ acting on $L^2(\mu)$ is an stochastic integral
operator (or a filtered operator), that is, there exists a
filtration ${\cal F}=({\cal F}_n)$ and $G=(G_n)$ a
${\cal F}-$predictable process, such that
$W=\sum_{n\in \NN} G_n
\big(\EE_{\mu}(\; | {\cal F}_n)-\EE_{\mu}
(\; | {\cal F}_{n-1})\big)$.
\end{proposition}

\medskip

For definitions and properties of stochastic integral operators
see \cite{dell1977}, and for its characterization
in the countable case see \cite{dart1988}.

\medskip

Let us consider
${\cal D}=\cup_{n\in \NN} L^2({\cal F}_n, \mu)$ the set
of simple functions over the algebra $\cup_{n\in \NN}{\cal F}_n$.
Clearly ${\cal D}$ is a dense subset in $L^2(\mu)$. Notice that
the operator
$L=\sum_{n\in \NN} G_n^{-1}
\big(\EE_{\mu}(\; | {\cal F}_n)-\EE_{\mu}(\; | {\cal F}_{n-1})\big)$
is well defined in ${\cal D}$.
As $G_n$ is ${\cal F}_{n-1}$ measurable
the following equalities hold on ${\cal D}$,
$$
LW=WL=\sum_{n\in \NN} \EE_{\mu}(\; | {\cal F}_n)-\EE_{\mu}(\; |
{\cal F}_{n-1})=\II_{\cal D}.
$$
Here $\II_{\cal D}$ is the identity on ${\cal D}$.
In particular, $Im(W)=W(L^2(\mu))$
contains ${\cal D}$, so $Im(W)$ is dense in $L^2(\mu)$. Since $W$
is a self adjoint operator, we get that $W$ is one-to-one.
Hence we can extend $L$ to $Im(W)$ by $Lg=f$ for $g\in Im(W),\;
g=Wf$. Therefore
$$
WL=\II _{Im(W)}, \;\; LW=\II _{L^2(\mu)}.
$$
We put $L=W^{-1}$ and we assume implicitly that its domain is
$Im(W)$, so
\begin{equation}
\label{re33}
W^{-1}=\sum_{n\in \NN} G_n^{-1}
\big(\EE_{\mu}(\; | {\cal F}_n)-\EE_{\mu}(\; | {\cal F}_{n-1})\big).
\end{equation}
Observe that $W^{-1}{\bf 1}=\alpha^{-1}$ $\mu-$a.e.. The operator
$-W^{-1}$ is a generator of a subMarkov kernel defined in the
boundary, that will be studied in section \ref{section600}.

\medskip

Let us compute $W^{-1}$ in $\cal D$. Fix s set
$C^n\in {\cal F}_n$. For $k\le n$ we denote by $C^k$ the
element in ${\cal F}_k$ such that $C^n\subseteq C^k$.
We also put $C^{-1}=\phi$.
From (\ref{re33}) we obtain
\begin{eqnarray}
\label{re34}
W^{-1} {\bf 1}_{C^n}&=&\sum_{k=0}^n \! G_k^{-1}
\big(\EE_{\mu}( {\bf 1}_{C^n} | {\cal F}_k)\!-\!\EE_{\mu}
({\bf 1}_{C^n} | {\cal F}_{k-1})\big)\\
\nonumber
&=&\;\sum_{k=0}^n \!G_k^{-1}
\Big(\!{{\mu (C^n)}\over {\mu (C^k)}}
{\bf 1}_{C^k}\!-\!{{\mu (C^n)}\over {\mu (C^{k-1})}}
{\bf 1}_{C^{k-1}}\!\Big)
\!=\!G_n^{-1} {\bf 1}_{C^n}\!+\!\sum_{k=0}^{n-1}\!
\big(\!G_k^{-1}\!- \!G_{k+1}^{-1}\! \big)
{{\mu (C^n)}\over {\mu (C^{k})}} {\bf 1}_{C^{k}}.
\end{eqnarray}

Let $\eta, \xi \in \partial_{\infty}$, $\eta\neq \xi$, and take
$n> |\eta \land \xi|$. Since $C^k(\xi)=C^{k}(\eta)$
for $k\le |\eta \land \xi|$ we get
\begin{equation}
\label{re32}
W^{-1} {\bf 1}_{C^n(\eta)}(\xi)=\sum_{k=0}^{|\eta \land \xi|}
\big( G_k^{-1}(\eta)- G_{k+1}^{-1}(\eta) \big) {{\mu (C^n(\eta))}
\over {\mu (C^{k}(\eta))}}.
\end{equation}
Then
$$
{{W^{-1} {\bf 1}_{C^n(\eta)}(\xi)}\over {\mu(C^n(\eta))}}=
-\sum_{k=0}^{|\eta \land \xi|} {{\Delta_k(w)}\over
{G_k(\eta)G_{k+1}(\eta)}}.
$$
Thus, for $\xi\neq\eta$ the following limit exists
\begin{equation}
\label{re35}
W^{-1}(\xi,\eta)=\lim\limits_{n \to \infty}
{W^{-1} {\bf 1}_{C^n(\eta)}(\xi)\over
{\mu(C^n(\eta))}}=-\sum_{k=0}^{|\eta \land \xi|} {{\Delta_k(w)}
\over {G_k(\eta)G_{k+1}(\eta)}}<0.
\end{equation}

\medskip

\begin{remark}
\label{ultima}
The operator
\begin{equation}
\label{Ma11}
{\underline {\bf W}}^{-1}:=W^{-1}-G_0^{-1}\EE_{\mu}=\sum_{n\ge 1}
G_n^{-1}
\big(\EE_{\mu}(\; | {\cal F}_n)-\EE_{\mu}(\; | {\cal F}_{n-1})\big)
\end{equation}
verifies ${\underline {\bf W}}^{-1}{\bf 1}=0$, and $-{\underline
{\bf W}}^{-1}$ is a generator of a Markov process in the boundary.
\end{remark}

In the next result we explicit the Dirichlet form associated to
$-W^{-1}$. More precisely, we get the Beurling-Deny formula
following closely the construction done in \cite{Fukushima1994}
(see Theorem 3.2.1). We compute it for simple functions using
mainly the fact that $(G_n)$ is predictable. Then it can be
extended by density arguments.

\begin{proposition}
\label{beurlingf}
Let ${\bf E}(f,g)=\int_{\partial_\infty} g\; W^{-1}f \, d\mu$
be the Dirichlet symmetric form associated to $-W^{-1}$ in
$L^2(\mu)$. Let $D=\{(\eta,\eta): \eta\in
\partial_\infty\}$ be the diagonal in $\partial_\infty^2$. Then
for all $f,g \in {\cal D}$, the set of simple functions, we have
$$
{\bf E}(f,g)=\frac{1}{2} \int\limits_{\partial_\infty\times
\partial_\infty \setminus D} (f(\eta)-f(\xi))(g(\eta)-g(\xi))\,
H(\eta,\xi)\, \mu\otimes \mu (d \eta,d \xi)+
\frac{1}{G_0}\int\limits_{\partial_\infty} f(\eta)g(\eta)\,
\mu(d\eta)\,,
$$
where
\begin{equation}
\label{Hdirichlet}
H=\sum_{n\ge 0}\sum_{j \in B^n}
\frac{1}{\mu(C_j)}\left(\frac{1}{G_{n+1}}-\frac{1}{G_{n}}\right)
{\bf 1}_{C_j\times C_j},
\end{equation}
with $C_j=\partial_\infty(j)$.
\end{proposition}

\begin{proof}
We notice that $H(\xi, \eta)$ in (\ref{Hdirichlet}) is well defined for
$\xi\neq \eta$ and it is symmetric because $G_{n+1}, G_n$ are constant
over $C_j$ for $j \in B^n$.

\medskip

We denote by $\EE_n=\EE_{\mu}(\; | {\cal F}_n)$ and  by
$\langle,\rangle$ the inner product in $L^2(\mu)$.
We follow the construction of ${\bf E}$ given in
\cite{Fukushima1994}.

\medskip

The resolvent
$R_\beta=\int_0^\infty e^{-\beta t} e^{-tW^{-1}} \, dt$ is given by
$$
\beta R_\beta=\sum\limits_{n\ge 0}\frac{\beta G_n}{\beta G_n+1}
\big(\EE_n-\EE_{n-1}\big)=
\sum\limits_{n\ge 0} h_n^{(\beta)}\, \EE_n,
$$
where $h_n^{(\beta)}=\frac{\beta G_n}{\beta G_n+1}-
\frac{\beta G_{n+1}}{\beta G_{n+1}+1} \in {\cal F}_n.$

\medskip

We have
$$
\begin{array}{ll}
\langle f,\beta R_\beta g \rangle &= \sum\limits_{n\ge 0} \int f
h^{(\beta)}_n \EE_n g \, d\mu=
\sum\limits_{n\ge 0} \int h^{(\beta)}_n \EE_n(f)  \EE_n(g) \, d\mu \\
& \\
&=\sum\limits_{n\ge 0}\sum\limits_{j\in B^n} \int_{C_j}
h^{(\beta)}_n \Big[\int_{C_j} f(\xi) \mu(d\xi) / \mu(C_j)\;
\int_{C_j} g(\eta) \mu(d\eta) / \mu(C_j)\Big] \, d\mu\\
& \\
&=\sum\limits_{n\ge 0} \int f(\xi) g(\eta) H^{(\beta)}_n(\xi,\eta) \,
\mu\otimes\mu(d\xi,d\eta),
\end{array}
$$
where $H^{(\beta)}_n=\sum_{n\ge 0}\sum_{j \in B^n}
\frac{1}{\mu(C_j)}\left(\frac{\beta G_n}{\beta G_n+1}-
\frac{\beta G_{n+1}}{\beta G_{n+1}+1}\right)
{\bf 1}_{C_j\times C_j}$. Outside the diagonal we have that
$$
\beta H^{(\beta)}_n(\xi,\eta)
\mathop{\longrightarrow}\limits_{\beta \to \infty} H_n(\xi,\eta).
$$
Since for $f,g \in {\cal D}$ with disjoint support we have
$$
{\bf E}(f,g)=\lim\limits_{\beta \to \infty}
-\beta \langle f, \beta R_\beta g\rangle
=-\int f(\xi) g(\eta) H(\xi,\eta) \, \mu\otimes\mu(d\xi,d\eta).
$$
Then, the result holds in this case.

\medskip

The only thing left to compute is
${\bf E}({\bf 1}_C,{\bf 1}_C)$, for any $C$ an atom of some
${\cal F}_n, n\ge 0$. This is done by linearity and the
following fact, which is direct to show
$$
{\bf E}({\bf 1}_C,{\bf 1})=\frac{1}{G_0} \int {\bf 1}^2_C \, d\mu.
$$
$\Box$
\end{proof}

\medskip

Hence, the diffusive part in the Beurling-Deny formula vanishes,
so the subMarkov process associated to $-W^{-1}$, is a pure jump
process. This conclusion can be also obtained directly by using
the arguments developed in \cite{albeverio2} Theorem 4.1.

\begin{remark}
Let ${\underline {\bf E}}(f,g)=\int_{\partial_\infty} g\;
{\underline {\bf W}}^{-1} f \, d\mu$
be the Dirichlet symmetric form associated to $-{\underline {\bf W}}^{-1}$
in $L^2(\mu)$. Then
for all $f,g \in {\cal D}$, the set of simple functions, we have
$$
{\underline {\bf E}}(f,g)=
\frac{1}{2}\int\limits_{\partial_\infty\times
\partial_\infty \setminus D} (f(\eta)-f(\xi))(g(\eta)-g(\xi))\,
H(\eta,\xi)\, \mu\otimes \mu (d \eta,d \xi)\;,
$$
that is, in this case the killing part disappears, as it must
happen by construction of $-{\underline {\bf W}}^{-1}$.
\end{remark}

\subsection{The Martin Kernel for Accessible Points}
\label{subs1200}

\medskip

From (\ref{d31}) the Martin kernel for an irregular point $\xi$
is given by
\begin{equation}
\label{re22}
\kappa(i,\xi)={U_{i\xi}-\int_{\partial_{\infty}} U_{ \eta \xi}
\PP_i \{ X_{\zeta}\in d \eta \} \over w_0-\int_{\partial_{\infty}}
U_{ \eta \xi} \PP_r \{ X_{\zeta}\in d \eta \}}.
\end{equation}

The study of the Martin kernel for regular points needs an extra
work because numerator and denominator vanish. This constitutes
the main object of this section. Next formulae relate the operator
$W$ and the exit measure.

\medskip

\begin{proposition}
\label{p80}
For any $i,j\in I$ we have
\begin{equation}
\label{re36}
\PP_i\{X_{\zeta}\in
\partial_\infty(j)\}=\int_{\partial_{\infty}}U_{i \xi} (W^{-1}
{\bf 1}_{\partial_\infty(j)})(\xi) \mu(d \xi).
\end{equation}
\end{proposition}

\begin{proof}
The function $h_1(i)=\PP_i\{X_{\zeta}\in
\partial_\infty(j)\}$ is harmonic and
bounded. Moreover for any regular point $\eta$
we have
$$
\lim\limits_{i\to \eta} \PP_i\{X_{\zeta}\in
\partial_\infty(j)\}={\bf 1}_{\partial_\infty(j)}(\eta),
$$
which implies that $\lim h_1={\bf 1}_{\partial_\infty(j)}$
$\mu-$a.e..

\medskip

On the other hand, consider $h_2(i):=\int_{\partial_{\infty}}U_{i
\xi} (W^{-1} {\bf 1}_{\partial_\infty(j)})(\xi) \mu(d \xi)$. This
function is also harmonic because for every $\xi\in
\partial_{\infty}$ the function $U_{i \xi}$ is harmonic on $I$.
Let us show that $h_2$ is a bounded function. From (\ref{re32})
one checks that
$||W^{-1} {\bf 1}_{\partial_\infty(j)}||_{\infty} < \infty$.
Then
$$
|h_2(i)|\le ||W^{-1} {\bf 1}_{\partial_\infty(j)}||_{\infty}
\int_{\partial_{\infty}}U_{\eta \xi}\mu(d\xi)= ||W^{-1}
{\bf 1}_{\partial_\infty(j)}||_{\infty} W{\bf 1}(\eta)<\infty ,
$$
where  $\eta$ is any point in $\partial_\infty(i)$. Hence $h_2$ is
bounded. Finally, by the Dominated Convergence Theorem we conclude
the pointwise convergence
$$
\lim\limits_{i\to \eta}h_2(i)=\int_{\partial_{\infty}}U_{\eta \xi}
(W^{-1} {\bf 1}_{\partial_\infty(j)})(\xi) \mu(d\xi).
$$
The result follows from the equality
$\int_{\partial_{\infty}}U_{\eta \xi} (W^{-1}
{\bf 1}_{\partial_\infty(j)})(\xi) \mu(d\xi)
={\bf 1}_{\partial_\infty(j)}(\eta)$
$\mu-$a.e. in $\eta\in \partial_{\infty}$. $\Box$
\end{proof}

\begin{corollary}
\label{c201}
Let $\; h:I\to \RR$ be a harmonic
function such that $\lim h=\varphi\; \mu-$a.e. (for example if $h$ is bounded).
Assume $\varphi$
is a simple function, that is $\varphi \in {\cal D}$  (in
particular $\varphi$ is in the domain of $W^{-1}$). Then for all
$i\in I$
\begin{equation}
h(i)=\int U_{i\xi} (W^{-1}\varphi)(\xi) \mu(d\xi).
\end{equation}
\end{corollary}

\begin{proof}
It is direct from (\ref{re36}) by decomposing
$\varphi$ as a finite linear combination of indicator functions
based on the sets $C^{n_1}(\eta_1),\cdots , \; C^{n_k}(\eta_k)$.
$\Box$
\end{proof}

\medskip

\begin{remark}
Then, in a "dense" class of harmonic functions we have the
representation $h(i)=\int U_{i\xi} d\nu(\xi)$ with $d\nu(\xi)=
W^{-1}\varphi (\xi) \mu(d\xi)$. This representation is similar to
the one using the Martin kernel as in \cite{cartier}.
Nevertheless, there are some differences. Even if $h$ is positive,
$\nu$ may be a signed measure. On the other hand the
characterization $d\nu=W^{-1}\varphi\; d\mu$ gives additional
information on this signed measure. We recall that in the Martin
representation, $\varphi$ is the Radon-Nikodym derivative of the
absolute continuous part of the representing measure with respect
to $\mu$ (see for example \cite{sawyer1997}).
\end{remark}


Recall that a real function $f$ is increasing in the tree, which
we denote by $\preceq-$increasing, if $i\preceq j$ implies
$f(i)\le f(j)$.

\medskip

\begin{theorem}
\label{p302}
A function $h:I\to \RR_+$ is
harmonic and $\preceq$-increasing if and only there exists a
finite (nonnegative) measure $\nu$ in $\partial_{\infty}$ such
that
\begin{equation}
\label{re37}
h(i)=\int_{\partial_{\infty}}U_{i \xi} d\nu(\xi)
\hbox{ for every } i\in I.
\end{equation}

\end{theorem}

\begin{proof}
If $h$ verifies (\ref{re37}) then it is harmonic and increasing
since $U_{\bullet \xi}$ is so. Let us assume now that $h$ is a
nonnegative harmonic and increasing function. From Proposition
\ref{p3} proven for finite matrices we get
$$
\forall n \;\; \exists ! \;\; \alpha^{(n)} : B^n\to \RR
\hbox{ such that if }
|i|\le n: \; h(i)=\sum_{j\in B^n}U_{ij}\alpha^{(n)}(j).
$$
In particular if $|i|=n-1$ we find
$$
h(i^+)=\!\sum_{j\in B^n}\!\!U_{i^+ j}\;\alpha^{(n)}(j)=\!\!\!\!
\sum_{j\in B^n, j\neq i^+}\!\!\!\!U_{i j}\alpha^{(n)}(j)+ U_{i^+
i^+}\alpha^{(n)}(i^+)=h(i)+(U_{i^+ i^+}-U_{i i^+})\alpha^{(n)}(i^+).
$$
Therefore
$$
\alpha^{(n)}(i^+)={{h(i^+)-h(i)}\over {U_{i^+ i^+}-U_{i i^+}}},
$$
and $\alpha^{(n)}$ is a measure in $B^n$. Let us show that these
measures verify the consistence property. We have
$$
\hbox{ for } |i|\le n: h(i)=
\sum_{j\in {\widetilde B}^{n+1}}U_{ij}\alpha^{(n+1)}(j)=
\sum_{k\in B^n}U_{ik}\Big(\sum_{j\in S_k}
\alpha^{(n+1)}(j)\Big) =\sum_{k\in B^n}U_{ik}\alpha^{(n)}(k).
$$
From uniqueness of $\alpha^{(n)}$
we deduce $\alpha^{(n)}(k)=\sum_{j\in S_k}\alpha^{(n+1)}(j)$. Then
the consistence property is verified. The total mass of $\alpha^{(n)}$ is given by
$h(r)=w_0 \sum_{j\in B^n} \alpha^{(n)}(j)$. Then there exists a finite
measure in the boundary such that
$h(i)=\int_{\partial_{\infty}}U_{i \xi} d\nu (\xi)$, for $i\in I$.
$\Box$
\end{proof}

\medskip

\begin{remark}
The measure $\nu$ in the
previous result can be singular with respect to $\mu$. For
example, it is enough to take $\xi$ an inaccessible point and consider
the function $h(i)=U_{i\xi}$, which is represented by the measure
$\nu=\delta_\xi$.
\end{remark}

\medskip

The next result is a representation  as an integral of $U$,
of all harmonic functions that satisfies a certain finite variation condition.

\begin{theorem}
\label{representation}
Assume that $h:I\to \RR$ is
a bounded harmonic function. Then, there exists a finite signed measure $\nu$
such that
\begin{equation}
\label{repres}
h(i)=\int_{\partial_{\infty}}U_{i \xi} d\nu(\xi)
\hbox{ for every } i\in I,
\end{equation}
if and only if the following condition holds
\begin{equation}
\label{finitevariation}
\sup\limits_{n\ge 1} \frac{1}{w_{n}-w_{n-1}}\sum\limits_{j\in B^{n}} |h(j)-h(j^-)| < \infty.
\end{equation}
In particular if this condition holds then $h=h^+-h^-$
is the difference of two increasing nonnegative harmonic functions
$h^+, h^-$ given by the positive and negative part of $\nu$.
\end{theorem}

\begin{proof} Let us first assume that $h$ is strictly positive.
If (\ref{repres}) holds, then
$$
h(i)-h(i^-)=\int (U_{i \xi}-U_{i^- \xi}) d\nu(\xi)=(U_{ii}-U_{i^- i})\nu(\partial_\infty(i)),
$$
from which we obtain
$$
|h(i)-h(i^-)|\le (w_n-w_{n-1})|\nu|(\partial_\infty(i)).
$$
Summing over $B^n$ this inequality yields
$$
\frac{1}{w_{n}-w_{n-1}}\sum\limits_{i\in B^{n}} |h(i)-h(i^-)|\le |\nu|(\partial_\infty)<\infty.
$$
Let us now assume that (\ref{finitevariation}) holds. As in the
proof of Theorem \ref{p302} we have that for all $n$ and all $
i\in I$ such that $|i|\le n$
$$
h(i)=\sum_{j\in B^n}U_{ij}\, \alpha^{(n)}(j),
$$
where
$$
\alpha^{(n)}(j)=\frac{h(j)-h(j^-)}{U_{j j}-U_{j j^-}}=\frac{h(j)-h(j^-)}{w_n-w_{n-1}}.
$$
Let us  define the signed measure $\nu_n$ by $\nu_n(\partial_\infty(j))=\alpha^n(j)$. Then
we obtain that
$\nu_n(\partial_\infty)=h(r)/w_0>0$ and
$$
\sup_n |\nu_n|(\partial_\infty)<\infty.
$$
Therefore, there exists a subsequence
$(\nu_{n_k})$ converging weakly to a finite signed measure $\nu\neq 0$. Moreover,
$\nu(\partial_\infty)=h(r)/w_0$ and since $U_{i\bullet}$ is a bounded continuous
function we get
$$
h(i)=\lim\limits_k \int U_{i\xi} d\nu_{n_k}(\xi)=\int U_{i\xi} d\nu(\xi)
$$

\medskip

For the general case remind that
$\ell(i)=:1-\bar g(i)=\PP_i( X_\zeta \in \partial_\infty)$
is a nonnegative harmonic function. $\ell$ is also an
increasing harmonic function with limit $1$ in the boundary, then
$$
\ell(i)=\int U_{i\xi} (W^{-1} {\bf 1})(\xi) \, d\mu(\xi)=
\frac{\ell(r)}{w_0} \int U_{i\xi} d\mu(\xi)=\int U_{i\xi} d\nu(\xi),
$$
where $\nu$ is the finite measure $\frac{\ell(r)}{w_0} \mu$. Since $\ell(i)\ge \ell(r)>0$ we
can take a large constant $C$ such that the function $\bar h=h+C\ell$ is a
nonnegative bounded harmonic
function. It is direct to check that $h$ satisfies (\ref{finitevariation})
if and only if $\bar h$ satisfies it, from where
the result holds.
$\Box$
\end{proof}

\bigskip

We notice that, since $h$ is harmonic we have
$$
\frac{1}{w_{n}-w_{n-1}} (h(j)-h(j^{-}))=
Q_{j j^-}(h(j)-h(j^-))=\sum_{j^+} Q_{jj^+}(h(j^+)-h(j)).
$$
Then,
$$
\frac{1}{w_{n}-w_{n-1}} |h(jn)-h(j^{-})|\le \sum_{j^+} \frac{1}{w_{n+1}-w_{n}} |h(j^+)-h(j)|,
$$
implying that $\frac{1}{w_{n}-w_{n-1}} \sum_{j \in B^n} |h(j)-h(j^{-})|$ is monotone in n.

\bigskip

Let us give a formula for the Martin kernel in
terms of $U$ and $\mu$. For this reason we first prove the
following result.

\medskip

\begin{proposition}
\label{p90}
For $\eta\in
\partial_{\infty},\; i\in I,\; n\ge 1$ we have

\begin{eqnarray*}
\PP_i\{X_{\zeta}\in C^n(\eta)\}= \mu(C^n(\eta))\Big[&
{U_{i\eta}\over G_{|i\wedge \eta|+1}(\eta)}\;
{\bf 1}_{I\setminus [\eta(n),\infty)}(i)+{1\over G_n(\eta)}
\EE(U_{i\bullet}|{\cal F}_n)(\eta)\;{\bf 1}_{[\eta(n),\infty)}(i) \\
+&\!\!\!\!\!\!\!\sum\limits_{k=0}^{n-1}
\left({1\over G_k(\eta)}-{1\over G_{k+1}(\eta)}\right)
\EE(U_{i\bullet}|{\cal F}_k)(\eta)\;
{\bf 1}_{[\eta(k),\infty)}(i) \Big].
\end{eqnarray*}

In particular, if $\eta$ is accessible and $n>|i\land \eta|$ we get
\begin{equation}
\label{re38}
{\PP_i\{X_{\zeta}\in C^n(\eta)\}\over \mu(C^n(\eta))}=
{U_{i\eta}\over G_{|i\wedge \eta|+1}(\eta)}+
\sum\limits_{k=0}^{|i\wedge \eta|}\left({1\over G_k(\eta)}
-{1\over G_{k+1}(\eta)}\right)
\EE(U_{i\bullet}|{\cal F}_k)(\eta).
\end{equation}
\end{proposition}

\begin{proof}
Let $\eta$ and $n$ be fixed. We denote
$C^k=C^k(\eta)=\partial_\infty(\eta(k))$ and
$A^k=[\eta(k),\infty)$. From (\ref{re36})
we have
$$
h_\eta(i):=\PP_i\{X_{\zeta}\in C^n\}=\int_{\partial_{\infty}}
U_{i \xi} (W^{-1} {\bf 1}_{C^n})(\xi) \mu(d\xi).
$$
Now, let us compute
$$
\rho^k(i):=\int_{\partial_{\infty}}U_{i \xi}
{\bf 1}_{C^k}(\xi) \mu(d\xi).
$$
We examine two different cases. If $i\not\in A^k$ then
$U_{i \xi}=U_{i \; \eta(k)}$ for every $\xi\in C^k$, and so
$\rho^k(i)=U_{i \; \eta(k)}\mu(C^k)$. If $i\in A^k$
then $\rho^k(i)=\sum\limits_{{j\in A^k}\atop{|j|=|i|}}
U_{i j}\; \mu(\partial_\infty(j))$. We
summarize these relations in
\begin{equation}
\rho^k(i)=U_{i \eta(k)}\mu(C^k)\; {\bf 1}_{I\setminus A^k}(i)+
\sum\limits_{{j\in A^n}\atop{|j|=|i|}} U_{i j}\;
\mu(\partial_\infty(j))\; {\bf 1}_{A^k}(i).
\end{equation}
Now we use (\ref{re34}) to get
\begin{eqnarray*}
&\int U_{i\xi} (W^{-1} {\bf 1}_{C^n})(\xi) \mu(d\xi)=
{1\over G_n}\left[ U_{i\eta(n)} \mu(C^n)\; {\bf 1}_{I\setminus
A^n}(i)+\!\!\!\sum\limits_{{j\in A^n}\atop{|j|=|i|}}\!\!\! U_{i j}\;
\mu(\partial_\infty(j))\; {\bf 1}_{A^n}(i)\right]\\
&\;\;+\sum\limits_{k=0}^{n-1} \left({1\over G_k}-
{1\over G_{k+1}}\right)
{\mu(C^n)\over \mu(C^k)}\left[ U_{i\eta(k)} \mu(C^k)\;
{\bf 1}_{I\setminus A^k}(i)+
\sum\limits_{{j\in A^k}\atop{|j|=|i|}} U_{i j}\;
\mu(\partial_\infty(j))\; {\bf 1}_{A^k}(i)\right].
\end{eqnarray*}
From
$$
\EE(U_{i\bullet}|{\cal F}_n)(\eta)=
\begin{cases}
\; U_{i \eta(k)}  &\hbox{ if } i\notin A^k \\
\sum\limits_{{j\in A^n}\atop{|j|=|i|}} U_{ij} \mu(\partial_\infty(j))
&\hbox{ if } i\in A^k\;,
\end{cases}
$$
to get
\begin{eqnarray*}
&{}& \int U_{i\xi} (W^{-1} {\bf 1}_{C^n})(\xi) \mu(d\xi)\\
&{}&\; = \mu(C^n)\left[
\sum\limits_{k=0}^{n-1}\left({1\over G_k}-{1\over G_{k+1}}\right)
\EE(U_{i\bullet}|{\cal F}_k)(\eta)\; {\bf 1}_{A^k}(i)+
{1\over G_n} \EE(U_{i\bullet}|{\cal F}_n)(\eta)\;
{\bf 1}_{A^n}(i)\right]\\
&{}& \; + \mu(C^n)\left[\sum\limits_{k=0}^{n-1}\left({1\over G_k}-
{1\over G_{k+1}}\right) U_{i\eta(k)} \;
{\bf 1}_{I\setminus A^k}(i)+ {1\over G_n}
U_{i\; \eta(n)}\; {\bf 1}_{I\setminus A^n}(i)\right].
\end{eqnarray*}
Now $i\in I\setminus A^k$ implies $k> |i\wedge\eta|$.
Since for $k\ge |i\wedge\eta|$ we have
$U_{i\eta(k)}=U_{i\eta}$, we can
simplify the last term in the previous equation to
$$
{\mu(C^n)\over G_{|i\wedge \eta|+1}}U_{i\eta}\;
{\bf 1}_{I\setminus A^n}(i).
$$
Then we get
\begin{eqnarray*}
\PP_i\{X_{\zeta}\in C^n(\eta)\}=\mu(C^n)\Big[&
{U_{i\eta}\over G_{|i\wedge \eta|+1}} \; {\bf 1}_{I\setminus A^n}(i)
+{1\over G_n} \EE(U_{i\bullet}|{\cal F}_n)(\eta)\; {\bf 1}_{A^n}(i)+\\
&\sum\limits_{k=0}^{n-1}
\left({1\over G_k(\eta)}-{1\over G_{k+1}(\eta)}\right)
\EE(U_{i\bullet}|{\cal F}_k)(\eta)\; {\bf 1}_{A^k}(i)\Big].
\end{eqnarray*}
$\Box$
\end{proof}

\begin{theorem}
\label{p91}
Let $i\in I$ and $\eta$ be an accessible point. Then
the Martin kernel has the representation
\begin{equation}
\label{re39}
\kappa(i,\eta)={1\over \PP_r\{X_{\zeta}\in
\partial_\infty \}} \sum\limits_{k=0}^{|i\wedge \eta|+1}
{1\over G_k(\eta)}
\big(\EE(U_{i\bullet}
|{\cal F}_k)(\eta)-\EE(U_{i\bullet}|{\cal F}_{k-1})(\eta)\big),
\end{equation}
where by convention $\EE(\; |{\cal F}_{-1})=0$.
\end{theorem}

\begin{proof}
We use Proposition \ref{p90} and the equality
$\frac{U_{r\eta}}{G_0}=\frac{w_0}{G_0}=
\PP_r\{X_{\zeta} \in \partial_\infty\}$ to
get
$$
\kappa(i,\eta)={1\over \PP_r\{X_{\zeta}\in \partial_\infty \}}
\left[{U_{i\eta}\over G_{|i\wedge \eta|+1}(\eta)}+
\sum\limits_{k=0}^{|i\wedge \eta|}
\left({1\over G_k(\eta)}-{1\over G_{k+1}(\eta)}\right)
\EE(U_{i\bullet}|{\cal F}_k)(\eta)\right],
$$
and the result follows. $\Box$
\end{proof}

\medskip

\begin{corollary}
\label{conuevo}
For $i\in I$ fixed, the Martin kernel $\kappa(i,\bullet)$
is the image of $U_{i\bullet}$ by an stochastic integral operator,
in fact
\begin{equation*}
\kappa(i,\eta)=\sum\limits_{k=0}^{\infty}
{\widetilde G}_k^{(i)}(\eta) \big(\EE(U_{i\bullet}
|{\cal F}_k)-\EE(U_{i\bullet}|{\cal F}_{k-1})\big)(\eta),
\end{equation*}
where ${\widetilde G}^{(i)}=({\widetilde G}^{(i)}_k: k\in \NN)$
is a ${\cal F}-$predictable process.
\end{corollary}

\begin{proof}
It suffices to take
${\widetilde G}^{(i)}_k={\bf 1}_{D^{(i)}_k}
{\PP_r\{X_{\zeta}\in \partial_\infty\}}^{-1}{G_k}^{-1}$,
where
$D^{(i)}_k=\{\xi\in \partial_\infty: \xi\land i\ge k-1\}$
is a ${\cal F}_{k-1}-$measurable set. $\Box$
\end{proof}

\medskip

\medskip

Let us revisit the Martin kernel for an irregular point $\eta$.
From (\ref{rn111}), if $\xi \in \partial_\infty^{reg}$ the kernel
$\kappa(i,\xi)$ is the Radon-Nykodim derivative of $\PP_i\{X_\zeta
\in d\xi\}$ with respect to $\PP_r\{X_\zeta \in d\xi\}$. Therefore
if $\eta$ is an accessible irregular point we obtain from
(\ref{re22})
$$
\kappa(i,\eta)= {U_{i\eta}-\int_{\partial_{\infty}^{reg}} U_{ \eta
\xi} \PP_i \{ X_{\zeta}\in d \xi \} \over
w_0-\int_{\partial_{\infty}} U_{ \eta \xi} \PP_r \{ X_{\zeta}\in d
\xi \}}= {U_{i\eta}-\int_{\partial_{\infty}^{reg}} U_{ \eta \xi}
\; \kappa(i,\xi) \PP_r \{ X_{\zeta}\in d \xi \} \over
w_0-\int_{\partial_{\infty}} U_{ \eta \xi} \PP_r \{ X_{\zeta}\in
d\xi \}}.
$$

\medskip

\section{Trees Potential without Absorption}
\label{secc7}

\medskip

\subsection{Reflecting at the root}
\label{subsecc71}

\medskip

Let $(I,{\cal T})$ be a tree rooted at $r$.
In this section we consider the case when $r$ is a
reflecting barrier.
As before we take a strictly positive and
strictly increasing sequence
$(w_n: n\in {\bf N})$ and consider a symmetric $q-$matrix $Q$ on
$I\times I$,
supported on the tree and the diagonal,
defined as in (\ref{re200}) except at the pair
$(r,r)$, where
$Q_{rr}=-{|S_r|\over w_1-w_0}$.
It is direct to check that $Q$ is conservative:
$\sum\limits_{j\in I} Q_{ij}=0$ for every $i\in I$.
We assume the Markov process $(X_t)$ associated to $Q$ is transient, that is
$\PP_r\{X_{\zeta}\in \partial_\infty\}=1$, and that
all points in $\partial_\infty$ are regular.

\medskip

The aim is to obtain
a representation of the potential $V$ for this process as well as
for the Martin kernel, in terms of the tree matrix
$U=(U_{ij}=w_{|i\wedge j|}: i,j\in I)$.
For this purpose, consider the translated matrix
$U^{(a)}:=U+a$, for $a>0$, which is the tree matrix associated to the
level function $w^{(a)}_n=w_n+a$. Define the matrix $Q^{(a)}$ on
$I\times I$ as in (\ref{re200}) with respect to this level function.
At $(r,r)$ it takes the value
$Q^{(a)}_{rr}=Q_{rr}-{1\over w_0+a}$.
We also put $Q^{(a)}_{r\partial_r}={1\over w_0+a}$. We notice that
the matrices $Q^{(a)}$ and $Q$ in $I\times I$, only differ at $(r,r)$.

\medskip

As $a$ tends to infinity,
$Q^{(a)}$ converges to $Q$, and the associated processes also
converge. In fact, a coupling argument allows us to construct an
increasing sequence of stopping times $\; T^{(a)}
\mathop{\uparrow}\limits_{a\to \infty} \infty\; $ such that
$$
X^{(a)}_t=X_t \hbox{ if } t<T^{(a)}
\hbox{ and } X^{(a)}_t= \partial_r \hbox{ if } t\ge
T^{(a)},
$$
is a Markov process with generator $Q^{(a)}$.
Notice that the lifetime variables $\zeta^{(a)}$
and $\zeta$ associated respectively to $X$ and $X^{(a)}$, verify
$\zeta^{(a)}=\zeta\wedge T^{(a)}$.
From this
representation it also follows immediately that
the potentials $V^{(a)}$ and $V$,
associated to $Q^{(a)}$ and $Q$, respectively, verify
$\forall i,j,$ $V^{(a)}_{ij} \mathop{\uparrow}\limits_{a\to
\infty} V_{ij}$. Therefore, the
representation (\ref{re21}) reads as follows
$$
U^{(a)}_{ij}-V^{(a)}_{ij}=\int_{\partial_\infty} U^{(a)}_{\eta j}
\; \PP_i\{X_{\zeta}\in d\eta, \zeta \le T^{(a)}\},
$$
or equivalently
$$
U_{ij}-V^{(a)}_{ij}=\int_{\partial_\infty} U_{\eta j} \;
\PP_i\{X_{\zeta}\in d\eta, \zeta \le T^{(a)}\}-a\PP_i\{T^{(a)}<
\zeta\}.
$$
Passing to the limit $a\to \infty$ we obtain that
$\lim\limits_{a \to \infty} a\PP_i\{T^{(a)}< \zeta\}$
exits and moreover
$$
U_{ij}-V_{ij}=\int_{\partial_\infty} U_{\eta j} \;
\PP_i\{X_{\zeta}\in d\eta \}- \lim\limits_{a \to \infty}
a\PP_i\{T^{(a)}< \zeta\}.
$$
Substituting $j$ by $r$ in the last equality and using that
$U_{ir}=U_{\eta r}=w_0$, we find $\lim\limits_{a \to
\infty} a\PP_i\{T^{(a)}< \zeta\}=V_{ir}$, and therefore we get
\begin{equation}
\label{re605}
U_{ij}-V_{ij}=\int_{\partial_\infty} U_{\eta j} \;
\PP_i\{X_{\zeta}\in d\eta \}-V_{ir}.
\end{equation}
Now, if we take $j\to \xi \in \partial_\infty^{reg}$ we obtain
$$
V_{ir}=\int_{\partial_\infty} U_{\eta \xi} \;
\PP_i\{X_{\zeta}\in d\eta \}-U_{i\xi}=
\int_{\partial_\infty} (U_{\eta \xi} -U_{i\xi})\;
\PP_i\{X_{\zeta}\in d\eta \}.
$$
Thus, we have proven that the following equality holds
\begin{equation}
\label{re61}
U_{ij}-V_{ij}=\int_{\partial_\infty}
(U_{\eta j}+U_{i\xi}-U_{\eta \xi}) \;
\PP_i\{X_{\zeta}\in d\eta \},
\end{equation}
which is independent of $\xi \in \partial_\infty^{reg}$.
Integrating (\ref{re61}) with respect to $\PP_j\{X_{\zeta} \in
d\xi\}$ gives
\begin{eqnarray*}
\label{eq12}
U_{ij}-V_{ij}=&\int_{\partial_\infty} U_{\eta j}\;
\PP_i\{X_{\zeta}\in d\eta \}+ \int_{\partial_\infty} U_{i\xi}\;
\PP_j\{X_{\zeta} \in d\xi\}-\\
&\int_{\partial_\infty}
\int_{\partial_\infty}U_{\eta \xi}\; \PP_j\{X_{\zeta} \in d\xi\}
\PP_i\{X_{\zeta} \in d\eta\}.
\end{eqnarray*}

The Martin kernel $\kappa^{(a)}$ associated to $Q^{(a)}$ can be
computed as in Theorem \ref{p91}. Take $i\in I$, $\eta\in
\partial_{\infty}$ and $n>|i\wedge \eta|$ then
$$
\kappa^{(a)}(i,\eta)={\PP_i\{X_{\zeta}\in C^n(\eta), \zeta\le
T^{(a)}\}\over \PP_r\{X_{\zeta}\in C^n(\eta), \zeta\le T^{(a)}\}}.
$$
Therefore, there is also continuity of the Martin kernel with
respect to $a$. Passing to the limit $a\to \infty$ and using the
representation (\ref{re39}) we obtain
\begin{equation}
\kappa(i,\eta)=\lim\limits_{a\to \infty}
\sum\limits_{k=0}^{|i\wedge \eta|+1} {1\over G^{(a)}_k}
\big(\EE_{\mu^{(a)}}(U^{(a)}_{i\bullet}|{\cal
F}_k)(\eta)-\EE_{\mu^{(a)}}(U^{(a)}_{i\bullet}|
{\cal F}_{k-1})(\eta)\big),
\end{equation}
where
$$
G^{(a)}_k=G^{(a)}_k(\eta)=\sum\limits_{n\ge k}
(w^{(a)}_n-w^{(a)}_{n-1}) \; \mu^{(a)} (C^n(\eta)),
$$
\begin{equation}
\label{mudea}
\mu^{(a)}(\bullet)={\PP_r\{X_{\zeta}\in \bullet, \zeta \le
T^{(a)}\}\over \PP_r\{\zeta \le T^{(a)}\}}.
\end{equation}
We notice
$$
G^{(a)}_0={w_0+a\over \PP_r\{\zeta \le T^{(a)}\}},
$$
and
\begin{eqnarray*}
G^{(a)}_k&=& G^{(a)}_0-\sum\limits_{n=0}^{k-1}
(w^{(a)}_n-w^{(a)}_{n-1}) \; \mu^{(a)} (C^n(\eta))\\
&=&
{w_0+a\over \PP_r\{\zeta \le T^{(a)}\}}
-(w_0+a)-\sum\limits_{n=1}^{k-1} (w_n-w_{n-1}) \;
\mu^{(a)}(C^n(\eta))\\
&=&(w_0+a){\PP_r\{T^{(a)}< \zeta \} \over
\PP_r\{\zeta \le T^{(a)}\}}-\sum\limits_{n=1}^{k-1} (w_n-w_{n-1})
\; \mu^{(a)}(C^n(\eta)).
\end{eqnarray*}
Therefore, the previous computations show the following result.

\medskip

\begin{theorem}
\label{th701} Let $\mu(\bullet)=\PP_r\{X_{\zeta}\in \bullet \}$.
Consider $G_0(\eta) := \int U_{\eta \xi} \PP_r\{X_{\zeta}\in d\xi
\}$ and $G_k :=\lim\limits_{a \to \infty} G^{(a)}_k$. Then
$G_0(\eta) = V_{rr}+w_0$ is a constant and $(G_k: k\ge 1)$ is a
positive decreasing predictable process that verifies
$$
G_k(\eta)=G_0-\sum\limits_{n=0}^{k-1}
\Delta_n(w)\; \mu(C^n(\eta))=\sum\limits_{n\ge k}
\Delta_n(w) \; \mu(C^n(\eta)) \hbox{ for } k\ge 1;
$$
and the following representation holds
\begin{equation}
\label{re64}
\kappa(i,\eta)=
1+ \sum\limits_{k=1}^{|i\wedge \eta|+1} {1\over G_k(\eta)}
\big(\EE_{\mu}(U_{i\bullet}|{\cal F}_k)(\eta)-
\EE_{\mu}(U_{i\bullet}|{\cal F}_{k-1})(\eta)\big).
\end{equation}
\end{theorem}

\medskip

\begin{remark}
It can be shown that $\mu^{(a)}$ defined in
(\ref{mudea}) does not depend on $a\ge 0$
(recall that $\mu^{(0)}$ is the measure
defined in (\ref{muab}) in subsection \ref{subs11}
for the chain absorbed at $\partial_r$). Indeed this follows
from the independence relation
$$
\PP_r\{X_{\zeta}\in \bullet, \zeta \le T^{(a)}\}=
\PP_r\{X_{\zeta}\in \bullet\}
\PP_r\{\zeta \le T^{(a)}\},
$$
then
$$
\mu^{(a)}=\mu \hbox{ for } a\ge 0, \;\; \hbox{ where }
\mu(\bullet)=\PP_r\{X_{\zeta}\in \bullet \} .
$$
Further, if $N^*_r$ is the number of visits in the strict future
to $r$ of the discrete skeleton of $(X_t)$, then
a simple argument shows that
$\mu(\bullet)=\PP_r\{X_{\zeta}\in \bullet | N^*_r=0\}$.
\end{remark}

\medskip

\begin{remark}
\label{ultima2}
If we take
$$
W^{(a)}=
\sum_{n\in \NN} G_n^{(a)}
\big(\EE_{\mu^{(a)}}(\; | {\cal F}_n)-\EE_{\mu^{(a)}}
(\; | {\cal F}_{n-1})\big),
$$
then
$$
\lim\limits_{a\to \infty}(W^{(a)})^{-1}=
\sum_{n\ge 1} G_n^{-1} \big(\EE_{\mu}(\; | {\cal F}_n)
-\EE_{\mu}(\; | {\cal F}_{n-1})\big).
$$
coincides with the operator
${\underline {\bf W}}^{-1}:=W^{-1}-G_0^{-1}\EE_{\mu}$
defined in (\ref{Ma11}) in Remark \ref{ultima}, it verifies
$\lim\limits_{a\to \infty}(W^{(a)})^{-1}{\bf 1}=0$
and it is the generator of a Markov process defined in the boundary
$\partial_\infty$ that will be studied in section \ref{section600}.
\end{remark}

\medskip

\subsection{Potential for Homogeneous Trees}
\label{subsecc72}

\medskip

In this section we consider standard random walk on a homogeneous
tree of degree $p+1\!\ge \!3$ and we show that in this case the
previous calculations give a close form to the Martin kernel. Some
of these computations are well known, see for instance
\cite{sawyer1997}. We assume $\cal T$ is an infinite rooted tree,
with $|S_r|=p+1$ and $|S_i|=p$ for $i\neq r$. As a weight function
we take $w_n=n+1$. Finally, we assume that $r$ is reflecting. In
this way we have
$$
Q_{ii^+}=1, \; Q_{ii}=-(p+1) \hbox{ for }i\in I
\hbox{ and } Q_{ii^-}=1 \hbox{ for }i\neq r.
$$
It is well known that this
tree matrix is transient for all $p\ge 2$.

\medskip

From symmetry considerations $\mu$ is the uniform measure on
$\partial_\infty$ and it is easy to see that all points in
$\partial_\infty$ are regular.
Let us know compute the quantities involved on (\ref{re64}).

\medskip

We fix $i\in I$, $\eta\in \partial_\infty$ and put $n=|i\wedge \eta|$,
$|i|=m$. We assume $m \ge 1$ because for $m=0$ we have $i=r$ and
$\kappa(r,\eta)=1$.
We set $C^k=C^k(\eta)=[\eta(k),\infty]\cap \partial_\infty$.
Therefore, $\mu(C^k)=((p+1) p^{k-1})^{-1}$ for all
$k\ge 1$, and $\mu(C^0)=1$. Then,
$$
G_k(\eta)=\sum\limits_{l\ge k} (w_l-w_{l-1}) \; \mu(C^l(\eta))=
\sum\limits_{l\ge k} {1\over (p+1)p^{l-1}}=
{1\over (p^2-1)p^{k-2}}\;\hbox{ for } k\ge 1.
$$
We need to compute
$\EE_{\mu}(U_{i\bullet}|{\cal F}_k)(\eta)$ when $k\le n+1$.
By definition we have
$$
\EE_{\mu}(U_{i\bullet}|{\cal F}_k)(\eta)\!=\!
{1\over \mu(C^k)}\int_{C^k} \!\!U_{i\xi} \mu(d\xi)=
\begin{cases}
{1\over (p+1)p^{m-1}}
\sum\limits_{{j\in I}\atop{|j|=m}} U_{ij}
&\hbox{ if }k=0;\\
{(p+1)p^{k-1}\over (p+1)p^{m-1}}
\sum\limits_{{j\in C^k}\atop{|j|=m}} U_{ij}
&\hbox{ if } 1\le k\le n+1.
\end{cases}
$$

If $k=n+1$ this gives
$\EE_{\mu}(U_{i\bullet}|{\cal F}_{n+1})(\eta)=n+1$.
When $0\le k \le n$, the values of $\{U_{ij}: j\in C^k, |j|=m\}$
range from $k+1$ to $m+1$. For a given integer $t$ in this range
denote by $M^k_t$ the number of sites $j$ for which $U_{ij}=t$.
We have $M^k_{m+1}=1$ and
$$
M^k_m\!=\!p\!-\!1,\;
M^k_{m-1}\!=\!(p\!-\!1)p,..,M^k_t\!=\!(p\!-\!1)p^{m-t},..,
M^k_{k+1}\!=\!(p\!-\!1)p^{m-(k+1)} \hbox{ for }k\ge 1;
$$
$$
M^0_m=p-1,\;
M^0_{m-1}\!=\!(p-1)p,..,M^0_t\!=\!(p-1)p^{m-t},...,
M^0_2\!=\!(p-1)p^{m-2},\;M^0_{1}\!=\!p^m \hbox{ for }k=0.
$$

From these expressions we obtain
\begin{eqnarray*}
\EE_{\mu}(U_{i\bullet}|{\cal F}_0)(\eta)&=&
{1\over (p+1)p^{m-1}} \left(m+1+p^m+(p-1)\sum\limits_{t=2}^m
tp^{m-t}\right), \\
\EE_{\mu}(U_{i\bullet}|{\cal F}_k)(\eta)&=& p^
{k-m}\left(m+1+(p-1)\sum\limits_{t=k+1}^m tp^{m-t}\right)
\hbox{ for }1\le k\le n.
\end{eqnarray*}
Hence we get
\begin{eqnarray*}
&{}& \EE_{\mu}(U_{i\bullet}|{\cal F}_1)(\eta)-
\EE_{\mu}(U_{i\bullet}|{\cal F}_{0})(\eta)=
p^2{1-p^{-m}\over p^2-1} \\
&{}& \EE_{\mu}(U_{i\bullet}|{\cal F}_k)(\eta)-
\EE_{\mu}(U_{i\bullet}|{\cal F}_{k-1})(\eta)= 1-p^{k-m-1} \;
\hbox{ for }2\le k\le n;
\\
&{}& \EE_{\mu}(U_{i\bullet}|{\cal F}_{n+1})(\eta)-
\EE_{\mu}(U_{i\bullet}|{\cal F}_n)(\eta)={p^{n-m}-1 \over p-1}.
\end{eqnarray*}

Finally we obtain (see for example \cite{sawyer1997} Theorem 8.1)
$$
\kappa(i,\eta)=p^{2n-m},\; \hbox{ where } m=|i|, n=|i\land \eta|.
$$

In particular if
$|i\wedge \eta|=0$ we get for $k\ge 1$
$$
p^{-m}=\kappa(i,\eta)={\PP_i\{X_{\zeta}\in C^k(\eta)\}\over
\PP_r\{X_{\zeta}\in C^k(\eta)\}}=\PP_i\{T_r<\infty\}.
$$
In a similar way we obtain
$\PP_i\{T_{i-}<\infty\}=p^{-1}$.
From (\ref{re605}) we have
$$
V_{rj}-V_{rr}=V_{rr}(\PP_j\{T_r<\infty\}-1)=1-\int
U_{j\eta}\PP_r\{X_{\zeta}\in d\eta\},
$$
where for the last integral we assume $|j|=m\ge 1$, and we get
$$
\int U_{j\eta}\PP_r\{X_{\zeta}\in d\eta\}={1\over (p+1)p^{m-1}}
\sum\limits_{t=1}^{m+1} t M^0_t.
$$
From this expression we find
$$
V_{rr}={p\over (p+1)(p-1)} \; \hbox{ and } V_{jr}={p^{1-m}\over
(p+1)(p-1)}.
$$
A simple argument based on time reversal shows that
$\PP_r\{T_j<\infty\}=\PP_j\{T_r<\infty\}=p^{-|j|}$,
and in general
$$
\PP_i\{T_j<\infty\}=p^{-|geod(i,j)|}.
$$
Since $V_{jr}=\PP_r\{T_j<\infty\}V_{jj}$ we deduce that
$$
V_{jj}={p\over (p+1)(p-1)},
$$
which can be also obtained from the invariance of the tree under
translations. Using the same argument,  if $|geod(i,j)|=m=|k|$ we
have
$$
V_{ij}=V_{kr}={p^{1-m}\over (p+1)(p-1)}={p^{1-|geod(i,j)|}\over
(p+1)(p-1)}.
$$
Finally, from (\ref{re605}) we get that
$$
\int U_{j\eta}\; \PP_i\{X_{\zeta}\in d\eta\}=|i\wedge
j|+1+p{p^{-|i|}-p^{-|geod(i,j)|}\over (p+1)(p-1)}.
$$

\medskip

\section{Ultrametricity}
\label{ul5}

There is a wide literature concerning ultrametricty, but it is not
a common notion in potential theory. So, we supply some basic
properties following from the ultrametric inequality (in our
notation, the ultrametric inequality is the one verified by $1/d$,
being $d$ an ultrametric distance). The core of this section are
subsections \ref{genultr} and \ref{boundultr}, where the Markov
semigroup and the harmonic functions emerging from the ultrametric
matrix, in terms of the minimal tree matrix extension are
constructed.

\medskip

\subsection{Basic Notions and the Minimal Rooted Tree Extension}

We impose conditions in order that an ultrametric matrix can be
immersed in a countable and locally finite tree. It is known that
a tree structure is behind an ultrametric (for a deep study of
this relation see \cite{hughes}), but we prefer here to give an
explicit construction because it allows a better understanding of
the main results of this section.

We note that up to Lemma \ref{le101} the set $I$ will have no
restriction. Most of the properties we present are easily deduced
from the ultrametric inequality, so they are established without a
proof.

\medskip

\begin{definition}
\label{d60}
$U=(U_{ij}: i,j\in I)$ is an ultrametric
arrangement if its is symmetric, that is $U_{ij}=U_{ji}$ for any couple
$i,j\in I$, and verifies the ultrametric inequality
$$
U_{ij}\ge \min\{U_{ik}, U_{kj}\} \hbox{ for any } i,j,k \in I\;.
$$
\end{definition}

In particular $U_{ii}\ge U_{ij}$ for any $i,j \in I$, so
$U_{ij}=U_{ii}\Rightarrow U_{jj}\ge U_{ii}$.

\medskip

Observe that for any triple $i_1, i_2, i_3\in I$ there exists a
permutation $\varphi$ of $\{1,2,3\}$ such that
$$
U_{i_{\varphi(1)}i_{\varphi(2)}}=
\min\{U_{i_{\varphi(2)}i_{\varphi(3)}},
U_{i_{\varphi(3)}i_{\varphi(1)}}\}.
$$
Hence, $U_{ik}> U_{kj}\Rightarrow U_{ik}> U_{kj}=U_{ij}$ and
$U_{ik}=U_{kj}\Rightarrow U_{ij}\ge U_{ik}=U_{kj}$.

\medskip

Let us introduce the equivalence relation
$$
i\sim j \Leftrightarrow \left(\forall k\in I:\;\; U_{ik}= U_{jk} \right).
$$
Notice that
$i\sim j \Leftrightarrow  U_{ii}=U_{ij}=U_{jj}$.
%
%
Let us introduce the relation
\begin{equation}
\label{orden2}
i\preceq j \Leftrightarrow  U_{ij}= U_{ii}.
\end{equation}
From
$U_{jk}\ge \min\{U_{ji}, U_{ik}\}=\min\{U_{ii}, U_{ik}\}=U_{ik}$.
we get
$$
i\preceq j \Leftrightarrow U_{i \bullet}\le U_{j \bullet}\;
(\hbox{ that is }\forall k\in I: U_{ik}\le U_{jk}),
$$
so the relation $\preceq$ is a preorder, that is it is
reflexive and transitive.
The equivalence relation associated to the preorder $\preceq$ is
$\sim$, this means
$\big[i\preceq j \hbox{ and } j\preceq i\big]\Leftrightarrow i\sim j$.
On the other hand $i\preceq j  \Rightarrow U_{ii}\le U_{jj}$.

\medskip

Now, we denote  $i{\cal G} j$ if $i,j$ are comparable, that is
$i\preceq j$ or $j\preceq i$.
We have
$i{\cal G} j \Leftrightarrow  U_{ij}\ge \min\{U_{ii}, U_{jj}\}$.
From definition we also get
$i\sim j \Leftrightarrow
\big[U_{ii}=U_{jj} \hbox{ and } i{\cal G} j\big]$.
The left and the right intervals defined by $i\in I$ are
respectively
$$
[i,\infty)^U=\{j\in I: i\preceq j\} \;\hbox{ and }
(-\infty, \;i]^U=\{j\in I: j\preceq i\}\;.
$$
Notice that for any $i\in I$ the set $(-\infty,\; i]^U$
is $\preceq-$totally preordered. This means that for
$$
\forall  j, k\in (-\infty,\; i]^U \hbox{ we have } j{\cal G} k.
$$
%

Some elementary properties deduced from the ultrametric
hypothesis are summarized below, they are easily proven by
analysis of cases. The first
two relations reveal a hierarchical structure.

\medskip


\noindent (i) $\big[i\!\!\not\!{\cal G} j,\; k\in [i,\infty)^U,\;
\ell \in [j,\infty)^U\big]$ implies $k\!\!\not\!{\cal G} \ell$.

\medskip

\noindent (ii) If $i\!\!\not\!{\cal G} j$ then
$[i,\infty)^U\cap [j,\infty)^U=\emptyset$.

\medskip

\noindent (iii) $i\preceq j$ and $k\preceq \ell$ imply
$U_{j \ell}\ge U_{ik}$.

\medskip

\noindent (iv) $i\preceq j$ implies that  for any $k\in I$ it
holds $\left(i\preceq k \hbox{ or } U_{jk}=U_{ik}\right)$.
\medskip

\noindent (v) $i\preceq j$ and $k\preceq \ell$ imply
$(i{\cal G}k \hbox{ or } U_{j \ell}=U_{ik})$.

\medskip

In the sequel we will assume the following condition holds
$$
i\sim j \Leftrightarrow i=j\;.  \eqno(H1)
$$
Property $(H1)$
is equivalent to the fact that
$\preceq$ is an order, or equivalently to the relation
$i\neq j \Rightarrow U_{ij}<\max\{U_{ii}, U_{jj}\}$.
We point out that if $I$ is finite and $U>0$ condition
($H1$) is equivalent to the nonsingularity of $U$ (see \cite{dell1996},
\cite{mcdonald1995} or \cite{nabben1995}).

\medskip

We denote ${\cal W}=\{U_{ij}: i,j\in I\}$ the set of values of $U$.
To every $w\in {\cal W}$ we associate the nonempty set
$J(w)=\{i\in I: U_{ii}\ge w\}$ and the relation
$$
i\equiv_w j \Leftrightarrow U_{ij}\ge w.
$$
The ultrametric inequality implies that $\equiv_w$ is an
equivalence relation in $J(w)$. By $E^w$ we mean an equivalence
class of $\equiv_w$, and $E^w_i$ denotes the equivalence
class containing
$i\in J(w)$. In the case $U_{ii}< w$,
that is $i\notin J(w)$, we put
$E^w_i=\phi$. As usual $J(w)/\equiv_w$ denotes the set of
equivalence classes of elements of $J(w)$.

\medskip

Let us introduce the following set
$$
\widetilde I=\{(E^w,w): E^w \in J(w)/\equiv_w,\; w\in W\}.
$$
The function
\begin{equation*}
{\bf i}^U: I\to \widetilde I, \;
{\bf i}^U(i)=(E^{U_{ii}}_i, U_{ii})
\end{equation*}
is one-to-one. In fact, if  ${\bf i}^U(i)={\bf i}^U(j)$,
then $U_{ij}\ge U_{ii}=U_{jj}$. From condition $(H1)$ we deduce
$i=j$. In this way we identify $i\in I$ with
${\bf i}^U(i)=(E^{U_{ii}}_i,U_{ii})\in \widetilde I$.

\medskip

Observe that
$E^{U_{ii}}_i=[i,\infty)^U$, for every  $i\in I$. Also it holds
\begin{equation*}
\left[w\le w' \Rightarrow E^{w'}\subseteq E^w\right]
\hbox{ and }
\left[\left(w\le w', E^{w'}\neq E^w\right)
\Rightarrow w< w' \right].
\end{equation*}

\begin{lemma}
\label{l55}
If $E^{w'}\not\subseteq E^w$ and
$E^w\not\subseteq E^{w'}$, then
$$
\forall k, k'\in E^w,\; \forall \ell, {\ell}'\in E^{w'}:\;\;
U_{k \ell}=U_{k' {\ell}'}<\min\{w,w'\} \hbox{ and }
E^w\cap E^{w'}=\phi.
$$
\end{lemma}

\medskip

\begin{proof}
Let $k \in E^w\setminus E^{w'}$ and $l \in
E^{w'}\setminus E^w$. Also take $k'\in E^w$, $l'\in E^{w'}$. Since
$\equiv_w, \equiv_{w'}$ are equivalent relations on their
respective domains, we get $U_{kl'}<w'$ and $U_{k'l}<w$.
In particular $U_{kl}<\min\{w,w'\}$. On the other hand
the definition of $E^w$ implies $U_{kk'}\ge w$.
Using the ultrametric property we get
$U_{k'l}\ge \min\{U_{k'k},U_{kl}\}=U_{kl}$,
and similarly $U_{kl}\ge U_{k'l}$, from which the equality
$U_{kl}= U_{k'l}$ holds. In an analogous way it is deduced the
equality $U_{kl}= U_{kl'}$, and we get that
$$
k' \in E^w\setminus E^{w'} \hbox{ and } l' \in E^{w'}\setminus
E^w.
$$
This implies that $E^w\cap E^{w'}=\emptyset$. Again using the
ultrametricity we find
$$
U_{k'l'}\ge \min\{U_{k'l},U_{ll'}\}=U_{k'l}=U_{kl}.
$$
By exchanging the roles of $k$ with $k'$ and $l$ with $l'$, we
deduce the result. $\Box$
\end{proof}

\medskip

The previous result implies that two classes $E^w$ and $E^{w'}$
are either disjoint or one is included in the other. Now we define
$\widetilde U$, an extension of $U$ to $\widetilde I$.

\medskip

\begin{definition}
\label{d77}
Let $\tilde {\imath}=(E^w,w)\in \widetilde I$,
$\tilde {\jmath}=(E^{w'},w')\in \widetilde I$. If $E^{w'}\subseteq E^w$
or $E^w\subseteq E^{w'}$ we put
${\tilde U}_{\tilde {\imath} \tilde {\jmath}} =\min\{w, w'\}$.
On the contrary, that is $E^w \cap E^{w'}=\phi$, we put
${\tilde U}_{\tilde {\imath} \tilde {\jmath}}=U_{k \ell}$,
where $k\in E^w$ and $\ell\in E^{w'}$.
\end{definition}

\medskip

From Lemma \ref{l55}, $\widetilde U$ is well defined. On the other hand
it is direct to prove that for any $i,j \in I$ it holds
$U_{ij}={\widetilde U}_{{\bf i}^U(i) \; {\bf i}^U(j)}$. In this way
$\widetilde U$ is an extension of $U$. Also, if $i\in E^w, j\in E^{w'}$
then $U_{ij}\ge {\widetilde U}_{\alpha \beta}$, where
$\alpha=(E^w,w),\; \beta=(E^{w'},w')$.

\medskip

\begin{lemma}
\label{le101}
$\widetilde U=(\widetilde U_{\tilde {\imath}\, \tilde {\jmath}}:
\tilde {\imath}, \tilde {\jmath} \in \widetilde I)$ is ultrametric.
\end{lemma}

\medskip

\begin{proof}
\noindent For $u,v,w \in W$ consider the following elements of
$\widetilde I$:
$\tilde {\imath}=(E^u,u)$, $\tilde {\jmath}=(E^v,v)$ and
$\tilde k=(E^w,w)$.
Take
$i\in E^u, j\in E^v, k\in E^w$. The proof is divided into
two cases.

\medskip

\noindent {\bf Case 1}. We assume $E^u\cap E^v=\emptyset$. The
ultrametric property
satisfied by $U$ and the definition of $\widetilde U$ imply
${\widetilde U}_{\tilde {\imath} \tilde {\jmath}}=
U_{ij}\ge \min\{U_{ik},U_{kj}\}\ge
\min\{{\widetilde U}_{\tilde {\imath} \tilde k},
{\widetilde U}_{\tilde k \tilde {\jmath}}\}$.
Then the property holds.

\medskip

\noindent {\bf Case 2}. We assume, without lost of generality that
$E^u\subseteq E^v$ and $v\le u$. If $E^w\cap E^v=\emptyset$ one gets
that ${\widetilde U}_{\tilde {\jmath} \tilde k}=U_{jk}<v=
{\widetilde U}_{\tilde {\imath} \tilde {\jmath}}$
and the property is verified. Finally, if
$E^w\cap E^v\neq \emptyset$ then
${\widetilde U}_{\tilde {\jmath} \tilde k}=
\min\{v,w\}\le v={\widetilde U}_{\tilde {\imath} \tilde {\jmath}}$.
$\Box$
\end{proof}

\bigskip

In the sequel we shall assume $I$ is countable and the following
hypothesis holds
$$
{\cal W}=\{U_{ij}: i,j\in I\}\subset \RR^*_+ \hbox{ has no finite
accumulation point.}  \eqno(H2)
$$
We put ${\cal W}=\{w_n: n\in \NN\}$ where $(w_n)$
increases with $n\in \NN$, $w_0>0$.
Under $(H2)$ we are able to define in  $\widetilde I$ the following
binary relation ${\widetilde{\cal T}}$. For $u, v \in {\cal W}$
we set
$$
((E^u,u), (E^v,v))\in {\widetilde{\cal T}} \Leftrightarrow \exists
n\in \NN:\; \{u,v\}=\{w_n,w_{n+1}\} \hbox{ and }E^u \cap E^v\neq
\emptyset\;.
$$
Two points $\tilde {\imath}, \tilde {\jmath} \in \widetilde I$ are
said to be neighbors in ${\widetilde{\cal T}}$ if
$(\tilde {\imath},\tilde {\jmath})\in {\widetilde{\cal T}}$.

\medskip

Observe that if $((E^{w_n},w_n),(E^{w_{n+1}},w_{n+1}))\in
{\widetilde{\cal T}}$, then
$E^{w_{n+1}}\subseteq E^{w_n}$.
The strict inclusion $E^{w_{n+1}}\neq E^{w_n}$ holds
if and only if there exists a unique $i\in E^{w_n}$
such that $w_{n}=U_{ii}$. Indeed, it suffices to show the uniqueness.
Let $i \in E^{w_n}\setminus E^{w_{n+1}}$
then $w_n\le U_{ii} < w_{n+1}$. For any other $k\in
E^{w_n}$ for which $U_{kk}=w_n$ it holds $U_{ik}\ge w_n$. We get
$i\sim k$ and from $(H1)$ we conclude $i=k$.

\medskip

It is easy to see that $({\widetilde{I}},{\widetilde{\cal T}})$ is a
tree rooted at $\tilde r$, where ${\widetilde U}_{\tilde r \tilde r}=w_0$.
This point $\tilde r$ exists (and it is unique) because either
there exists $i_0\in I$ verifying $U_{i_0 i_0}=w_0$
in which case $\tilde r={\tilde {\imath}}_0$, or in the contrary,
our construction  adds a point
$\tilde r \in \widetilde I\setminus I$ such that
${\widetilde U}_{\tilde r \tilde r}=w_0$.

\medskip

By construction $\widetilde U$ is the minimal tree matrix
extending $U$, that is we can immerse $\widetilde U$ in any other
tree extension of $U$. The tree $(\widetilde I, {\widetilde{\cal
T}})$, supporting this minimal extension, is locally finite if and
only if the following assumption is verified
$$
\forall \; w\in {\cal W} \hbox{ it holds } |J(w)/\equiv_w|<\infty\;.
\eqno(H3)
$$

\medskip

Since $(\widetilde I, {\widetilde{\cal T}})$ is a rooted tree, all
the concepts defined in the Introduction applied to it. In particular
we denote by $\widetilde \preceq$ the order relation introduced
in (\ref{orden}); by $\widetilde \land$ the associated minimum,
by $[\tilde {\imath},\infty)$ the branch born at $\tilde {\imath}$
and by $geod(\tilde {\imath}, \tilde {\jmath})$  the geodesic between
two points in $\widetilde I$. Since we have identified
$i\in I$ with ${\bf i}^U(i)\in \widetilde I$,
all these concepts  have a meaning
for elements in $I$. In particular $\widetilde \preceq$
is an extension of the order relation
$\preceq$ defined on $I$ on (\ref{orden2}), and we have the equality
$[i,\infty)^U=[i,\infty)\cap I$.

\medskip

Observe that the $\widetilde\preceq-$minimum in $(\widetilde I,
\widetilde{\cal T})$ is characterized as follows. Take
$(E^u,u),(E^v,v) \in \widetilde I$, and any $i\in E^u$, then
$(E^u,u)\widetilde \wedge(E^v,v)=E_i^w$, where $w=\sup\{z\in {\cal
W}:  z\le u, \; E_i^z\supseteq E^v\}$. Notice that $E_i^{w_0}=I$.

\medskip

\subsection{Neighbor Relation}
\label{neirel}

\medskip

We will assume that hypotheses $(H1)$-$(H3)$ are fulfilled.
The next definition is a notion of neighbor
on $I$ giving a better understanding of the embedding
$I$ in $\widetilde I$, in particular to describe how the
elements in $\widetilde I \setminus I$ are surrounded by $I$.

\medskip

\begin{definition}
\label{de1}
Let $i\in I$.

\medskip

\noindent (i) The set ${\cal V}(i)=\{j\in I: j\neq i,
geod(i,j)\cap I=\{i,j\}\}$ is called the set of $U-$neighbors
of $i$. We will also put ${\cal V}^*(i)={\cal V}(i)\cup \{i\}$.

\medskip

\noindent (ii) The set ${\cal B}(i)=\{\tilde {\jmath} \in \widetilde I:
geod(i,\tilde {\jmath})\cap I \subseteq \{i,\tilde {\jmath} \}\}$ is
called the attraction basin of $i$.

\end{definition}

\medskip

Notice that ${\cal V}^*(i)\subseteq {\cal B}(i)$.
In the next result we summarize some useful properties of
${\cal B}(i)$, ${\cal V}(i)$ and ${\cal V}^*(i)$.

\medskip

\begin{lemma}
\label{le31}

\noindent (i) $\tilde {\jmath} \in {\cal  B}(i)\setminus {\cal V}(i)$
if and only if $geod(\tilde {\jmath},i)\cap I=\{i\}$.
Moreover ${\cal V}^*(i)={\cal  B}(i)\cap I$ and
${\cal  B}(i)\setminus {\cal V}^*(i)={\cal B}(i)\setminus I$.

\medskip

\noindent (ii) If $\tilde {\jmath} \in {\cal  B}(i)\setminus
{\cal V}^*(i)$
then
all its neighbors in $(\widetilde I,\widetilde{\cal T})$ belong
to ${\cal B}(i)$. Thus,
$({\cal B}(i),
\widetilde {\cal T}|_{{\cal B}(i)\times {\cal B}(i)})$
is a tree. If we fix the root of this tree at $i$ then
the set of leaves is ${\cal V}(i)$.

\medskip

\noindent (iii) For every $\tilde {\jmath} \notin {\cal B}(i)$ there
exists a unique $k=k(i)\in {\cal V}(i)$ such that
$geod(i,\tilde {\jmath})\cap {\cal V}^*(i)=\{i,k\}$.
This unique $k$ also verifies that
$k\in geod(\tilde l,\tilde {\jmath})\cap {\cal V}^*(i)$
for every $\tilde l \in {\cal B}(i)$.

\medskip

\noindent (iv) For every $\tilde {\jmath} \in \widetilde I$ there exists
$i\in I$ such that $\tilde {\jmath} \in {\cal B}(i)$.

\medskip

\noindent (v) For $j\in {\cal V}(i)$ either
$(i,j)\in \widetilde{\cal T}$, that
is $i,j$ are neighbors on $\widetilde{\cal T}$, or there is a unique
$\tilde k \in \widetilde I\setminus I$ such that
$(\tilde k,i)\in \widetilde{\cal T}$
and $\tilde k \in {\cal B}(j)\cap geod(i,j)$.
\end{lemma}

\medskip

\begin{proof}

\noindent $(i)$ and $(ii)$ are direct from the definitions.
%
%

\medskip

\noindent $(iii)$ Take $\tilde {\jmath} \notin {\cal B}(i)$. If
$geod(\tilde {\jmath},i)\cap {\cal V}^*(i) =\{i\}$
then $geod(\tilde {\jmath},i)\cap I=\{i\}$.
In fact, if this intersection contains another point $\ell\in I$
and if we take $m \in (geod(\ell,i)\cap I)\setminus \{i\}$, the
closest point to $i$, we obtain $m\in {\cal V}^*(i)$ which is a
contradiction.
Therefore, $geod(\tilde {\jmath},i)\cap I=\{i\}$ and then
$\tilde {\jmath} \in {\cal B}(i)$ which is also a contradiction.

\medskip

Thus we can assume $|geod(\tilde {\jmath},i)\cap {\cal V}^*(i)|\ge 2$.
If this intersection has at least 3 points, from the inclusion
$geod(\tilde {\jmath},i)\subseteq geod(\tilde \ell,\tilde {\jmath})
\cup geod(\tilde \ell,i)$ for any $\tilde \ell \in \widetilde I$,
we would find a point $k \in {\cal V}^*(i)$ for which
$geod(k,i)\cap I$ contains at least 3 points. This is
a contradiction, and the result follows.

\medskip

\noindent $(iv)$ For $\tilde {\jmath}$ and $k\in I$ we consider
$geod(\tilde {\jmath},k)$. The first point in this geodesics
(when starting from $\tilde {\jmath}$) belonging to $I$
makes the job.

\medskip

\noindent $(v)$ If $i,j$ are not neighbors in
$\widetilde{\cal T}$ then
$geod(i,j)$ contains strictly
$\{i,j\}$. Take $\tilde k\neq i$ the
closest point to $i$ in $geod(i,j)$. Clearly
$\tilde k \in \widetilde I\setminus I$, otherwise
$j\notin {\cal V}^*(i)$. By the
same reason $geod(\tilde k,j)\cap I=\{j\}$ and therefore
$\tilde k \in {\cal B}(j)$. $\Box$
\end{proof}

\medskip

Let us fix some $\tilde {\jmath} \in \widetilde I\setminus I$.
From Lemma \ref{le31}
there exists $i\in I$ such that $\tilde {\jmath} \in {\cal B}(i)$.
Then the following set is well defined and
the following equality holds,
\begin{equation}
\label{ye1}
{\widetilde I} (\tilde {\jmath}):=
\bigcap\limits_{i\in I: \tilde {\jmath} \in {\cal B}(i)} {\cal B}(i)=
\{\tilde k \in \widetilde I:\;
geod(\tilde {\jmath}, \tilde k)\cap (I\setminus\{\tilde k\})=\emptyset \}.
\end{equation}
The set ${\widetilde I} (\tilde {\jmath})$
endowed with the set of edges
$\widetilde{\cal T}\cap \left({\widetilde I} (\tilde {\jmath})\times
{\widetilde I}(\tilde {\jmath})\right)$, is
the smallest subtree containing $\tilde {\jmath}$ and whose
extremal points
${\cal E}(\tilde {\jmath})
=\{\tilde k \in {\widetilde I} (\tilde {\jmath}):
\tilde k \hbox{ has a unique neighbour in }
{\widetilde I}(\tilde {\jmath})\}$
are all in $I$.

\medskip

The property that every point in $I$ has a
finite number of $U$-neighbors supplies a good
example for the next section. Observe that
the sets ${\cal V}(i)$ are finite for $i\in I$,
is clearly equivalent to
the fact that
$ {\cal B}(i)$ are finite, for $i\in I$.
This property can be easily expressed in
terms of $U$.

\begin{lemma}
\label{le1}
The sets ${\cal B}(i)$ are finite for all $i\in I$
if and only if
\begin{equation}
\label{proo}
\forall w\in {\cal W}\; \exists I^w\subset I \hbox{ finite}
\hbox{ such that: }
\forall \; i\in I\setminus I^w, \;
\max\{U_{ij}\!:\! j\in I^w, U_{ij}\!=\!U_{jj}\}> w.
\end{equation}
\end{lemma}

\medskip

\begin{proof}
Assume ${\cal B}(i)$ are finite. Clearly, it is enough to prove
(\ref{proo}) for large $w\in W$. We shall assume that the finite set
$L=\{j\in I: U_{jj}\le w\}$ is non empty and we define
$I^w=\cup_{j\in L} {\cal V}^*(j)$.

\medskip

Fix $i_0 \in L$ as one of the
closest points in $I$ to the root $\tilde r$. For
$i \in I\setminus I^w$, the geodesic $geod(i,\tilde r)$ must contain
points on $I^w$, otherwise $geod(i,i_0)=\{i,i_0\}$ which implies
$i\in {\cal V}^*(i_0)$, a contradiction. Take
$k \in geod(i,\tilde r)\cap I^w$
the farthest point from $\tilde r$. It is clear that
$U_{ik}=U_{kk}$. Assume $U_{kk}\le w$, so $k\in L$.
If $geod(k,i)\cap I =\{k,i\}$
then $i\in I^w$ which is a contradiction. Therefore,
there is at least one $m\in (geod(k,i)\cap I)\setminus \{i,k\}$.
Take $m$ the closest of such points to $k$. Clearly
$m\in {\cal V}^*(k)\subseteq I^w$
contradicting the maximality of $k$. Then
$U_{kk}>w$, proving the desired property.

\medskip

Conversely, take $i\in I$ and consider $w=U_{ii}$. We shall prove
that ${\cal V}^*(i)\subseteq I^w$. In fact, take $j \in{\cal
V}^*(i)\setminus I^w$. By hypothesis there is $k\in I^w$ such that
$U_{kk}=U_{jk}>w$. Since $U_{jj}\ge U_{jk}=U_{kk}>w=U_{ii}$ and
$j\in {\cal V}^*(i)$, we conclude $k\in geod(i,j)$ and $k\neq i$.
Since $k\neq j$, because $k\in I^w$, we arrive to a contradiction
with the definition of ${\cal V}^*(i)$, proving the result. $\Box$
\end{proof}

\subsection{Generator and harmonic functions of an Ultrametric Matrix}
\label{genultr}

\medskip

In this section we associate to an ultrametric matrix  $U$ a $q$-matrix
through its extension $\widetilde U$.
Consider the $q$-matrix $\widetilde Q$
given by (\ref{re200}), which satisfies
$\widetilde Q \widetilde U=\widetilde U \widetilde Q=-\II_{\widetilde I}$.
We can also assume that $\widetilde Q$ is defined
in ${\widetilde I}\cup {\partial_{\tilde r}}$ as in (\ref{re81}).
Further, we consider
${\widetilde X}$  the Markov process
associated to $\widetilde Q $ with lifetime $\widetilde \zeta$.

\medskip

We assume that ${\widetilde X}$ is transient.
We denote by $\tilde \mu$
the probability measure defined on ${\tilde {\partial}}_{\infty}$,
the boundary of $(\widetilde I, {\widetilde{\cal T}})$, that
is proportional to the exit distribution of ${\widetilde X}$.

\medskip

Consider
$$
\tau:=\inf\{ t>0: \; \widetilde X_t \in I\cup
\partial_{\tilde r}\} \wedge \widetilde \zeta,
$$
We point out that $\widetilde X_{\tau}$ belongs to
$I\cup \partial_{\tilde r}\cup {\widetilde{\partial}_\infty}$
with probability one.
Notice that if $\widetilde X(0)=\tilde j \in \widetilde I\setminus I$
then
$\tau=\inf\{ t>0: \;
\widetilde X_t \in {\cal E}(\tilde {\jmath})
\cup \partial_{\tilde r}\} \wedge {\widetilde \zeta}$.

\medskip

Our main assumption is
$$
\forall \tilde {\jmath} \in \widetilde I\setminus I:\;\;
\PP_{\tilde {\jmath}}\{ \widetilde X_{\tau}
\in I\cup \partial_{\tilde r}\}=1\;.
\eqno(H4)
$$
We can also write $(H4)$ as
$\PP_{\tilde {\jmath}}\{ \tau<\widetilde \zeta\}=1$
for every $\tilde {\jmath} \in \widetilde I\setminus I$.
This is also equivalent to
$\PP_{\tilde {\jmath}}
\{ \widetilde X_{\tau} \in
{\widetilde{\partial}_\infty}\}=0$
for every $\tilde {\jmath} \in \widetilde I\setminus I$.

\medskip
In the next Theorem we associate a $q$-matrix to a general
ultrametric matrix verifying $(H1)$-$(H4)$.
\medskip

\begin{theorem}
\label{p100}
Assume $U$ satisfies $(H1)$-$(H4)$, then
there exists a matrix $Q:I\times I\to \RR$ such that
$QU=UQ=-\II_I$. Moreover $Q_{ij}\neq 0$ if and only if $j \in {\cal
V}^*(i)$,
and we have
\begin{equation}
\label{re50}
Q_{ij}=\widetilde Q_{ij}+
\sum\limits_{\tilde k \in \widetilde I\setminus I}
\widetilde Q_{i\tilde k} \PP_{\tilde k}(\widetilde X_\tau=j).
\end{equation}
For $i\neq j$ this formula takes the form
$$
Q_{ij}=\widetilde Q_{ij} \hbox{ if } (i,j)\in \widetilde {\cal T}
\hbox{ and }\;
Q_{ij}=\widetilde Q_{i\tilde k}\PP_{\tilde k}(\widetilde X_\tau=j)
\hbox{ if }
(i,j)\notin \widetilde {\cal T},\; j\in {\cal V}^*(i),
$$
where $\tilde k\in {\widetilde I}\setminus I$ is the unique neighbor of
$i$ in $\widetilde {\cal T}$, that belongs to $geod(i,j)$.
\end{theorem}

\medskip

\begin{proof}
We set
$A\!=\!\tilde Q_{II}$, $B\!=\!\widetilde Q_{I, \widetilde I\setminus I}$
and $V\!=\!\widetilde U_{\widetilde I\setminus I, I}$. Since
$\widetilde U_{II}\!=\!U$ we get $AU\!+\!BV\!=\!-\II_I$.

\medskip

The crucial step in the proof is to get a
$(\widetilde I\setminus I) \times I$
matrix $Z$ whose rows
are summable and  verifies $ZU=V$, which means
$$
\widetilde U_{\tilde {\jmath} i}=
\sum\limits_{k \in I} Z_{\tilde {\jmath} k} U_{k i},\;
\hbox{ for all }
\tilde {\jmath} \in \widetilde I\setminus I,\; i\in I.
$$
For any $\tilde {\jmath} \in \widetilde I\setminus I$
consider the subtree
$\widetilde J:=\widetilde{I} (\tilde {\jmath})$
given by (\ref{ye1}). We denote
by ${\cal E}\subset I$ the set of extremal points
of $\widetilde J$. Note that
$\widetilde J\setminus {\cal E} \subseteq \widetilde I$.
We consider the following $q-$matrix on
$\widetilde J \times \widetilde J$
$$
C_{\tilde l \tilde k}=\widetilde Q_{\tilde l \tilde k} \hbox{ if }
\tilde l\in \widetilde J\setminus {\cal E}
\hbox{ and } C_{\tilde l \tilde k}=0 \hbox{ otherwise }.
$$
By definition of $\tau$,
the Markov process induced by $C$ is just the stopped process
$\widetilde X^{\tau}$. From the property
$\widetilde Q \widetilde U=-\II_{\widetilde I}$
it is deduced that for each $i\in I$ the
restriction  of $U_{\bullet i}$ to $\widetilde J$, is a $C$-harmonic
function. Therefore,
$$
U_{\tilde {\jmath} i}=\EE_{\tilde {\jmath}}(U_{\widetilde X_\tau i})=
\sum_{k\in {\cal E}} \PP_{\tilde {\jmath}}(\widetilde X_\tau=k) U_{ki},
$$
which gives the desired matrix $Z$.
Since $B$ is finitely supported and the rows of $Z$ are
summable we get
\begin{equation}
\label{re52}
(A+BZ)U=-\II_I,
\end{equation}
then $Q=A+BZ$ should be the desired
$q-$matrix. The explicit formula for $Q$ is
\begin{equation}
\label{re53}
Q_{ij}=\widetilde Q_{ij}+
\sum\limits_{\tilde k\in \widetilde I\setminus I}
\widetilde Q_{i\tilde k}Z_{\tilde k j}=
\widetilde Q_{ij}+
\sum\limits_{\tilde k\in \widetilde I\setminus I}
\widetilde Q_{i\tilde k}\PP_{\tilde k}(\widetilde X_\tau = j).
\end{equation}
From the structure of $\widetilde Q$ the last sum in (\ref{re53}) runs
over $\tilde k \in \widetilde I\setminus I$
which are neighbors of $i$ with respect to
$\widetilde {\cal T}$. From the shape of $Z$ these values of
$\tilde k$ are further restricted to the set ${\cal V}^*(j)$.
According to the Lemma \ref{le31} part ($v$)
the set of such points is
not empty when $(i,j)\notin \widetilde {\cal T}$ and moreover this set
contains exactly one point $\tilde k \in \widetilde I$.
In summary, we have for $i\neq j$
$$
Q_{ij}=\widetilde Q_{ij} \hbox{ if } (i,j)\in \widetilde {\cal T}
\hbox{ and } Q_{ij}=
\widetilde Q_{i\tilde k}\PP_{\tilde k}(\widetilde X_\tau=j) \hbox{ if }
(i,j)\notin \widetilde {\cal T},\; j\in {\cal V}(i);
$$
where in the last case, $\tilde k$ is the unique neighbor of $i$
in $\widetilde {\cal T}$ belonging to $geod(i,j)$. From this
formula we deduce that for $i\neq j$ we have
$Q_{ij}>0$ if and only if $j\in {\cal V}(i)$.
From (\ref{re52}) we deduce that $Q_{ii}<0$. Also
we get
$$
Q_{ii}=\widetilde Q_{ii}+\sum\limits_{\tilde k\in \widetilde I:
\; (\tilde k,i) \in
\widetilde {\cal T}} \widetilde Q_{i\tilde k}
\PP_{\tilde k}(\widetilde X_\tau=i).
$$
Now, let us prove that $Q$ is a $q-$matrix. Let
$k \in {\cal V}(i)$ be
such that $U_{ki}=\min\{U_{ji}:\; j\in {\cal V}(i)\}$.
This minimum is attained because the set $\{ w \in {\cal W}:
w\le U_{ii}\}$ is finite.
From the ultrametric property of $U$ we have
$U_{jk}\ge \min\{U_{ji},U_{ik}\}=U_{ik}$
for $j\in {\cal V}^*(i)$. Then, by using (\ref{re52}) we
deduce that
$$
0\ge Q_{ii}U_{ik}+\sum\limits_{j\in {\cal V}(i)}
Q_{ij}U_{jk}\ge U_{ik}(\sum_{j\in I} Q_{ij}).
$$
Hence $Q$ is a $q$-matrix.

\medskip

To finish the proof it is enough to show that $Q$ is a symmetric
matrix. This
is equivalent to prove that
\begin{equation}
\label{re54}
\widetilde Q_{i\tilde k}\PP_{\tilde k}(\widetilde X_\tau =j)=
\widetilde Q_{j\tilde l} \PP_{\tilde l}(\widetilde X_\tau=i),
\hbox{ for } j\in {\cal V}(i),\;
(j,i)\notin \widetilde {\cal T},
\end{equation}
where $\tilde k$ (respectively $\tilde l$) is the unique neighbor
in $\widetilde {\cal T}$ of $i$ (of $j$ respectively)
given by Lemma \ref{le31} part (v). The probabilities appearing in
(\ref{re54}) can be
computed in terms of $\widetilde Y=(\widetilde Y_n)_{n\in \NN}$, the
discrete skeleton
of the Markov chain on $\widetilde X$
taking values on $\widetilde I$. The
transition probabilities for this chain are
$$
\PP(\widetilde Y_1=y_1|\widetilde Y_0=y_0)=
{\widetilde Q_{y_0 y_1}\over (-\widetilde Q_{y_0 y_0})}.
$$
If we define $N=\min\{n\ge 0:\; \widetilde Y_n \in I\cup
\{\partial_{\tilde r}\}\}$ then
$$
\PP_{\tilde k}(\widetilde X_\tau =j)=
\PP_{\tilde k}(\widetilde Y_N =j).
$$
This last probability can be computed by summing up all possible
trajectories
$\widetilde Y_0=\tilde k,\, \widetilde Y_1=y_1,...,
\widetilde Y_{n-2}=y_{n-2},\, \widetilde Y_{n-1}={\tilde \ell},\,
\widetilde Y_n=j$, which
do not visit $I$ at any intermediate state. The probability of such
trajectory is
$$
{\widetilde Q_{\tilde k y_1}\over (-\widetilde Q_{\tilde k \tilde k})}
{\widetilde Q_{y_1 y_2}\over (-\widetilde Q_{y_1 y_1})}
\cdots
{\widetilde Q_{y_{n-2} {\tilde \ell}}\over
(-\widetilde Q_{y_{n-2} y_{n-2}})}
{\widetilde Q_{{\tilde \ell} j} \over
(-\tilde Q_{{\tilde \ell} {\tilde \ell}})}
$$
The probability of the reverse trajectory
$\widetilde Y_0={\tilde \ell},\, \widetilde Y_1=y_{n-2},
...,\widetilde Y_{n-2}=y_1,\, \widetilde Y_{n-1}=\tilde k,\,
\widetilde Y_n=i$, is
$$
{\widetilde Q_{{\tilde \ell} y_{n-2}}\over
(-\widetilde Q_{{\tilde \ell} {\tilde \ell}})}
{\widetilde Q_{y_{n-2} y_{n-3}}\over
(-\widetilde Q_{y_{n-2} y_{n-2}})}
\cdots
{\widetilde Q_{y_1 {\tilde k}}\over (-\widetilde Q_{y_1 y_1})}
{\widetilde Q_{\tilde k i}\over (-\widetilde Q_{\tilde k \tilde k})}.
$$
The symmetry of $\widetilde Q$ implies that (\ref{re54}) holds. Therefore,
$Q$ is symmetric and we deduce that $UQ=-\II_I$.
This finishes the proof. $\Box$
\end{proof}

\medskip

As usual we say that a function $h:I\to \RR$ is $Q-$harmonic if
$Qh=0$.
Our main result in relation with harmonic functions
for ultrametric matrices is the following one.

\medskip

\begin{theorem}
\label{t1}
Assume $U$ satisfies $(H1)$-$(H4)$.
Given a bounded  $Q$-harmonic function $h$ defined
on $I$ there exists a unique $\widetilde Q$-harmonic function
$\tilde h$ defined on $\widetilde I$, which is an extension of $h$.
\end{theorem}

\medskip

\begin{proof}
Consider the function
$$
\tilde h(\tilde {\imath})=\EE_{\tilde {\imath}}
\left(h(\widetilde X_\tau)\right), \; \tilde {\imath} \in \widetilde I.
$$
Clearly $\tilde h$ is an extension
of $h$. Using the strong Markov property for the time of first
jump of $\widetilde X$ we deduce that $\tilde h$ is
$\widetilde Q$-harmonic at every
$\tilde {\jmath} \in \widetilde I\setminus I$. Now, for
$i\in I$ we have

\begin{eqnarray*}
\sum\limits_{\tilde {\jmath} \in \widetilde I}
\widetilde Q_{i \tilde {\jmath}}
\tilde h (\tilde {\jmath})&=& \sum\limits_{j\in I} \widetilde Q_{i j}
h(j)+\sum\limits_{\tilde {\jmath} \in \widetilde I\setminus I}
\widetilde Q_{i \tilde {\jmath}} \tilde h (\tilde {\jmath})=
\sum\limits_{j \in I} \widetilde Q_{i j} h(j)+
\sum\limits_{\tilde {\jmath} \in \widetilde I\setminus I}
\widetilde Q_{i \tilde {\jmath}}
\EE_{\tilde {\jmath}}(h(\widetilde X_\tau))\\
&=&\sum\limits_{j \in I} \widetilde Q_{i j} h (j)+
\sum\limits_{\tilde {\jmath} \in \widetilde I \setminus I}
\widetilde Q_{i \tilde {\jmath}} \Big(\sum\limits_{k\in I}
\PP_{\tilde {\jmath}}(\widetilde X_\tau=k)h(k)\Big) \\
&=&\sum\limits_{j\in I} \Big(\widetilde Q_{ij}
+\sum\limits_{\tilde k \in \widetilde I \setminus I}
\widetilde Q_{i \tilde k}
\PP_{\tilde k}(\widetilde X_\tau=j)\Big) h(j)=
\sum\limits_{j\in I} Q_{ij} h(j),
\end{eqnarray*}
where the last equality follows from (\ref{re50}). Since $h$ is
$Q$-harmonic we get
$\sum\limits_{\tilde {\jmath} \in \widetilde I}
\widetilde Q_{i \tilde {\jmath}} \tilde h (\tilde {\jmath})=0$.
Then $\tilde h$ is $\widetilde Q$-harmonic at $i\in I$. $\Box$
\end{proof}

\medskip

\subsection{The Boundary of an Ultrametric Matrix}
\label{boundultr}

Recall that ${\tilde {\partial}}_{\infty}$
can be identified with
$$
{\tilde {\partial}}_{\infty}=\{({\tilde {\imath}}_n: n\ge 0):
{\tilde {\imath}}_0=\tilde r,
\forall n\in \NN, |{\tilde {\imath}}_n|=n \hbox{ and }
({\tilde {\imath}}_n, {\tilde {\imath}}_{n+1})
\in {\widetilde {\cal T}}
\}.
$$
endowed with the topology generated by the sets
${\widetilde {\cal C}}=
\{\tilde {\partial}_{\infty}(\tilde {\imath})
=[\tilde {\imath}, \infty]\cap {\tilde {\partial}}_{\infty}:
\tilde {\imath}\in \widetilde I\}$. We denote by
${\widetilde {\cal F}}_\infty$ the associated $\sigma-$field.
We will denote by
${\cal F}_\infty$ the $\sigma-$field
on ${\tilde {\partial}}_{\infty}$
generated by the sets
$\{\tilde {\partial}_{\infty}(i): i\in I\}$. We have
${\cal F}_\infty\subseteq {\widetilde {\cal F}}_\infty$, and as
we shall see further in an example, this inclusion can be strict.

\medskip

The following definition of the boundary $\partial_{\infty}^U$
associated to an ultrametric matrix extends the one for a
tree. An infinite path $(i_n: n\in \NN)$ in $I$
is called a $\preceq-$chain if $i_n\prec i_{n+1}$
for every $n\in \NN$, and the $\preceq-$chain is maximal
if we cannot add any element of $I$ to it in order that it
continues to be a $\preceq-$chain. In a tree a
$\preceq-$chain $(i_n: n\in \NN)$ is maximal, if and only
if $i_0=r$ and $|i_n|=n$ for every $n\in \NN$.
The boundary of $I$ with respect to the ultrametric matrix $U$ is
defined as
$\partial_{\infty}^U=
\{(i_n: n\in \NN) \hbox{ is a maximal $\preceq-$chain }\}$.
We endowed $\partial_{\infty}^U$ with the trace topology
from  $\tilde \partial_{\infty}$.
From the equality
$$
\partial_{\infty}^U=
\cap_{n\ge 0}\Big(
\cup_{m\ge n} \cup_{i\in I: |i|=m}\{\xi\in {\widetilde{\partial}_\infty}:
\xi(m)=i \}\Big),
$$
we get that $\partial_{\infty}^U\in {\widetilde {\cal F}}_{\infty}$.

\medskip

The function
given by
\begin{equation}
\label{fronteras}
{\bf i}^U_\infty:\partial_{\infty}^U\to {\widetilde {\partial}}_{\infty},
\;
{\bf i}^U_\infty((i_n: n\!\ge \!0))\!=\!({\tilde {\imath}}_n: n\!\ge \!0)
\Leftrightarrow \{i_n: n\!\ge \!0\}
\subseteq \{{\tilde {\imath}}_n: n\!\ge \!0\},
\end{equation}
is a well-defined one-to-one function.
We will identify $\partial_{\infty}^U$ and
${\bf i}^U_\infty(\partial_{\infty}^U)$.

\medskip

In general,
${\bf i}^U_\infty$ is not onto as shows the following example.

\medskip

\noindent {\bf Examples}.
Let $A$ be a finite set (an alphabet), we set $A^*$ the set of
finite words. In particular the empty word, denoted
by $r$ is an element of  $A^*$.
The length of a word $i$ is denoted by $|i|$, so
$|r|=0$. If $|i|\ge 1$ and $1\le m\le |i|$ we denote by
$i[1,m]$ the set of first $m$ coordinates of $i$. For any two words
$i$, $j$ we define the function $N(i,j)$ by
$N(i,j)=0$ if $i=r$ or $j=r$, and $N(i,j)=\max\{m\le \min\{|i|,|j|\}:
i[1,m]=j[1,m]\}$ when $i$ and $j$ are not $r$.
Let $w:\NN\to \RR_+$ be
a positive strictly increasing function.
For an infinite subset
$I\subseteq A^*$ we define the matrix $U$ by
$U_{i j}=w({N(i,j)})$,
for $i, j \in I$. In the sequel we fix $A=\{0,1,2\}$.

\medskip

\noindent {\it Example 1}. Let $I$ be the set of finite words finishing by
$1$. Then it is easy to see that
the minimal tree extension can be identified with the rooted tree
$(\widetilde I, {\widetilde{\cal T}})$ where
$\widetilde I=\{0,1,2\}^*$ and such that
two points $i, j$
are ${\widetilde{\cal T}}-$neighbors if $||i|-|j||=1$ and
$N(i,j)=\min\{|i|,|j|\}$.
Therefore ${\tilde {\partial}}_{\infty}$ can be identified
with $\{0,1,2\}^\NN$ and $\partial_{\infty}^U$ with the
set of infinite sequences in $\{0,1,2\}^\NN$ containing an
infinite number of $1's$. In this example
${\widetilde {\cal F}}_\infty$
does not coincide
with ${\cal F}_\infty$ on ${\tilde {\partial}}_{\infty}$,
because in this last $\sigma-$field
all the infinite sequences in $\{0,2\}^\NN$ cannot be separated.

\medskip

\noindent {\it Example 2}. Let $I$ be the set of finite words of the form
$\{0,2\}^* 1$, that is they finish by $1$ and all the other letters
are $0$ or $2$. Then in the minimal tree extension we can identify
$\widetilde I=I$, and ${\tilde {\partial}}_{\infty}$
with $\{0,2\}^\NN$. Nevertheless, $\partial_{\infty}^U$ is empty.
$\Box$

\bigskip

\centerline{\resizebox{5cm}{!}{\begin{picture}(0,0)%
\includegraphics{figura3.pstex}%
\end{picture}%
\setlength{\unitlength}{4144sp}%
\begingroup\makeatletter\ifx\SetFigFont\undefined%
\gdef\SetFigFont#1#2#3#4#5{%
  \reset@font\fontsize{#1}{#2pt}%
  \fontfamily{#3}\fontseries{#4}\fontshape{#5}%
  \selectfont}%
\fi\endgroup%
\begin{picture}(4549,4104)(1501,-5053)
\put(1501,-2746){\makebox(0,0)[lb]{\smash{\SetFigFont{17}{20.4}{\rmdefault}{\mddefault}{\updefault}{\color[rgb]{0,0,0}$i$}%
}}}
\put(2356,-3676){\makebox(0,0)[lb]{\smash{\SetFigFont{17}{20.4}{\rmdefault}{\mddefault}{\updefault}{\color[rgb]{0,0,0}$j$}%
}}}
\put(3256,-4546){\makebox(0,0)[lb]{\smash{\SetFigFont{17}{20.4}{\rmdefault}{\mddefault}{\updefault}{\color[rgb]{0,0,0}$k$}%
}}}
\put(5926,-4981){\makebox(0,0)[lb]{\smash{\SetFigFont{17}{20.4}{\rmdefault}{\mddefault}{\updefault}{\color[rgb]{0,0,0}$\xi$}%
}}}
\end{picture}
}}

\medskip

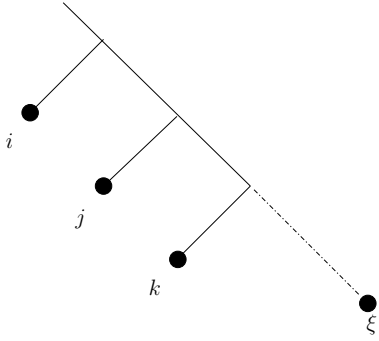
\centerline{Figure 3: $i,\ j,\ k\in I$;\ \
$\xi \in\widetilde\partial_\infty\setminus
\partial^U_\infty$}

\bigskip

Then, in general $\partial_{\infty}^U$ is small compare to
${\widetilde {\partial}}_{\infty}$, but as the following result
shows, under $(H4)$ it has full $\tilde \mu$-measure.

\begin{lemma}
\label{lefr1}
Property $(H4)$ is equivalent to $\tilde \mu(\partial_{\infty}^U)=1$.
\end{lemma}

\medskip

\begin{proof}
First notice that if for some $\tilde {\jmath} \in \widetilde I\setminus I$
it holds
$\PP_{\tilde {\jmath}}\{ \widetilde X_{\tau}
\in I\cup \partial_{\tilde r}\}<1$, then
$\PP_{\tilde {\jmath}}\{ {\widetilde X}_{\tilde \zeta}
\in {\widetilde{\partial}_{\infty}}\setminus
\partial_{\infty}^U\}>0$. Hence, the condition is necessary for $(H4)$.
For the reciprocal,
assume that
$\tilde \mu\Big({\widetilde{\partial}_{\infty}}\setminus
\partial_{\infty}^U\Big)>0$.
Therefore there exists $n\ge 0$ such that
$\PP_{\tilde r}\{A_n\}>0$, where
$A_n=
\cap_{m\ge n} \cap_{i\in I: |i|=m}\{\xi\in {\widetilde{\partial}_\infty}:
\xi(m)\neq i \}\Big)$. Take any
${\tilde {\imath}}\in \widetilde I\setminus I$, $|{\tilde {\imath}}|=n$
such that $\widetilde \partial_{\tilde {\imath}}(\infty)\cap A_n$
has positive $\PP_{\tilde r}$-measure. Then we have
$\PP_{{\tilde {\imath}}}\{{\widetilde X}_{\tilde \zeta}\in
{\widetilde{\partial}_{\infty}},\,
 {\widetilde X}_{\tilde \zeta}(\ell)\notin I,\; \forall \ell\ge 0\}>0$,
which contradicts hypothesis $(H4)$. $\Box$
\end{proof}

\medskip

\begin{theorem}
\label{t2}
Assume $U$ satisfies $(H1)$-$(H4)$.
Let $h$ be a bounded $Q$-harmonic function
such that $\lim_{i\to \xi} h(i)=\varphi(\xi)$ for every
$\xi\in \partial^U_{\infty}$. Then, there exists
$\tilde \varphi=\lim \tilde h \; \mu-$a.e.,
where $\tilde h$ is the harmonic function
associated to $h$ in Theorem \ref{t1}. Moreover, if
$\tilde \varphi$ is in the domain of ${\widetilde W}^{-1}$,
then $h$ has the representation
\begin{equation}
\label{repr01}
h(i)=\int_{{\tilde \partial}_\infty} U_{i,\eta}
({\widetilde W}^{-1} \tilde \varphi) (\eta) {\tilde \mu}(d\eta).
\end{equation}
\end{theorem}

\medskip

\begin{proof}
From Lemma \ref{lefr1}
we have
$\partial^U_{\infty}={\tilde \partial}_{\infty},\; \mu-$a.e. and
therefore (almost) every
point $\xi\in {\widetilde \partial}_\infty$ verifies
$|\{n\in \NN: \xi(n)\in I\}|=\infty$.
Also, from the hypothesis there exists
$a=\lim\limits_{{n\to \infty}\atop{\xi(n)\in I}}
h(\xi(n)$. For the first part of the statement it suffices to show
$a=\lim\limits_{n\to \infty}h(\xi(n))$. Let us consider the
subsequence $k(n)=\max\{m\le n: \xi(m)\in I \}$. We have
$\lim\limits_{n\to \infty}k(n)=\infty$. On
the other hand, for large $n$, ${\cal V}(\xi(n))\subset [\xi(k(n)),\infty)$, then
$\tilde h(\xi(n))=\EE_{\xi(n)}
\left(h(\widetilde X_\tau)\right)$ belongs to the convex closure of
the set $\{h(\xi(m)):\; \xi(m)\in I,\;  m\ge k(n)\}$. Hence
the result follows.

\medskip

Now we are able to show relation (\ref{repr01}). It suffices
to notice that for every
$i\in I\cup {\widetilde \partial}_\infty$
and $\tilde \mu-$a.e. in $\eta \in {\widetilde \partial}_\infty$,
it holds ${\widetilde U}_{i \eta}=U_{i \eta}$. Then the
proof follows from Corollary \ref{c201}. $\Box$
\end{proof}

\medskip

\begin{remark} From a topological point of view
$\partial_{\infty}^U$ is dense in
${\widetilde {\partial}}_{\infty}$
if for all $i\in I$ there exists $j\in I, j\neq i$,
such that $U_{ij}=U_{ii}$
(that is, if  for all $i\in I$ the set $[i,\infty)^U$
is infinite).
In fact, it suffices to note that
by definition of the minimal tree, for every
$\xi\in {\widetilde {\partial}}_{\infty}$ and
$n\ge 1$, there exists some $i\in I$ such that
$i\in {\widetilde {\partial}}_{\infty}(\xi(n))$. The desired
density follows by taking any $\eta \in \partial_{\infty}^U$
hanging from $i$.
\end{remark}

\medskip

\section{The Process in the Boundary}
\label{section600}

In this section we describe the process at the boundary. In
Theorem \ref{nucleo} we explicit the kernel of the process
and in Theorem \ref{esc10} we relate the behavior of the
processes at different levels, that is when we killed it
deeper and deeper in the tree. This allows to get
exit times from the elementary  pieces of the boundary, and
further to construct a simulation of the process. We emphasize
that no regularity on the tree is imposed.

\subsection{Definition and Description of the Process}

\medskip

In the sequel we put $Z\sim \exp[\lambda]$
to mean that $Z$ is a random variable exponentially
distributed with mean  $1/\lambda\in [0,\infty]$
and we denote $B\sim Ber(a)$
a Bernoulli random variable $B$ with
$\PP\{B=1\}=a\in [0,1]$.

\medskip

First, let us describe the transition probability of
the process at the boundary.

\medskip

\begin{theorem}
\label{nucleo}
Consider the symmetric kernel
\begin{equation}
\label{dens1} p(t,\xi,\eta)=\sum\limits_{n=0}^{|\xi\land \eta|}
{e^{-t/G_n(\xi)}-e^{-t/G_{n+1}(\xi)} \over \mu(C^n(\xi))},\;
(\xi,\eta)\in \partial_\infty^{reg}\times \partial_\infty^{reg},\;
t>0.
\end{equation}
This kernel is sub-Markovian with total mass
\begin{equation}
\label{dens100} e^{-t/G_0}=\int p(t,\xi,\eta) \mu(d\eta),
\end{equation}
and it is also a Feller transition kernel.

The sub-Markov semigroup $P^W_t f (\xi)=\int p(t,\xi,\eta) f(\eta)
\mu(d\eta)$, induced by this kernel in $L^2(\mu)$ verifies
$$
P^W_t  f= \sum\limits_{n\ge 0} e^{-t/G_n} \Big(\EE_{\mu} (f |
{\cal F}_n)- \EE_{\mu} (f | {\cal F}_{n-1})\Big).
$$
The infinitesimal generator of this semigroup is an extension of
$\, -W^{-1}$ defined on $\cal D$, and its potential is $W$.
Moreover the Green's kernel of this semigroup is $U$, that is
$$
U_{\xi\eta}=\int\limits_0^\infty p(t,\xi,\eta) dt \hbox{ for }
\xi, \eta\in \partial_\infty.
$$
\end{theorem}

\begin{proof}
We first notice that by integrating (\ref{dens1}) we obtain
$e^{-t/G_0}=P^W_t {\bf 1}$, that is (\ref{dens100}) holds.
Consider the following family of operators acting on $\cal D$
\begin{equation}
\label{4001}
e^{-tW^{-1}}f:=
\lim\limits_{k\to \infty}
\left(I_{\cal D}-{tW^{-1}\over k}\right)^k f
=\sum\limits_{n\ge 0} e^{-t/G_n} \Big(\EE_{\mu} (f | {\cal F}_n)-
\EE_{\mu} (f | {\cal F}_{n-1})\Big),
\end{equation}
where the last equality follows from the fact that $(G_n)$ is
predictable, that is $G_n$ is ${\cal F}_{n-1}$-measurable.
Moreover since $(G_n)$ is decreasing and positive we
also obtain
$$
\|e^{-tW^{-1}} f\|_2 \le e^{-t/G_0} \|f\|_2.
$$
Therefore, $e^{-tW^{-1}}$ has a unique continuous extension to
$L^2(\mu)$ whose norm is bounded by $e^{-t/G_0}$. Clearly
$e^{-tW^{-1}} {\bf 1}=e^{-t/G_0}$, which implies that the norm of
$e^{-tW^{-1}}$ is $e^{-t/G_0}$. It can be also proven that
$(e^{-tW^{-1}}:\; t\ge 0)$ is a sub-Markovian semigroup acting on
$L^2(\mu)$.

\medskip

A simple computation yields for $\xi\neq \eta$ and
$m\ge |\xi\land \eta|$
\begin{equation}
\label{eqq29} e^{-tW^{-1}} {\bf 1}_{C^m(\eta)}
(\xi)\!=\!\mu(C^m(\eta))\! \sum\limits_{n=0}^{|\xi\land \eta|}
{e^{-t/G_n(\xi)}-e^{-t/G_{n+1}(\xi)} \over
\mu(C^n(\xi))}=\!\!\int\! \! p(t,\xi,\eta) {\bf
1}_{C^m(\eta)}(\eta) \mu(d\eta).
\end{equation}

In the case $\mu(\{\xi^*\})>0$, the series
$$
\sum\limits_{n=0}^{\infty}
{e^{-t/G_n({\xi^*})}-e^{-t/G_{n+1}({\xi^*})}
\over \mu(C^n({\xi^*}))}\le {1\over \mu(\{{\xi^*}\})}
\sum\limits_{n=0}^{\infty}
e^{-t/G_n({\xi^*})}-e^{-t/G_{n+1}({\xi^*})}=
{e^{-t/G_0({\xi^*})}\over \mu(\{{\xi^*}\})},
$$
is convergent and
\begin{equation}
\label{eqq30}
e^{-tW^{-1}} {\bf 1}_{\{{\xi^*}\}} (\xi)=\int
p(t,\xi,\eta) {\bf 1}_{\{{\xi^*}\}}(\eta) \mu(d\eta).
\end{equation}

From equations (\ref{eqq29}) and (\ref{eqq30}) we deduce that
$$
e^{-tW^{-1}}f=P^W_t f \;\; \mu-a.e.,
\hbox{ for all } f\in \cal{D}.
$$
Thus $P^W_t$ is a pointwise representation of $e^{-tW^{-1}}$ in
$L^2(\mu)$.

\medskip

Notice that from (\ref{fil88}) the equalities
$$
\int_0^\infty p(t,\xi,\eta) dt=\sum_{n=0}^{|\xi \land \eta|}
(G_n(\xi)-G_{n+1}(\xi))/\mu(C^n(\xi))=U_{\xi \eta}
$$
hold for all $\xi, \, \eta$.
Therefore,
for any $f\ge 0$  in $L^2(\mu)$ we have by Fubini's Theorem
$$
\int_0^\infty P^W_t f (\xi) dt=
\int \int_0^\infty p(t,\xi,\eta) dt\; f(\eta) \mu(d\eta)=
\int U_{\xi \eta} f(\eta) \mu(d\eta)=W f(\xi).
$$
Also a direct computation shows that for any $f \in \cal{D}$
$$
{d\over dt}P^W_t f|_{t=0}=-W^{-1} f.
$$

The Feller property of the transition kernel $p$ is direct to
check and it follows from the fact that for a simple function $f$
we have $P_t^W f$ is also simple (in particular continuous) and
$P_t^W f\to f$ as $t\to 0$ . $\Box$
\end{proof}
\bigskip

\begin{remark} It is easy to show that for any $t>0$ fixed,
the kernel $p(t,\xi,\eta)$ given by (\ref{dens1}) verifies the
ultrametric inequality
$$
p(t,\xi,\eta)\ge \min\{p(t,\xi,\delta), p(t,\delta,\eta)\}, \hbox{
for every } \xi, \eta, \delta\in \partial_\infty^{reg}.
$$
\end{remark}

To the semigroup $P^W$ we associate a Markov process denoted by
$(\Xi_t: 0\le t \le \Upsilon)$, where $\Upsilon=\inf\{t>0: \Xi_t
\notin \partial_\infty^{reg}\}$ is its lifetime. The coffin state
of this process is written $\dag$, that is $\Xi_\Upsilon=\dag$. By
$\Xi^\nu$ we denote a copy of this Markov process with initial
distribution $\nu$ in $\partial_\infty$ and  when it starts from
$\xi$ we put $\Xi^\xi:= \Xi^{\delta_{\xi}}$. The Feller property
of $p$ implies that $\Xi$ has a right continuous with left limits
version (see \cite{Blumenthal1968}, Theorem I.9.4). We shall
always take that version. On the other hand, and as we have already
pointed out, by using
Proposition \ref{beurlingf} or by the arguments developed in
\cite{albeverio2} Theorem 4.1, the
diffusive part in the Beurling-Deny formula vanishes, so $\Xi$ is a
pure jump process.

\medskip

Let us describe more precisely the process $\Xi$, in which a main role
is played by the killing times.
Since the total mass verifies $P_t^W{\bf 1}=e^{-t/G_0}$,
the random time $\Upsilon$ is exponentially distributed:
$\Upsilon\sim \exp[1/G_0]$. By using this fact and
the symmetry of the kernel $p(t,\cdot, \cdot)$, we can check that
$\mu$ is a quasi-stationary distribution for $\Xi$, that is
\begin{equation}
\label{qsd1}
\PP_\mu\{ \Xi_t \in A\}=e^{-t/G_0}\mu(A)
\hbox{ for any measurable } A\subseteq \partial_\infty.
\end{equation}

\medskip

We will interpret the formula (\ref{dens1}) for the transition kernel
$p(t,\xi,\eta)$ in a recursive way.
Let $r_1\in S_r$ be a successor of the root $r$ such that
$\mu(\partial_\infty(r_1))>0$.
Let $\left(\bar I, \bar {\cal T}\right)$ be the subtree
rooted by $r_1$, where
$\bar I=[r_1,\infty)$ and
${\bar {\cal T}}={\cal T}\cap {\bar I}\times {\bar I}$.
We add an absorbing state identified with $r$
and we denote by
${\bar \partial}_\infty=\partial_\infty(r_1)$ the
boundary of the tree $\left(\bar I, \bar {\cal T}\right)$.
The induced level function is $||k\land l||:=|k\land l|-1$,
$k, l\in \bar I$. We also note that $\bar C^n(\xi)=C^{n+1}(\xi)$
for $\xi\in {\bar \partial}_\infty$.
Consider the tree matrix $\bar U$ induced by the weight
function $\bar\omega_k$ satisfying the recursion
$$
\bar\omega_{-1}=0\; \hbox{ and }
\Delta_k ({\bar \omega})=
\mu({\bar\partial}_\infty) \Delta_{k+1}(\omega)  \;
\hbox{ for } k\ge 0.
$$
The limit probability measure on ${\bar \partial}_\infty$ is given
by the conditional measure $\bar \mu=\mu(\bullet  |  {\bar
\partial}_\infty)$.
From this definition the new sequence of $\sigma-$ fields is
${\bar {\cal F}}_n= \sigma(C^{n+1}\cap {\bar\partial}_\infty:
C^{n+1}\in {\cal F}_n)$. Also by definition $\bar G_n=G_{n+1}$ for
all $n\ge 0$, and $(\bar G_n)$ is $({\bar {\cal
F}}_n)-$predictable.

\medskip

From the strong Markov property we obtain
for any measurable $C\subseteq {\bar \partial}_\infty$
$$
\PP_{r_1}\{X_\zeta \in C, T_{r}=\infty\}=
\PP_{r_1}\{X_\zeta \in C \}-
\PP_{r_1}\{T_{r}<\infty\}\PP_{r}\{X_\zeta \in C \}.
$$
Also we have
$$
\PP_{r_1}\{X_\zeta \in C\}={\PP_r\{X_\zeta \in C\} \over
\PP_r\{T_{r_1}<\infty\}}.
$$
Hence, we deduce that
$$
\PP_{r_1}\{X_\zeta \in C, T_{r}=\infty\}=
a\bar \mu(C),
$$
where
$$
a=\PP_r\{X_\zeta \in \partial_\infty(r_1)\}
\left({1\over
\PP_r\{T_{r_1}<\infty\}}-\PP_{r_1}\{T_r<\infty\}\right).
$$
The constant $a$ is positive because
$\mu(\partial_\infty(r_1))>0$. Then, the set of regular points
${\bar\partial}_\infty^{reg}$ is exactly the set
$\partial_\infty^{reg}\cap
\partial_\infty(r_1)$.
\medskip

Consider the operator $\bar W$ acting on
$L^2({\bar \partial}_\infty,\bar \mu )$
given by
$$
{\bar W} f(\xi)=\int \bar U_{\xi\eta} f(\eta)
\bar \mu(d\eta)
$$
Denote by $(\bar P_t:\; t\ge 0)$ the semigroup associated to $\bar
W$, and $(\bar \Xi_t: 0\le t\le \bar \Upsilon)$ the induced Markov
process on ${\bar \partial}_\infty$ with coffin state $\bar \dag$.
We denote by $\bar \Xi^\xi$ a copy of $\bar \Xi$ starting from
$\xi\in {\bar \partial}_\infty^{reg}$.

\medskip

The transition kernel for this semigroup in ${\bar
\partial}_\infty^{reg}$ is given by (see Theorem \ref{nucleo})
\begin{equation}
\label{ja1}
\bar p(t,\xi,\eta)=\sum\limits_{n=0}^{||\xi\land\eta||}
{e^{-t/\bar G_n}-e^{-t/\bar G_{n+1}}\over
\bar \mu(\bar C^n (\xi))}=
\mu({\bar \partial}_\infty)(p(t,\xi,\eta)-(e^{-t/G_0}-e^{-t/G_1})).
\end{equation}
The total mass for this kernel is $e^{-t/\bar G_0}=e^{-t/G_1}$ and
therefore $\bar \Upsilon\sim \exp[1/G_1]$, that is
$\PP_\xi\{\bar \Upsilon >t\}=e^{-t/G_1}$.

\medskip

\begin{theorem}
\label{esc10}
Fix $\xi \in {\bar \partial}_\infty^{reg}$ and
consider $\bar \Xi^{\xi}$, $\Xi^\mu$ two random independent
elements ($\Xi^\mu$ is a copy of the process $\Xi$ with initial
distribution $\mu$). Let $B\sim Ber(1-G_1/G_0)$ be a Bernoulli
variable independent of $\Xi^{\mu}$ and $\bar \Xi^{\xi}$. Under
$\PP_\xi$ the following Markov process
\begin{equation}
\label{repre1}
\widetilde\Xi^{\xi}_t=
\begin{cases}
\bar\Xi^\xi_t &\hbox{ if } t<\bar\Upsilon\\
\dag &\hbox{ if } t\ge \bar\Upsilon \hbox{ and } B=0\\
\Xi^\mu_{t-\bar \Upsilon} &\hbox{ otherwise }
\end{cases}
\end{equation}
is a copy of $\Xi^\xi$ (that is $\widetilde\Xi^{\xi}$ and
$\Xi^{\xi}$ are identically
distributed).
\end{theorem}

\begin{proof}
For $k\ge 1$ let $\xi_1,..,\xi_k\in \partial_\infty^{reg}$,
$\xi_0=\xi\in \bar{\partial}_\infty^{reg}$ and
$t_0=0<t_1<...<t_k<t_{k+1}=\infty$. We must prove
\begin{equation}
\label{2904}
\PP_\xi\{\widetilde \Xi_{t_u}\in d \xi_u, u=1,...,k\}=
\PP_\xi\{\Xi_{t_u}\in d \xi_u, u=1,...,k\}.
\end{equation}
Let us study the case $k=1$. For $\eta\in
\partial_\infty^{reg}\setminus {{\bar \partial}_\infty}$ we have
$$
\PP_\xi\{\widetilde \Xi_t \!\in \!d \eta\}=\left(1\!-\!{G_1\over
G_0}\right)
\! \int\limits_0^t\!
\PP_\mu\{\Xi_{t-u}\! \in \!d \eta\} {e^{-u/G_1}\over G_1} \; du\!=\!
\mu(d \eta)\left({1\over G_1} \!-\!{1\over G_0}\right) \!\int\limits_0^t
\!
e^{-(t-u)/G_0} e^{-u/G_1} \; du,
$$
where the last equality follows from the fact that $\mu$ is
quasi-stationary for $\Xi$ (see (\ref{qsd1})).
From (\ref{dens1}) we obtain
$$
\PP_\xi\{\widetilde \Xi_t \in d \eta\}=
\mu(d \eta)(e^{-t/G_0}-e^{-t/G_1})=\PP_\xi\{\Xi_t\in d \eta\}.
$$
\medskip

Now, when $\eta\in {\bar \partial}_\infty^{reg}$, we get again
from (\ref{qsd1})
$$
\begin{array}{ll}
\PP_\xi\{\widetilde \Xi_t \in d \eta \}&=
\PP_\xi\{\bar \Xi_t \in d \eta \}+
(1-{G_1\over G_0})\int\limits_0^t \PP_\mu\{\Xi_{t-u} \in d \eta\}
{e^{-u/G_1}\over G_1} \; du\\
&=\bar p(t,\xi,\eta)\;
\bar \mu(d\eta)+
(e^{-t/G_0}-e^{-t/G_1}) \mu(d\eta).
\end{array}
$$
Thus, from (\ref{ja1}) we find
\begin{equation}
\label{igualdad2} \bar p(t,\xi,\eta)\; \bar \mu(d\eta)=
(p(t,\xi,\eta)-(e^{-t/G_0}-e^{-t/G_1}))\mu(d\eta) \hbox{ for }\xi,
\eta \in {\bar \partial}_\infty^{reg} \;,
\end{equation}
so
$$
\PP_\xi\{\widetilde \Xi_t \in d \eta \}=
p(t,\xi,\eta)\; \mu(d\eta)=
\PP_\xi\{\Xi_t \in d \eta \},
$$
showing the case $k=1$.

\medskip

Assume that $k\ge 2$.
By a recursive argument it is sufficient to show
\begin{equation}
\label{recurr2}
\PP_\xi\{\widetilde \Xi_{t_u}\in d \xi_u, u=1,..,k\}=
\PP_\xi\{\widetilde \Xi_{t_u}\in d \xi_u, u=1,..,k-1\}
p(t_k-t_{k-1},\xi_{k-1},\xi_k).
\end{equation}
It is useful to consider the set
${\cal K}=\{u\le k: \xi_u\in
\partial_\infty^{reg}\setminus \bar{\partial}_\infty\}$. If
${\cal K}=\emptyset$ we define $\ell=k+1$ so $t_\ell=\infty$.
Otherwise we put
$\ell=\min {\cal K}$.
We have
\begin{eqnarray*}
\PP_\xi\{\widetilde \Xi_{t_u}\in d \xi_u, u=1,..,k\}=
\int_0^{t_\ell}
\PP_\xi\{\widetilde \Xi_{t_u}\in d \xi_u, u=1,..,k;
\bar\Upsilon\in dt\}.
\end{eqnarray*}
Observe that from the definition of $(\widetilde \Xi_t)$
we also have
\begin{equation}
\label{eq3112}
\{\widetilde \Xi_{t_u}\in d \xi_u, u=1,\dots,k; \bar \Upsilon>t_k\}=
\{\bar \Xi_{t_u}\in d \xi_u, u=1,\dots,k\}.
\end{equation}

\noindent ({\bf I}). Let us assume $\ell\le k$.
By definition of $(\widetilde \Xi_t)$ we find
\begin{eqnarray*}
&{}&\PP_\xi\{\widetilde \Xi_{t_u}\in d\xi_u, u=1,..,k\}\\
&=&
\sum_{j=1}^{\ell}\int_{t_{j-1}}^{t_j}
\PP_\mu\{\Xi_{t_u-t}\in d\xi_u, u=j,..,k\}
e^{-t/G_1}(G_1^{-1}\!-\!G_0^{-1})
\PP_\xi\{\bar \Xi_{t_u}\in d\xi_u, u=1,..,j\!-\!1\} dt,
\end{eqnarray*}
where it is implicit that
$\PP_\xi\{\bar \Xi_{t_u}\in d\xi_u, u=1,..,j\!-\!1\}=1$ if $j=1$.
From (\ref{qsd1}), we get
\begin{equation*}
\forall t\in (t_{j-1},t_j): \;
\PP_\mu\{\Xi_{t_u-t}\in d \xi_u, u\!=\!j,..,k\}=
\mu(d \xi_j)e^{-(t_j-t)/G_0}\PP_{\xi_j}\{\Xi_{t_u-t_j}\in d\xi_u,
u\!=\!j\!+\!1,..,k\}.
\end{equation*}

If $j\le k-1\land \ell$ and $t\in (t_{j-1},t_j)$, we can use the Markov
property of $\Xi_{t}$ to obtain,
\begin{eqnarray*}
&{}&\PP_\mu\{\Xi_{t_u-t}\in d \xi_u, u\!=\!j,..,k\}\\
&=&\mu(d \xi_j)e^{-(t_j-t)/G_0}
\PP_{\xi_j}\{\Xi_{t_u-t_j}\in d \xi_u, u=j\!+\!1,..,k\!-\!1\}
p(t_k-t_{k-1},\xi_{k-1},\xi_k)\mu(d \xi_k)\\
&=&  \PP_\mu\{\Xi_{t_u-t}\in d \xi_u, u=j,..,k\!-\!1\}
p(t_k\!-\!t_{k-1},\xi_{k-1},\xi_k)\mu(d \xi_k).
\end{eqnarray*}
Then
\begin{eqnarray}
\label{eqapoyo}
&{}&\PP_\xi\{\widetilde \Xi_{t_u}\in d\xi_u, u=1,..,k ;
\bar\Upsilon \le t_{k-1\land \ell}\}=\\
\nonumber
&{}&\PP_\xi\{\widetilde \Xi_{t_u}\in d\xi_u, u=1,..,k-1;
\bar\Upsilon \le t_{k-1\land \ell}\}
p(t_k-t_{k-1},\xi_{k-1},\xi_k) \mu(d \xi_k).
\end{eqnarray}

In the case $\ell\le k-1$ these last estimates lead to
\begin{eqnarray*}
\PP_\xi\{\widetilde \Xi_{t_u}\in d\xi_u, u=1,...,k\}=
\PP_\xi\{\widetilde \Xi_{t_u}\in d\xi_u, u=1,..,k\!-\!1\}
p(t_k\!-\!t_{k-1},\xi_{k-1},\xi_k)\mu(d\xi_k),
\end{eqnarray*}
so relation (\ref{recurr2}) is verified.

\medskip

To finish case ({\bf I}) we assume $\ell=k$. By decomposing on the events
$\{\bar\Upsilon < t_{k-1}\}$ and
$\{\bar\Upsilon \in (t_{k-1},t_k)\}$ we find
\begin{eqnarray*}
&{}&\PP_\xi\{\widetilde \Xi_{t_u}\in d\xi_u, u=1,..,k\}=\\
&{}&\PP_\xi\{\widetilde
\Xi_{t_u}\in d\xi_u, u=1,..,k\!-\!1, \bar\Upsilon \le t_{k-1}\}
p(t_k-t_{k-1},\xi_{k-1},\xi_k)\mu(d \xi_k)+\\
&{}&\PP_\xi\{\bar \Xi_{t_u}\in d\xi_u, u=1,..,k\!-\!1\}
\left(\int_{t_{k-1}}^{t_k}\!\!\!
e^{-(t-t_{k-1})/G_1}(G_1^{-1}\!-\!G_0^{-1}) e^{-(t_k-t)/G_0}dt\right)
\mu(d \xi_k).
\end{eqnarray*}

Now, we use (\ref{eq3112}) to obtain
\begin{eqnarray*}
&{}&\PP_\xi\{\widetilde \Xi_{t_u}\in d\xi_u, u=1,..,k-1,
\bar\Upsilon \le t_{k-1}\}=\\
&{}&\PP_\xi\{\widetilde \Xi_{t_u}\in d\xi_u, u=1,..,k-1\}
-\PP_\xi\{\bar \Xi_{t_u}\in d\xi_u, u=1,..,k-1\}.
\end{eqnarray*}
Therefore
\begin{eqnarray*}
\PP_\xi\{\widetilde \Xi_{t_u}\in d\xi_u, u=1,..,k\}&=&
\PP_\xi\{\widetilde \Xi_{t_u}\in d\xi_u, u=1,..,k\!-\!1\}
p(t_k\!-\!t_{k-1},\xi_{k-1},\xi_k)\mu(d \xi_k)\\
&{}& -A(t_{k-1},t_k, \xi_{k-1},\xi_k)
\PP_\xi\{\bar \Xi_{t_u}\in d\xi_u, u\!=\!1,..,k\!-\!1\}
\mu(d \xi_k),
\end{eqnarray*}
with
$$
A(t_{k-1},t_k,\xi_{k-1},\xi_k)=
p(t_k\!-\!t_{k-1},\xi_{k-1},\xi_k)+
e^{-\frac{t_k}{G_0}+\frac{t_{k-1}}{G_1}}
\left(e^{-t_{k-1}(\frac{1}{G_1}-\frac{1}{G_0})}-e^{-t_{k}(\frac{1}{G_1}
-\frac{1}{G_0})}\right).
$$
From (\ref{dens1}) and since $|\xi_{k-1}\land \xi_k|=0$ we have
$p(t_k-t_{k-1},\xi_{k-1},\xi_k)=
e^{-(t_k-t_{k-1})/G_0}-e^{-(t_k-t_{k-1})/G_1}$.
A simple computation gives $A(t_{k-1},t_k,\xi_{k-1},\xi_k)=0$.
Hence, we have shown
$$
\PP_\xi\{\widetilde \Xi_{t_u} \in d\xi_u,
u=1,..,k\}=
\PP_\xi\{\widetilde \Xi_{t_u}\in d\xi_u, u=1,..,k\!-\!1\}
p(t_k\!-\!t_{k-1},\xi_{k-1},\xi_k) \mu(d \xi_k),
$$
and equality (\ref{recurr2}) holds.

\medskip

\noindent ({\bf II}) Assume $\ell=k+1$. By decomposing on the events
$\{\bar\Upsilon < t_{k-1}\}$,
$\{\bar\Upsilon \in (t_{k-1},t_k)\}$ and
$\{\bar\Upsilon > t_{k}\}$ we obtain
\begin{eqnarray*}
&{}&\PP_\xi\{\widetilde \Xi_{t_u}\in d \xi_u, u=1,...,k\}=\\
&{}&\quad  \PP_\xi\{\widetilde \Xi_{t_u}\!\in \!d\xi_u, u=1,..,k\!-\!1\}
p(t_k-t_{k-1},\xi_{k-1},\xi_k)\mu(d \xi_k)-\\
&{}& \quad \PP_\xi\{\bar \Xi_{t_u}\in d\xi_u,u=1,..,k\!-\!1\}
p(t_k-t_{k-1},\xi_{k-1},\xi_k)\mu(d \xi_k)+\\
&{}& \quad \PP_\xi\{\bar \Xi_{t_u}\in d\xi_u, u\!=\!1,..,k\!-\!1\}
e^{-\frac{t_k}{G_0}+\frac{t_{k-1}}{G_1}}
\left(e^{-t_{k-1}(\frac{1}{G_1}-\frac{1}{G_0})}-e^{-t_{k}(\frac{1}{G_1}
-\frac{1}{G_0})}\right)
\mu(d \xi_k)+\\
&{}& \quad \PP_\xi\{\bar \Xi_{t_u}\in d\xi_u, u\!=\!1,..,k\}.
\end{eqnarray*}
From equality (\ref{igualdad2}) we have
\begin{eqnarray*}
&{}&\PP_\xi\{\bar \Xi_{t_u}\in d\xi_u, u\!=\!1,..,k\}=
\PP_\xi\{\bar \Xi_{t_u}\in d\xi_u, u\!=\!1,..,k\!-\!1\}
\bar p(t_k\!-\!t_{k-1},\xi_{k-1}, \xi_k)
\bar \mu(d\xi_k)\\
&=&\PP_\xi\{\bar \Xi_{t_u}\in d\xi_u, u=1,..,k\!-\!1\}
\Bigg(p(t_k-t_{k-1},\xi_{k-1},\xi_k)-
\bigg(e^{-\frac{t_k-t_{k-1}}{G_0}}-e^{-\frac{t_k-t_{k-1}}{G_1}}\bigg)\Bigg)
\mu(d \xi_k).
\end{eqnarray*}
Hence, the proof is finished because
$A'(t_{k-1},t_k, \xi_{k-1},\xi_k)=0$ with
\begin{eqnarray*}
A'(t_{k-1},t_k, \xi_{k-1},\xi_k)&=&
 -p(t_k-t_{k-1},\xi_{k-1},\xi_k)+
e^{-\frac{t_k}{G_0}+\frac{t_{k-1}}{G_1}}
\left(e^{-t_{k-1}(\frac{1}{G_1}-\frac{1}{G_0})}-e^{-t_{k}(\frac{1}{G_1}
-\frac{1}{G_0})}\right)\\
&{}& + p(t_k-t_{k-1},\xi_{k-1},\xi_k)-
\bigg(e^{-\frac{t_k-t_{k-1}}{G_0}}-e^{-\frac{t_k-t_{k-1}}{G_1}}\bigg).
\end{eqnarray*}
$\Box$
\end{proof}

\medskip

Let us define the iterated of the above procedure. We fix
$\xi^*\in \partial_\infty^{reg}$ a regular point of the boundary
and consider the points $\xi^*(n)$ in its geodesic starting at the
root. Let ${}^n\!\mu=\mu(\bullet \;| \;C^n(\xi^*))$ be the
conditional measure to $C^n(\xi^*)$ and ${}^nU$ be the tree matrix
induced by the weight function ${}^n\omega$ satisfying the
recursion
$$
{}^n\!\omega_{-1}=0 \hbox{ and }
\Delta_k ({}^n\omega)=\mu(C^n(\xi^*))
\Delta_{k+n}(\omega)
\hbox{ for } k\ge 0.
$$
Consider the operator
$$
{}^n Wf(\xi)=\int {}^n U_{\xi \eta}\; f(\xi)\; {}^n\!\mu(d\eta)
\hbox{ on }L^2(\partial_\infty(\xi^*(n)),{}^n\!\mu).
$$
Denote by ${}^n\Xi$ the process with generator $-({}^nW)^{-1}$ and
coffin state ${}^n\dag$. When the process starts from the
distribution $\nu$ we put ${}^n\Xi^{\nu}$. So ${}^n\Xi^{\xi}$
denotes a version of the process starting at $\xi\in
\partial^{reg}_\infty(\xi^*(n))$. With this notation ${}^0\Xi=\Xi$ and
${}^1\Xi=\bar\Xi$.

\medskip

The lifetime of ${}^n\Xi$ is written ${}^n\Upsilon$ which verifies
${}^n\Upsilon\sim \exp[1/G_n]$. We have ${}^0\Upsilon=\Upsilon$,
${}^1\Upsilon=\bar \Upsilon$. For $\xi\in
\partial^{reg}_\infty(\xi^*(n+1))$ it holds $\PP_\xi\{{}^n\Upsilon\ge
{}^{n+1}\Upsilon\}=1$ and
$$
\PP_\xi\{{}^n\Upsilon>t>{}^{n+1}\Upsilon\}=
e^{-t/G_n}-e^{-t/G_{n+1}}.
$$

\medskip

The variable $\Upsilon$ is the exit time of $\Xi$ from
$\partial_\infty^{reg}$, but for $n\ge 1$, ${}^n\Upsilon$ is not
the exit time of $\Xi$ from $C^n(\xi)$. We write ${\cal R}_n :=
\inf\{t>0:\; \Xi_t \notin C^n(\xi^*)\}$, the exit time from
$C^n(\xi^*)$.

\medskip

\begin{proposition}
The exit time ${\cal R}_n$ from $C^n(\xi^*)$ starting from a
regular point $\xi\in C^n(\xi^*)$ is exponentially distributed
with parameter
\begin{equation}
\label{ja21} \beta_n(\xi^*)=\mu(C^n(\xi^*))\left[{1\over
G_0(\xi^*)}+\sum\limits_{k=1}^n {1\over G_k(\xi^*)}\left({1\over
\mu(C^k(\xi^*))}- {1\over \mu(C^{k-1}(\xi^*))}\right)\right]
\end{equation}
that is  $\PP_\xi\{{\cal R}_n>t\}=e^{-t\beta_n(\xi^*)}$.
\end{proposition}

\begin{proof}
For $n=0$, ${\cal R}_0$ is the lifetime of $\Xi$ which is
exponentially distributed with parameter $\beta_0=1/G_0$. Now, we
will do the computation only the case $n=1$, the general case is
proven analogously. From (\ref{repre1}) we compute the
distribution of ${\cal R}_1$ by
\begin{equation}
\label{dist1}
\PP_\xi\{{\cal R}_1>t\}=
e^{-t/G_1}+\left(1-{G_1\over G_0}\right)
\int\limits_0^t {e^{-u/G_1}\over G_1}
\int_{{\bar \partial}_\infty}
\PP_\eta\{{\cal R}_1>t-u\} \mu(d\eta)\; du.
\end{equation}
Integrating this relation with respect to
$\xi \in {\bar \partial}_\infty$
we obtain the following equation for
$\psi(t)=\int_{{\bar \partial}_\infty}
\PP_\eta\{{\cal R}_1>t\} \mu(d\eta)$
$$
\psi(t)=\mu({\bar \partial}_\infty)\left(e^{-t/G_1}+
\left(1-{G_1\over G_0}\right)
\int\limits_0^t {e^{-u/G_1}\over G_1} \psi(t-u) \; du\right)\;.
$$
The solution to this equation is given by
$\psi(t)=\mu({\bar \partial}_\infty) e^{-t\beta_1}$,
where
$$
\beta_1={1- (1-{G_1\over G_0})\mu({\bar \partial}_\infty)
\over G_1}=
{\mu(\partial_\infty \setminus {\bar \partial}_\infty)
\over G_1}+
{\mu({\bar \partial}_\infty)\over G_0}\in
\left({1\over G_1},{1\over G_0}\right).
$$
Replacing this expression on the right hand side of (\ref{dist1})
we obtain
$\PP_\xi\{{\cal R}_1>t\}=e^{-t\beta_1}$. $\Box$
\end{proof}

\begin{remark}
We notice that $\beta_n(\xi)=\beta_n(\xi^*)$ for all regular
points $\xi \in C^n(\xi^*)$.
\end{remark}

In what follows we explain in detail a scheme for
simulating the process $\Xi$ using exponential random variables,
and  a natural generalization of
Theorem \ref{esc10}. In this result we denote by ${}_0\widetilde
\Xi$ a copy of $\Xi$.

For an approximation of the process $\Xi$ using
the projections of its generator onto the spaces associated to the
filtration defined by the levels of the tree see \cite{MRSM}.

\medskip

\begin{theorem}
\label{essn} Let $n\ge 1$, $\xi \in {\partial}_\infty$ and $(B_k:
k\ge 1)$ be a sequence of independent Bernoulli random variables
with
$\PP_\xi\{B_k=1\}=1-\PP_\xi\{B_k=0\}=1-G_k(\xi)/G_{k-1}(\xi)$.
Then, under $\PP_\xi$ the following Markov process, defined
recursively,
\begin{equation}
\label{repre2} {}_n\widetilde\Xi^{\xi}_t\!=\!
\begin{cases}
{}^n\Xi^\xi_t &\hbox{ if } t < {}^n\Upsilon\\
\dag &\hbox{ if } t \ge  {}^n\Upsilon \hbox{ and }
B_k = 0  \hbox{ for } 1 \le k \le n \\
{}^{k}\widetilde \Xi^{\;{}^k\! \mu}_{t- {}^n  \Upsilon} &\hbox{ if
} t \ge {}^n\Upsilon \hbox{ and }B_{k+1}=1, B_{p} = 0 \hbox{ for }
k+1< p \le n,
\end{cases}
\end{equation}
is a copy of $\Xi^\xi$ (recall that $\; {}^k\! \mu=\mu(\bullet\,|
C^k(\xi)))$.
\end{theorem}

\bigskip

\medskip

Therefore if we could define properly $\lim\limits_{n\to \infty}
{}_n\widetilde\Xi$, we would get this limit is also distributed as
$\Xi$. We will achieve this by using a backward construction of
the process ${\Xi}$. First we state a result on exponential
variables whose proof is left to the reader.

\begin{lemma}
\label{exps}
Let $0<\lambda_0<\lambda_1$.

\medskip

\noindent (i) Let $\Theta_1$, $\Theta_0$ and $B$ be independent
random variables such that $\Theta_1\sim \exp[\lambda_1]$,
$\Theta_0\sim \exp[\lambda_0]$ and $B\sim
Ber(1-\lambda_0/\lambda_1)$. Then the variable
${\Gamma}_0=\Theta_1+B \Theta_0$ is distributed as $
\exp[\lambda_0]$.

\medskip

\noindent (ii) Let ${\Gamma}_0$, ${\Gamma '}_0$ and $Z_1$ be
independent random variables such that
${\Gamma}_0\sim {\Gamma '}_0\sim \exp[\lambda_0]$ and
$Z_1\sim \exp[\lambda_1-\lambda_0]$. Consider the
random vector $(\Theta_1, \Theta_0, B)$
defined in the following conditional way
$$
\Theta_1={\Gamma}_0, \Theta_0={\Gamma '}_0, B=0 \hbox{ if }
Z_1\ge {\Gamma}_0 \hbox{ and }
\Theta_1=Z_1, \Theta_0={\Gamma}_0-Z_1, B=1
\hbox{ if } Z_1< {\Gamma}_0.
$$
Then $\Theta_1$, $\Theta_0$ and $B$ are
independent random variables that verify
$\Theta_1\sim \exp[\lambda_1]$, $\Theta_0\sim \exp[\lambda_0]$,
$B\sim Ber(1-\lambda_0/\lambda_1)$ and
${\Gamma}_0=\Theta_1+B \Theta_0$.
\end{lemma}

\medskip

Now we introduce the elements involved in the simulation of the
process. First, for $t>0$ we will denote by
${\bf K}_{t}\,{}^n\Xi^\xi$
a copy of the process ${}^n\Xi^\xi$,
conditioned to the fact that the killing time ${}^n\Upsilon$
verifies ${}^n\Upsilon=t$. In particular ${\bf K}_{t}\, \Xi^\xi$
denotes a copy of the process $\Xi^\xi$, conditioned to be killed
at time $\Upsilon=t$.

\medskip

Now, consider the following set of sites
$$
{\cal M}=\{\vec k=(k_0,..,k_n): 0= k_0\le .. \le k_n, k_i\in \NN,
n\in \NN \}.
$$
Let $\vec k=(k_0,..,k_n)\in {\cal M}$. We put $|\vec k|=n$ and
call it the length of $\vec k$. ${\cal M}$ is the set of sites of
a tree with root $(0)$ and where every site $\vec k$ has a
countable number of successors $(\vec k,m)=(k_0,..,k_n,m)$ with
$m\ge k_n$. If $|\vec k|\ge 1$ we denote by ${\vec
k}^-=(k_0,..,k_{n-1})$ its predecessor. We call level $n$ the
class of sites with length $n$. We define $\vec k +1$ as follows,
$$
(0)+1=(1) \hbox{ and } \vec k +1=({\vec k}^-, k_n+1)
\hbox{ for } |\vec k|=n\ge 1.
$$
We denote ${\cal M}+1=\{\vec k +1: \vec k \in {\cal M} \}$.
Observe that $\vec k +1$ is in ${\cal M}$ except in the case $\vec
k=(0)$. On the other hand, if $\vec k\in {\cal M}$ and
$|\vec k|=n\ge 1$, then $\vec k\in {\cal M}+1$ if and only if
$k_n>k_{n-1}$. We put $|(1)|=0$ so $|\vec k +1|=|\vec k|$ holds
for all $\vec k \in {\cal M}$.

\medskip

Now, we will define a countable random set of points
$\Lambda=\left(\Lambda({\vec k}): \vec k\in {\cal M}+1\right)$
taking values in $\partial_\infty$. We will do it in a recursive
way on the length of $\vec k$. First we fix
$$
\Lambda({(1)})=\xi\in \partial_\infty,
$$
For the other levels these random variables satisfy the following
conditional laws. Let $n\ge 0$. For level $n+1$ and $|\vec k|=n$ we put,
$$
\PP\{ \Lambda({(\vec k,m)})\in A_m: m > k_n \;|\; \Lambda({\vec
k'}), {\vec k'}\in {\cal M}+1, |\vec k'| \le n \}= \prod_{m> k_n}
\mu\{ A_m | C^{m-1} (\Lambda({\vec k +1})) \};
$$
with $A_m$ a measurable set, $A_m\subseteq C^{m-1}
(\Lambda({\vec k +1}))$, $m\ge 1$. In particular
$\Lambda({(\vec k,m)})$ given $\Lambda({\vec k'}),
{\vec k'}\in {\cal M}+1, |\vec k'| \le n$, is distributed
according to $\mu(\cdot | C^{m-1} (\Lambda({\vec k +1})))$.

\medskip

Conditionally on $\Lambda$, we consider the following countable
family of independent random variables
$\left({\bf Z}_{m}^{\vec k}: \vec k=(k_0,..,k_n)\in
{\cal M}, m>k_n \right)$,
whose marginal distributions verify
$$
{\bf Z}_{m}^{\vec k}\sim exp
\left[\frac{1}{G_{m}(\Lambda({\vec k +1}))}-
\frac{1}{G_{m-1}(\Lambda({\vec k +1}))}\right] , \; m> k_n.
$$

\medskip

\begin{lemma}
\label{exp100}
We have
\begin{equation}
\label{exp101}
\PP\{\forall {\vec k}\in {\cal M}:
\liminf\limits_{m\to \infty}
{\bf Z}_{m}^{\vec k}=0 \; | \; \Lambda \}=1
\end{equation}
\end{lemma}

\medskip

\begin{proof}
Since $\Lambda({\vec k})$ is a regular point,  we get
$G_m(\Lambda({\vec k +1}))>G_{m+1}(\Lambda({\vec k +1}))$ and
$\lim\limits_{n\to \infty} 1/G_n(\Lambda({\vec k +1}))=\infty$.
Then, the telescopic property gives
$$
\forall x>0, \forall m_0> k_n:
\PP\{\forall m\ge m_0: {\bf Z}_{m}^{\vec k}\ge x
\; | \; \Lambda\}=
\prod_{m\ge m_0}
e^{-x\left(\frac{1}{G_m(\Lambda({\vec k +1}))}-
\frac{1}{G_{m-1}(\Lambda({\vec k +1}))}\right)}=0.
$$
Therefore
$$
\PP\{\liminf\limits_{m\to \infty}
{\bf Z}_{m}^{\vec k}=0 \; | \; \Lambda \}=1.
$$
Since $\cal M$ is a countable set, the result is shown. $\Box$
\end{proof}

\medskip

Hence, we can assume $\liminf\limits_{m\to \infty}
{\bf Z}_{m}^{\vec k}=0$. We will denote
${\bf Z}^{\vec k}= \left({\bf Z}_{m}^{\vec k}: m>k_n \right)$.
In particular
${\bf Z}^{(0)}= \left({\bf Z}_{m}^{(0)}: m>0 \right)$.
\medskip

Now, we need to introduce some operations in the class of strictly
positive sequences having $0$ as an accumulation point. Let
$\ell\ge 0$ be a positive integer, $b>a\ge 0$ and ${\bf z}=({\bf
z}_n: n> \ell)$ be a strictly positive sequence verifying
$\liminf\limits_{n\to \infty} {\bf z}_n=0$. Consider the strictly
increasing sequence $({\underline \gamma}_{m}:={\underline
\gamma}_m[{\bf z}, \ell;a,b] : m\ge \ell)$ given by
\begin{eqnarray*}
&{}&
{\underline \gamma}_\ell:=\ell, \;\;
{\underline \gamma}_{\ell+1}:=\inf\{n>{\underline \gamma}_{\ell}:
{\bf z}_{n}<b-a \} \hbox{ and } \\
&{}&
{\underline \gamma}_{m+1}:=\inf\{n>{\underline \gamma}_{m}:
{\bf z}_{n}<{\bf z}_{{\underline \gamma}_{m}} \} \hbox{ for } m> \ell+1.
\end{eqnarray*}
Associated to it, we define a new sequence ${\bf z}':={\bf
z}[\ell; a,b]$ whose elements ${\bf z}'=({\bf z}'_n: n\ge \ell)$
are given by
\begin{equation}
\label{exp12}
{\bf z}'_\ell=b \hbox{ and }
{\bf z}'_{m}=a+{\bf z}_{{\underline \gamma}_{m}} \hbox{ for } m> \ell.
\end{equation}
We also introduce the following sequence of integers
$$
{\overline \gamma}_m=
{\overline \gamma}_m[{\bf z}, \ell;a,b] :=
{\underline \gamma}_{m+1}-1 \hbox{ for }
m\ge \ell.
$$
Therefore ${\underline \gamma}_m\le {\overline \gamma}_m$. Notice
that the sequence $({\bf z}'_n: n\ge \ell)$ strictly decreases to
$a$.

\medskip

The next step consists in defining a countable random set of times
${\bf t}=\{{\bf t}_{\vec k}: {\vec k}\in {\cal M}\cup \{(1)\}\}$
taking values in $\RR_+$, conditioned to $\Lambda$. Also, to each
point ${\bf t}_{\vec k}$ we associate a point $\xi_{\vec k}\in
\partial_\infty$.

\medskip

This construction will be done in a recursive way on the levels of
${\cal M}$. For level $0$ we put ${\bf t}_{(1)}=0$ and we choose
${\bf t}_{(0)}\sim \exp[1/G_0]$. We will also put
${\underline \gamma}_0=0$.
We define $\xi_{(1)}=\xi$ and $\xi_{(0)}=\dag$.

\medskip

Let us define ${\bf t}_{\vec k}$ for level $1$, that is when
$|\vec k|=1$. From Lemma \ref{exp100} we have
$\liminf\limits_{m\to \infty} {\bf Z}_{m}^{\vec k}=0$, then we can
define
$$
{\bf t}_{(0,m)}={\bf Z}'_m \hbox{ where } {\bf Z}'= {\bf Z}^{(0)}
[0; 0,{\bf t}_{(0)}].
$$
Therefore the sequence ${\bf t}_{(0,m)}$ starts from ${\bf t}_{(0,0)}=
{\bf t}_{(0)}$ and it is strictly decreasing to
$0$. We introduce the sequences
$$
{\underline \gamma}_m^{(0)}= {\underline \gamma}_m[{\bf Z}^{(0)},
0; 0, {\bf t}_{(0)}] \hbox{ and } {\overline \gamma}_m^{(0)}=
{\overline \gamma}_m[{\bf Z}^{(0)}, 0; 0, {\bf t}_{(0)}]
\hbox{ for } m\ge 0.
$$
By definition ${\underline \gamma}_m^{(0)}\le {\overline
\gamma}_m^{(0)}={\underline \gamma}_{m+1}^{(0)}-1$,
${\underline \gamma}_0^{(0)}=0$, ${\underline \gamma}_1^{(0)}=
\inf\{m>0: {\bf Z}^{(0)}_m<{\bf t}_{(0)}\}$ and
$$
{\bf t}_{(0,m)}={\bf Z}^{(0)}_{{\underline \gamma}_m^{(0)}},
\hbox{ for } m\ge 1.
$$
We associate to ${\bf t}_{(0,0)}$ the value
$\xi_{(0,0)}=\xi_{(0)}=\dag$ and for $m\ge 1$ we
associate to each ${\bf t}_{(0,m)}$ the value
$\xi_{(0,m)}=\Lambda((0,{\underline \gamma}_m^{(0)}))$. In each
interval $[{\bf t}_{(0,m)},{\bf t}_{(0,m-1)})$ we put a copy of
the process
$$
{\bf K}_{{\bf t}_{(0,m-1)}-{\bf t}_{(0,m)}}\, {}^{{\overline
\gamma}_{m-1}^{(0)}}\Xi^{\xi_{(0,m)}},
$$
that is a copy of the process of level
${\overline \gamma}_{m-1}^{(0)}$,
that starts at time ${\bf t}_{(0,m)}$ at the
point $\xi_{(0,m)}$, conditioned that its lifetime is ${\bf
t}_{(0,m-1)}-{\bf t}_{(0,m)}$.  From Theorem \ref{essn} and Lemma
\ref{exps} we get that the whole process defined on $[0, {\bf
t}_{(0)}]$ is a copy of ${\bf K}_{{\bf t}_{(0)}}\Xi^\xi$. The
intervals of level $1$ are, from right to left,
$[{\bf t}_{(0,1)},{\bf t}_{(0,0)})$, $[{\bf t}_{(0,2)},
{\bf t}_{(0,1)}), \dots ,[{\bf t}_{(0,m)},{\bf t}_{(0,m-1)}),\dots$.
Their left extremes are
respectively ${\bf t}_{(0,1)},{\bf t}_{(0,2)},\dots,{\bf
t}_{(0,m)},\dots$ and the points on the boundary associated are
$\xi_{(0,1)},\xi_{(0,2)},\dots,\xi_{(0,m)},\dots$.
We associate to each $\vec k=(0,m)$ the index
$\vec k^*=:(0,\overline \gamma^{(0)}_m)$, then
$\xi_{\vec k +1}=\Lambda(\vec k^* +1)$.

\bigskip
\bigskip

\centerline{\resizebox{15cm}{!}{\begin{picture}(0,0)%
\includegraphics{firststep.pstex}%
\end{picture}%
\setlength{\unitlength}{4144sp}%
\begingroup\makeatletter\ifx\SetFigFont\undefined%
\gdef\SetFigFont#1#2#3#4#5{%
  \reset@font\fontsize{#1}{#2pt}%
  \fontfamily{#3}\fontseries{#4}\fontshape{#5}%
  \selectfont}%
\fi\endgroup%
\begin{picture}(11584,1360)(-179,-1267)
\put(-179,-736){\makebox(0,0)[lb]{\smash{\SetFigFont{17}{20.4}{\rmdefault}{\mddefault}{\updefault}{\color[rgb]{0,0,0}$t_{(1)}=0$}%
}}}
\put(-179,-1186){\makebox(0,0)[lb]{\smash{\SetFigFont{17}{20.4}{\rmdefault}{\mddefault}{\updefault}{\color[rgb]{0,0,0}$\xi_{(1)}=\xi$}%
}}}
\put(10666,-781){\makebox(0,0)[lb]{\smash{\SetFigFont{17}{20.4}{\rmdefault}{\mddefault}{\updefault}{\color[rgb]{0,0,0}$t_{(0)}\sim [1/G_0]$}%
}}}
\put(10936,-1186){\makebox(0,0)[lb]{\smash{\SetFigFont{17}{20.4}{\rmdefault}{\mddefault}{\updefault}{\color[rgb]{0,0,0}$\xi_{(0)}=\dag$}%
}}}
\put(1936,-1186){\makebox(0,0)[lb]{\smash{\SetFigFont{17}{20.4}{\rmdefault}{\mddefault}{\updefault}{\color[rgb]{0,0,0}$\xi_{(0,3)}\sim\mu | C^2(\xi)$}%
}}}
\put(4816,-1186){\makebox(0,0)[lb]{\smash{\SetFigFont{17}{20.4}{\rmdefault}{\mddefault}{\updefault}{\color[rgb]{0,0,0}$\xi_{(0,2)}\sim \mu | C^1(\xi)$}%
}}}
\put(7471,-1186){\makebox(0,0)[lb]{\smash{\SetFigFont{17}{20.4}{\rmdefault}{\mddefault}{\updefault}{\color[rgb]{0,0,0}$\xi_{(0,1)}\sim\mu | C^0(\xi)$}%
}}}
\put(2476,-736){\makebox(0,0)[lb]{\smash{\SetFigFont{17}{20.4}{\rmdefault}{\mddefault}{\updefault}{\color[rgb]{0,0,0}$t_{(0,3)}$}%
}}}
\put(5176,-736){\makebox(0,0)[lb]{\smash{\SetFigFont{17}{20.4}{\rmdefault}{\mddefault}{\updefault}{\color[rgb]{0,0,0}$t_{(0,2)}$}%
}}}
\put(7966,-736){\makebox(0,0)[lb]{\smash{\SetFigFont{17}{20.4}{\rmdefault}{\mddefault}{\updefault}{\color[rgb]{0,0,0}$t_{(0,1)}$}%
}}}
\end{picture}
}}

\medskip

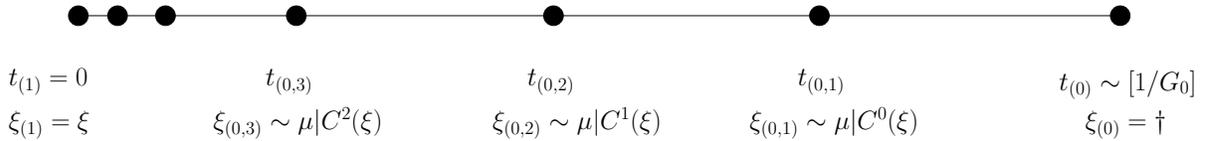
\centerline{Figure 4: First Step of Simulation}

\bigskip

Now we iterate this procedure. We assume the construction has been
made up to some $n\ge 1$. Consider an interval
$[{\bf t}_{\vec k +1}, {\bf t}_{\vec k})$ of the level $n$
characterized by $\vec k=(k_0,\dots,k_n)\in {\cal M}$ and its
corresponding $\vec k^*\in {\cal M}$. The associated point
to its left extreme ${\bf t}_{\vec k +1}$ is
$\xi_{\vec k+1}=\Lambda(\vec k^* +1)$.
In this interval we need to simulate a
copy of the conditional process
$$
{\bf K}_{{\bf t}_{\vec k}-{\bf t}_{\vec k+1}}\,{}^{\overline
\gamma^{{\vec k}^-}_{k_n}} \Xi^{\xi_{\vec k+1}}.
$$
This requires to simulate exponential random variables distributed
as
$$
exp[1/G_{m+1}(\xi_{\vec k+1})-1/G_m(\xi_{\vec k+1})],\;
m > \overline \gamma^{{\vec k}^-}_{k_n}.
$$
That is, we should consider the variables
$$
Z^{\vec k^*}_m \; \hbox{ for }
m > \overline \gamma^{{\vec k}^- }_{k_n}.
$$
We put ${\bf t}_{(\vec k,k_n)}={\bf t}_{\vec k}$,
$\xi_{(\vec k,k_n)}=\xi_{\vec k}$ and for
$m > k_n$
$$
{\bf t}_{(\vec k,m)}={\bf t}_{\vec k +1}+{\bf Z}'_m \hbox{ with }
{\bf Z}'= {\bf Z}^{{\vec k}^*}[{\overline \gamma}_{k_n}^{{\vec k}^-};
{\bf t}_{\vec k +1}, {\bf t}_{\vec k}].
$$
We also set
$$
{\underline \gamma}_m^{\vec k}=
{\underline \gamma}_m[{\bf Z}^{{\vec k}^*},
{\overline \gamma}_{k_n}^{{\vec k}^-};
{\bf t}_{\vec k +1}, {\bf t}_{\vec k}] \hbox{ and }
{\overline \gamma}_m^{\vec k }=
{\overline \gamma}_m[{\bf Z}^{{\vec k}^*},
{\overline \gamma}_{k_n}^{{\vec k}^-};
{\bf t}_{\vec k +1}, {\bf t}_{\vec k}].
$$

In the interval whose index is $\vec h=:(\vec k ,m-1)$, we
associate to the left extreme ${\bf t}_{(\vec k ,m)}$ the point
$\xi_{(\vec k ,m)}=
\Lambda\left((\vec k^*,\underline \gamma_m^{\vec k })\right)=
\Lambda\left((\vec k^*,\overline \gamma_{m-1}^{\vec k })+1\right)$
that belongs to
$C^{\overline \gamma^{\vec k}_{m-1}}(\Lambda(\vec k^*+1))$
which was  chosen in this set in a uniform way according to $\mu$.
In this way we define
$\vec h^*=(\vec k,m-1)^*=:
(\vec k^*,\overline \gamma^{\vec k}_{m-1})$
obtaining that
$$
\xi_{\vec h+1}=\Lambda\left(\vec h^*+1\right).
$$
In the interval considered we put a copy of the killed process
$$
{\bf K}_{{\bf t}_{\vec h}-{\bf t}_{\vec h+1}}\,{}^{\overline
\gamma ^{{\vec h}^-}_{m-1}}\Xi^{\xi_{\vec h+1}}.
$$

By construction we have
$$
\lim\limits_{m\to \infty}
\mathop{\downarrow}{\bf t}_{(\vec k,m)}={\bf t}_{\vec k +1}
$$
Therefore every point ${\bf t}_{\vec k}$, ${\vec k}\in {\cal M}+1$,
is an accumulation point of $({\bf t}_{({\vec k},m)})$. On
the other hand for $m>k_n$ above construction gives
$$
\xi_{(\vec k,m)}\in  C^{\overline \gamma^{\vec k}_{m-1}}
\left(\Lambda({\vec k^* +1})\right).
$$
Then $\lim\limits_{m\to \infty}\xi_{(\vec k,m)}=
\Lambda({\vec k^* +1})=\xi_{\vec k+1}$.

\medskip

In the sequel we adopt the following notation: for $k\ge 0$
and $p\ge 1$ by $k^{[p]}$ we mean the sequence
of $p$ symbols $k$, that is $k^{[p]}=\underbrace{k,\dots,k}_{p}$.

\begin{lemma}
\label{medible}
For every $\vec k=(k_0,...,k_n)$ the set of random variables
$\left(Z_{k_n+1}^{({\vec k},k_n^{[p]})}: p\ge 1 \right)$ are independent
and identically distributed.
\end{lemma}

\medskip

\begin{proof}
We must only show they are identically distributed.
An inductive argument implies that it suffices to show
that $Z_{k_n+1}^{({\vec k},k_n)}$ and
$Z_{k_n+1}^{({\vec k},k_n,k_n)}$
have the same distribution.

\medskip

We have
\begin{eqnarray*}
Z_{k_n+1}^{({\vec k},k_n)}&\sim& \exp
\left[\frac{1}{G_{k_n+1}(\Lambda({\vec k}+1))}-
\frac{1}{G_{k_n}(\Lambda({\vec k}+1))}\right],\\
Z_{k_n+1}^{({\vec k},k_n,k_n)}&\sim& \exp
\left[\frac{1}{G_{k_n+1}(\Lambda(({\vec k}, k_n+1)))}-
\frac{1}{G_{k_n}(\Lambda(({\vec k}, k_n+1)))}\right].
\end{eqnarray*}

We notice that by construction $\Lambda({\vec k}+1)\in
C^{k_n}(\Lambda({\vec k^-}+1))$, and $\Lambda(({\vec k},k_n+1))\in
C^{k_n}(\Lambda({\vec k}+1))=C^{k_n}(\Lambda({\vec k^-}+1))$.
Since $G_{k_n+1}, G_{k_n}$ are
${\cal F}_{k_n}-$measurable we deduce
$$
G_{k_n+1}(\Lambda({\vec k}+1))=(G_{k_n+1}(\Lambda(({\vec k}, k_n+1)))
$$
and similarly
$G_{k_n}(\Lambda({\vec k}+1))=(G_{k_n}(\Lambda(({\vec k}, k_n+1)))$,
proving the result. $\Box$
\end{proof}

\medskip

\begin{corollary}
\label{medida0}
We have
$$
\PP\{\forall x>0, \forall {\vec k}\in {\cal M}:
\exists p\ge 1, Z_{k_n+1}^{({\vec k},k_n^{[p]})}>x \; | \; \Lambda \}=1.
$$
\end{corollary}

\medskip

\begin{proof}
It is obtained directly from the last Lemma. $\Box$
\end{proof}

\medskip

Therefore we can assume that, conditioned to $\Lambda$,
for every fixed $x>0$ and ${\vec k}\in {\cal M}$, there
exists $p\ge 1$ such that $Z_{k_n+1}^{({\vec k},k_n^{[p]})}>x$.

\medskip

\begin{corollary}
\label{finitud}
In every interval $[{\bf t}_{\vec k+1}, {\bf t}_{\vec k})$
and for every $\ell\ge 0$ there exists only a finite number
of points ${\bf t}_{\vec h}$ in its interior,
that is ${\vec h}=({\vec k},k_{n+1},...,k_s)$,
such that ${\underline \gamma}^{{\vec h}^-}_{k_s}=\ell$.

\end{corollary}

\medskip

\begin{proof}
Notice that $T_0=:\{{\bf t}_{\vec k}: {\underline \gamma}^{{\vec k}^-}_{k_n}=0\}=\{{\bf t}_{(0)}\}$.
The fact that the set
$T_1=:\{{\bf t}_{\vec k}: {\underline \gamma}^{{\vec k}^-}_{k_n}=1\}$
is finite follows from the inclusion of events
$\{|T_1|=\infty\}\subseteq \{Z_1^{(0^{[r]})}<t_{(0)}: r\ge 1\}$
and last Corollary.
A recurrence
argument using Corollary \ref{medida0} finishes the proof. $\Box$
\end{proof}

\medskip

\begin{theorem}
\label{sim100}
The process $(\Xi_t: t< \Upsilon)$
has a version that is right continuous with left limits
and in the set of points
$[0,\Upsilon)\setminus \{{\bf t}_{\vec k} :{\vec k}\in {\cal M}\}$
it is continuous.
\end{theorem}

\medskip

\begin{proof}
Let us fix an interval $[{\bf t}_{\vec k+1}, {\bf t}_{\vec k})$.
We denote by $H=\{{\vec h}:
{\bf t}_{\vec h}\in ({\bf t}_{\vec k+1},{\bf t}_{\vec k})\}$
and a generic ${\vec h}\in H$ is denoted
by ${\vec h}=({\vec k},k_{n+1},...,k_s)$.
To each $\ell\ge 0$ we associate the set
$T^{\vec k}_\ell=\{{\bf t}_{\vec h}\in
({\bf t}_{\vec k+1},{\bf t}_{\vec k}):
{\underline \gamma}^{{\vec h}^-}_{k_s}=\ell\}$.
Consider the set of nonnegative integers
$L^{\vec k}=\{\ell: T^{\vec k}_\ell \neq \emptyset\}$,
and for each $\ell\in L^{\vec k}$ denote
${\bf t}_{{\vec h}_\ell} = \max T^{\vec k}_\ell$
and put ${\vec h}_\ell=({\vec k},k_{n+1},...,k_{s_\ell})$.
By construction ${\bf t}_{{\vec h}_\ell}$
strictly increases along $\ell\in L^{\vec k}$.

\medskip

Assume $L^{\vec k}$ is finite and let $\ell^*$
be its maximal value. We necessarily have
$\Xi_t=\xi_{{\vec h}_{\ell^*}}$
for $t\in [{\bf t}_{{\vec h}_{\ell^*}}, {\bf t}_{\vec k})$,
because in the contrary there would be some time
$\tilde t\in ({\bf t}_{{\vec h}_{\ell^*}}, {\bf t}_{\vec k})$
for which $\tilde t={}^r\dag$, contradicting the maximality
of ${\bf t}_{{\vec h}_{\ell^*}}$.

\medskip

Now assume  $L^{\vec k}$ is infinite. Since
${\bf t}_{{\vec h}_{\ell}}$
is increasing, there exists
$t^*=\lim\limits_{{\ell\to \infty}\atop {\ell\in L^{\vec k}}}
{\bf t}_{{\vec h}_{\ell}}$.
Observe that for every $\ell\in L^{\vec k}$ and
$t\in ({\bf t}_{{\vec h}^{\ell}}, {\bf t}_{\vec k})$
we have $\Xi_t\in C^{k_{s_\ell}-1}(\Lambda({{\vec h}_\ell}^-))$.
Then there exists
$\xi^*=\lim\limits_{{\ell\to \infty}\atop {\ell\in L^{\vec k}}}
\xi_{{\vec h}_\ell}$.
If $t^* < {\bf t}_{\vec k}$, we can show as before that necessarily
$\Xi_t=\xi^*$ for $t\in [t^*, {\bf t}_{\vec k})$.

\medskip

Let us summarize. We have shown that at every point
$\{{\bf t}_{\vec h}: {\vec h}\in H\}$ the killed process is
continuous from the right with a limit at the left.
Now we take
$t\in ({\bf t}_{\vec k+1}, {\bf t}_{\vec k})\setminus \{{\bf t}_{\vec
h}: {\vec h}\in H\}$. Assume it is an
accumulation point of $\{{\bf t}_{\vec h}: {\vec h}\in H\}$.

\medskip

If it is not an accumulation point from the right we
put ${\vec h}^*$ the closest element of
$\{ {\bf t}_{\vec h}: {\vec h}\in H\}$
to the right of $t$, also let ${\bf t}_{{\vec h}_n}$ be an
increasing sequence converging to $t$. By the same
arguments as before there exists
$\xi^*=\lim\limits_{n\to \infty} \xi_{{\vec h}_n}$
and we also have
$\Xi_t=\xi^*$ for $t\in [t, {\bf t}_{{\vec h}^*})$.
If it is not an accumulation point from the left we
put ${\vec h}^*$ the closest element of
$\{{\bf t}_{\vec h}: {\vec h}\in H\}$
to the left of $t$. Therefore, by construction, we can assume
that the decreasing sequence ${\bf t}_{{\vec h}_n}$
in $\{ {\bf t}_{\vec h}: {\vec h}\in H\}$
converging to $t$, verifies $\xi_{{\vec h}_n}\in
C^{\ell_n}(\xi_{{\vec h}^*})$, with $\ell_n$ increasing
to $\infty$ as $n$ does. Therefore
$\xi_{{\vec h}^*}=\lim\limits_{n\to \infty}\xi_{{\vec h}_n}$.
Hence $\Xi_t=\xi^*$ for $t\in [{\bf t}_{{\vec h}^*},t]$.

\medskip

Now assume $t$ is an accumulation point from the right
and the left. Let
${\bf t}_{{\vec h}_n}$ be a decreasing sequence and
${\bf t}_{{\vec l}_n}$ be an increasing sequence,
in $\{ {\bf t}_{\vec h}: {\vec h}\in H\}$,
converging to $t$. For $n$ sufficiently large there exists
$m_n$ and $\ell_n$, both converging to $\infty$ as $n$ does,
such that $\xi_{{\vec h}_n}\in C^{\ell_n}(\xi_{{\vec l}_{m_n}})$.
Therefore
$\lim\limits_{n\to \infty}\xi_{{\vec h}_n}=
\lim\limits_{n\to \infty}\xi_{{\vec l}_n}$ and then
$\xi_t$ is this common limit.

\medskip

We have shown our construction fulfills the properties stated
in the Theorem. $\Box$
\end{proof}

\medskip

\begin{remark}
We notice that the set of discontinuities for the process $\Xi$ is given by
$\{{}^n\!\Upsilon :  n\ge 0\}=\{{\bf t}_{\vec k} :{\vec k}\in {\cal M}\}$.
\end{remark}

\begin{theorem}
\label{sim101}
If the measure $\mu$ is atomless then the process $(\Xi_t: t< \Upsilon)$
has no interval of constancy.
\end{theorem}

\medskip

\begin{proof} Using the Markov property it is enough to prove
that for almost all $\xi$ and all $t>0$ we have
$$
\PP_\xi\{\forall_{0< s <t}\; \Xi_s=\xi \}=0.
$$
Since $\mu$ has no atoms we obtain the existence of a strictly
increasing sequence of integers $(n_i)$, such that
$C^0(\xi)\varsupsetneqq C^{n_i}(\xi)\varsupsetneqq
C^{n_{i+1}}(\xi)\downarrow \{\xi\}$. We
consider the random times $^{n_i}\Upsilon$. At these
times the process makes a random selection on
$C^{n_i-1}(\xi)$, then we have to prove
$$
\PP_\xi\{^{n_i}\Upsilon >t\; \hbox{ for all } i\}=0.
$$
We notice that each of these random variables is exponentially
distributed with parameter $1/G_{n_i}\uparrow \infty$
and the result follows.
$\Box$
\end{proof}

\subsection{The Markov Process in the Boundary under reflection at the root}
\label{sp210}
\medskip

The operator ${\underline {\bf
W}}^{-1}=W^{-1}-G_0^{-1}\EE_{\mu}=\sum_{n\ge 1} G_n^{-1}
\big(\EE_{\mu}(\; | {\cal F}_n)-\EE_{\mu}(\; | {\cal
F}_{n-1})\big)$ generates a (conservative) Markov process. Notice
that ${\underline {\bf W}}^{-1}$ has the same form as
$$
W^{-1}=\sum_{n\ge 0} G_n^{-1} \big(\EE_{\mu}(\; | {\cal
F}_n)-\EE_{\mu}(\; | {\cal F}_{n-1})\big)
$$
where in the last expression $G_0\equiv \infty$. Therefore the
analogous of Theorem \ref{nucleo} holds.

\begin{theorem}
\label{nucleo2} The symmetric kernel
\begin{equation}
\label{dens110} p(t,\xi,\eta)= 1-e^{-t/G_{1}(\xi)} +
\sum\limits_{n=1}^{|\xi\land \eta|}
{e^{-t/G_n(\xi)}-e^{-t/G_{n+1}(\xi)} \over \mu(C^n(\xi))},\;
(\xi,\eta)\in \partial_\infty^{reg}\times \partial_\infty^{reg},\;
t>0,
\end{equation}
is Markovian (with total mass $1$) and the Markov semigroup
$P^{{\underline {\bf W}}}_t$ induced in $L^2(\mu)$ verifies
$$
P^{{\underline {\bf W}}}_t  f= \sum\limits_{n\ge 1} e^{-t/G_n}
\Big(\EE_{\mu} (f | {\cal F}_n)- \EE_{\mu} (f | {\cal
F}_{n-1})\Big).
$$
The infinitesimal generator of this semigroup is an extension of
$-{\underline {\bf W}}^{-1}$ defined on $\cal D$.
\end{theorem}

\begin{remark} The formula (\ref{dens110}) shares some similarities
with the formula (3.1) in \cite{albeverio2} (see also (2.9) in
\cite{albeverio3}) developed for random walks on the p-adic field.
Nevertheless, in our case no homogeneity of the tree is needed.
\end{remark}

Let ${\underline \Xi}=({\underline \Xi}_t)$ be the Markov (conservative)
process
associated to the Markov semigroup $P^{{\underline {\bf W}}}_t$.
To simulate the process starting from $\xi$, we first generate a
sequence of independent identically distributed random variables
$(Y_n: n\ge 1)$ with law $exp[1/G_1]$, and we select a sequence of
points $(\xi_n: n\ge 1)$ independent identically distributed in
$\partial_\infty$ with law $\mu$. We define ${}^1\! \Upsilon_0=0$,
$\xi_0=\xi$, ${}^1\! \Upsilon_k=Y_1+...+Y_k$. In each random
interval $[{}^1\! \Upsilon_k, {}^1\! \Upsilon_{k+1})$ we put a
copy of the process ${\bf K}_{{}^1\! \Upsilon_{k+1}-{}^1\!
\Upsilon_k}\,{}^1\Xi^{\xi_k}$, which is the process
${}^1\Xi^{\xi_k}$ conditioned to the fact that the killing time
${}^1\Upsilon$ verifies ${}^1\Upsilon={}^1\! \Upsilon_{k+1}-{}^1\!
\Upsilon_k$. We summarize the main properties of $\underline \Xi$ in
the following result.

\begin{theorem}
\label{sim110} The process $({\underline \Xi}_t: t\ge 0)$ has a
version that is right continuous with left limits. The set of
points of continuity is the complement of $\{{}^n\!\Upsilon_k :
n\ge 1, k\ge 0\}$.
\end{theorem}

\end{document}